\documentclass[leqno, 10pt]{amsart}
\usepackage{amsmath, amsthm, amssymb, latexsym}
 \newtheorem{definition}{Definition}[section]
 \newtheorem{theorem}[definition]{Theorem}
 \newtheorem{lemma}[definition]{Lemma}
 \newtheorem{proposition}[definition]{Proposition}

 \newtheorem*{theorem*}{Theorem}
\newtheorem*{proposition*}{Proposition}

\newtheorem*{lemma*}{Lemma}

 \theoremstyle{remark}
 
 \newtheorem{remark}[definition]{Remark}

  \newtheorem*{acknowledgements}{Acknowledgements}

\newcommand{\op}[1]{\operatorname{#1}}

\newcommand{\acou}[2]{\ensuremath{\langle #1 , #2 \rangle}}

\newcommand{\brak}[1]{\ensuremath{\langle\! #1\!\rangle}}

\newcommand{\tr}{\op{tr}}
\newcommand{\Tra}{\ensuremath{\op{Trace}}}

\newcommand{\TR}{\ensuremath{\op{TR}}}

\newcommand{\Res}{\ensuremath{\op{Res}}}

\newcommand{\C}{\ensuremath{\mathbb{C}}} 
 
\newcommand{\N}{\ensuremath{\mathbb{N}}} 
\newcommand{\R}{\ensuremath{\mathbb{R}}} 
\newcommand{\Z}{\ensuremath{\mathbb{Z}}} 
\newcommand{\CZ}{\ensuremath{\mathbb{C}\!\setminus\!\mathbb{Z}}} 

\newcommand{\Rn}{\ensuremath{\R^{n}}}
\newcommand{\Rtn}{\ensuremath{\R^{2n+1}}}
\newcommand{\URtn}{\ensuremath{U\times \R^{2n+1}}}
\newcommand{\ORn}{\ensuremath{\Omega\times \R^{2n+1}}}
\newcommand{\Rno}{\R^{n}\!\setminus\! 0}
\newcommand{\URn}{U\times\R^{n}}
\newcommand{\URno}{U\times(\R^{n}\!\setminus\! 0)}

\newcommand{\Rd}{\ensuremath{\R^{d+1}}}
\newcommand{\Rdo}{\R^{d+1}\!\setminus\! 0}

\newcommand{\URd}{U\times\R^{d+1}}

\newcommand{\fg}{\ensuremath{\mathfrak{g}}}

\newcommand{\fp}{\ensuremath{\mathfrak{p}}}

\newcommand{\fq}{\ensuremath{\mathfrak{q}}}
\newcommand{\fr}{\ensuremath{\mathfrak{r}}}
\newcommand{\fs}{\ensuremath{\mathfrak{s}}}

\newcommand{\Ca}[1]{\ensuremath{\mathcal{#1}}}

\newcommand{\cE}{\Ca{E}}
\newcommand{\cF}{\ensuremath{\mathcal{F}}}
\newcommand{\cG}{\ensuremath{\mathcal{G}}}

\newcommand{\cI}{\ensuremath{\mathcal{I}}}
\newcommand{\cJ}{\ensuremath{\mathcal{J}}}
\newcommand{\cK}{\ensuremath{\mathcal{K}}}
\newcommand{\cL}{\ensuremath{\mathcal{L}}}

\newcommand{\cN}{\ensuremath{\mathcal{N}}}
\newcommand{\cP}{\ensuremath{\mathcal{P}}}

\newcommand{\cT}{\ensuremath{\mathcal{T}}}

\newcommand{\sD}{\ensuremath{{/\!\!\!\!D}}}%{/{\!\!\!\!D}}

\newcommand{\pdo}{\ensuremath{\Psi}} 
\newcommand{\pdoi}{\ensuremath{\Psi^{\op{int}}}} 
\newcommand{\pdoz}{\ensuremath{\Psi^{\Z}}} 
\newcommand{\pdocz}{\ensuremath{\Psi^{\CZ}}} 
 
\newcommand{\psido}{$\Psi$DO} 
\newcommand{\psidos}{$\Psi$DOs}

\newcommand{\psivdo}{$\Psi_{H}$DO}
\newcommand{\psivdos}{$\Psi_{H}$DOs}
\newcommand{\pvdo}{\ensuremath{\Psi_{H}}} 
\newcommand{\pvdoi}{\ensuremath{\Psi_{H}^{\op{int}}}} 
\newcommand{\pvdoz}{\ensuremath{\Psi_{H}^{\Z}}} 
\newcommand{\pvdocz}{\ensuremath{\Psi_{H}^{\CZ}}} 

\renewcommand{\Box}{\square}

\newcommand{\ord}{{\op{ord}}}

\newcommand{\End}{\ensuremath{\op{End}}}
\newcommand{\END}{\ensuremath{\op{END}}}

\newcommand{\hotimes}{\hat\otimes}

\begin{document}
\title{Logarithmic singularities of Schwartz kernels and local invariants of conformal and CR structures} 

\author{Rapha\"el Ponge}

\address{Department of Mathematics, University of Toronto, Canada.}
\email{ponge@math.toronto.edu}
\thanks{Research partially supported by JSPS fellowship PE06016 (Japan), by NSERC discovery grant 341328-07 (Canada) and 
by a New Staff Matching grant of the Connaught Fund of the University of Toronto (Canada)}
\begin{abstract}
 This paper consists of two parts. In the first part we show that in odd dimension, as well as in even dimension below the critical weight (i.e.~half the 
 dimension), the logarithmic singularities of Schwartz kernels and Green kernels of conformal invariant 
 pseudodifferential operators are linear combinations of Weyl conformal invariants, i.e., of local conformal invariants arising from complete 
 tensorial contractions of the covariant derivatives of the Lorentz ambient metric of Fefferman-Graham. In even dimension and above the critical 
 weight exceptional local conformal invariants may further come into play.  As a consequence, this allows us to get invariant expressions 
 for the logarithmic singularities of the Green kernels of the GJMS operators (including the Yamabe and Paneitz operators). 
 In the second part,  we prove analogues of these results in CR geometry.   Namely, we prove that the logarithmic singularities of  
 Schwartz kernels and Green kernels of CR invariant Heisenberg pseudodifferential 
 operators give rise to local CR invariants, and below the critical weight are linear combinations of complete  tensorial contractions of the covariant 
 derivatives of Fefferman's K\"alher-Lorentz ambient metric. As a consequence, we can obtain invariant descriptions of the logarithmic 
 singularities of the Green kernels of the CR GJMS operators  of Gover-Graham (including the CR Yamabe operator of Jerison-Lee).  
 \end{abstract}

\maketitle 
\numberwithin{equation}{section}

 \maketitle 
    
\section*{Introduction}
Motivated by the analysis of the singularity of the Bergman kernel of a strictly pseudoconvex domain $D\subset \C^{n+1}$ Fefferman~\cite{Fe:PITCA} 
launched the program of determining \emph{all} local invariants of a strictly pseudoconvex CR structure. This program was subsequently 
extended to deal with local invariants of other parabolic geometries, including conformal geometry (see~\cite{FG:CI}). Since Fefferman's seminal 
paper further progress has been made, especially recently (see, e.g., \cite{Al:OCIDO}, \cite{BEG:ITCCRG}, \cite{GH:IAM}, \cite{Hi:CBILDBK}, 
\cite{Hi:AMCCRIDO}). In addition, there is a very recent upsurge of new conformally invariant Riemannian differential operators (see~\cite{Al:OCIDO}, 
\cite{Ju:FCDOQCH}).

In this paper we turn to the analysis of the logarithmic singularities of the Schwartz kernels and Green kernels of \emph{general} invariant 
pseudodifferential operators in conformal and CR geometry. This connects nicely with results of Hirachi~(\cite{Hi:CBILDBK}, \cite{Hi:AMCCRIDO}) 
on the logarithmic singularities of the Bergman and Szeg\"o kernels on boundaries of strictly pseudoconvex domains. 
% In this paper we turn to the analysis of the singularities of the Green kernels of the GJMS operators and of their CR analogues. The GJMS operators 
% have been constructed by Graham-Jenne-Mason-Sparling~\cite{GJMS:CIPLIE} as conformally invariant powers of the Laplacian and so they are higher order 
% generalizations of the Yamabe and Paneitz. Their CR analogues have been constructed by Gover-Graham~\cite{GG:CRIPSL} as 
% CR invariant powers of the sublaplacian. 

The main result in the conformal case (Theorem~\ref{thm:GJMS.main-result}) asserts that in odd dimension, as well as in even dimension below the critical 
weight (i.e.~half of the dimension), the logarithmic singularities of Schwartz kernels and Green kernels of conformally invariant Riemannian 
\psidos\ are linear combinations of Weyl conformal invariants, that is, of local conformal invariants arising complete tensorial contractions of covariant derivatives 
of the ambient 
 Lorentz metric of Fefferman-Graham~(\cite{FG:CI}, \cite{FG:AM}). Above the critical weight the description in even dimension involve the ambiguity 
 independent Weyl conformal invariants recently defined by Graham-Hirachi~\cite{GH:IAM}, as well as the exceptional local conformal invariants of 
 Bailey-Gover~\cite{BG:EIPITCG}. 
 In particular, by specializing this result to the GJMS operators of~\cite{GJMS:CIPLIE}, including 
 the Yamabe and Paneitz operators, we obtain invariant expressions for the logarithmic singularities of the Green kernels of these operators 
 (see Theorem~\ref{thm:GJMS.GJMS}). 

In the CR setting the relevant class of pseudodifferential operators is the class of \psivdos\ introduced by Beals-Greiner~\cite{BG:CHM} and 
Taylor~\cite{Ta:NCMA}. In this context the main result (Theorem~\ref{thm:CR.main-result}) asserts that the logarithmic logarithmic singularities of 
Schwartz kernels and Green kernels of CR invariant \psivdos\ are local CR invariants, and below the critical weight are linear combinations of complete tensorial contractions 
of covariant derivatives of the curvature
of the ambient K\"ahler-Lorentz metric of Fefferman~\cite{Fe:PITCA}. As a consequence this allows us to get invariant expressions for the logarithmic 
singularities of the Green kernels of the CR GJMS operators of~\cite{GG:CRIPSL} (see Theorem~\ref{thm:CR.CRGJMS}). 
% Furthermore, in the 
% special case of the boundary of strictly pseudoconvex domain $\Omega\subset \C^{n+1}$ this implies that the aforementioned logarithmic singularities 
% are local biholomorphic invariants of $\Omega$.

The proof of the main result in the conformal case is divided into three steps. In the first step we show that, given a \psido\ on a Riemannian 
manifold transforming 
conformally under a conformal change of metrics, the logarithmic singularity of its Schwartz kernel, as well as that of its Green kernel when the operator is 
elliptic, transform conformally under a conformal change of metrics (Proposition~\ref{prop:Conformal.main}). This result unifies and extends several previous results of 
Parker-Rosenberg~\cite{PR:ICL}, Gilkey~\cite{Gi:ITHEASIT} and Paycha-Rosenberg~\cite{PR:CACT}.

The second step is a Riemannian invariant version of the first step. Namely, we show that the logarithmic singularities of Schwartz kernels and  
Green kernels of Riemannian invariant \psidos\ are local Riemannian invariants, hence can expressed as linear combinations of complete contractions of 
covariant derivatives of the curvature tensor (see Proposition~\ref{prop:GJMS.Riemannian-invariant-log-sing} for the precise statement). 
This result is very much reminiscent of the Riemannian invariant expression of the coefficients of the heat kernel asymptotics of 
Laplace-type operators (see~\cite{ABP:OHEIT}, \cite{Gi:ITHEASIT}).

In odd dimension, as well as in even dimension below the critical weight, an important result of Bailey-Eastwood-Graham~\cite{BEG:ITCCRG} 
shows that all local conformal invariants as linear combinations of Weyl conformal invariants. Recently, in even dimension the remaining cases have been dealt with 
by Graham-Hirachi~\cite{GH:IAM}. Therefore, in  the final third step, we can simply combine these results with the results of the first two 
steps to deduce our main results in the conformal case.

% tells us that any local conformal invariant is a linear combination of 
% Weyl conformal invariant provided the dimension is odd, or the dimension is even and the weight of the invariant is less than or equal to a critical weight.  
% Therefore, in the third step we can combine the previous two steps with this result to obtain the main result in the conformal case. 

Notice that thanks to the Ricci flatness of the ambient metric they are much fewer Weyl conformal 
invariant than Weyl Riemannian invariants.  Therefore, in the third step we get a more precise information on the forms of the logarithmic 
singularities at stake than the Riemannian invariant expressions provided by the second step.  

Next, the proof of the main result in the CR setting follows a similar pattern.  First, we prove that, given a \psivdo\ on a contact manifold which transforms 
conformally under a conformal change of contact form, the logarithmic singularities of 
its Schwartz kernel and its Green kernel (when the operator is hypoelliptic) transform conformally under a conformal change of contact form (see 
Proposition~\ref{prop:Contact.main}). This extends a previous result of N.K.~Stanton~\cite{St:SICRM}.

In the second step we deal with the logarithmic singularities of pseudohermitian invariant \psivdos\ (these objects are defined in Section~\ref{sec:Pseudohermitian}). 
More precisely, we show that the logarithmic singularities of the 
Schwartz kernels and the Green kernels of pseudohermitian invariant \psivdos\ are local pseudohermitian invariants (see 
Proposition~\ref{prop:CR.log-sing-pseudohermitian-invariants}).  
Therefore these logarithmic singularities appear as universal linear combinations of complete tensorial contractions of covariant derivatives of the 
(pseudohermitian) curvature tensor and of the torsion tensor of the Tanaka-Webster connection.

Similarly to a conformally invariant Riemannian \psido, a CR invariant \psivdo\ is a pseudohermitian invariant \psivdos\ that transforms 
conformally under a conformal change of contact form. Furthermore, we know from Fefferman~\cite{Fe:PITCA} and 
Bailey-Eastwood-Graham~\cite{BEG:ITCCRG} that any local CR invariant of weight less than or equal to the critical weight is linear combination of Weyl CR invariants. 
Combining this with the previous steps allows us to prove the main results in the CR case.

The first and third steps in the CR case are carried along similar lines as that of the corresponding steps in the conformal case. There are some 
technical issues with the second step because we need to introduce definitions of local pseudohermitian invariant and of pseudohermitian invariant \psivdos\ in such way 
that the former is equivalent to the usual definition of a local pseudohermitian invariant and both definitions are suitable for working with the Heisenberg calculus. 
In particular, it is important to take into account the tangent structure of a strictly pseudoconvex CR manifold, in which the Heisenberg group comes 
into play. The bulk of this step then is to prove all the properties of local pseudohermitian invariants and pseudohermitian invariant 
\psivdos\ that are needed in order to prove that the logarithmic singularities of the Schwartz kernels and the Green kernels of the latter do give 
rise to local pseudohermitian invariants. More generally, the arguments used in this step pave the way for proving that various local 
invariants attached to pseudohermitian invariant \psivdos\ (e.g.~local zeta function invariants) give rise to local pseudohermitian invariants. 

Finally, it is believed that by making use of the ambient metric construction of the GJMS operators in~\cite{GJMS:CIPLIE} 
we could compute the logarithmic singularities of these operators in the conformal case, as well as in the CR case. It is conjectured 
that there should be related in a somewhat 
explicit way to the coefficients of the heat kernel asymptotics of the Laplace operator, which has been thoroughly studied 
(see~\cite{Gi:ITHEASIT} and the references therein). We hope to report more on this in a subsequent paper.

The paper is organized as follows. 

In Section~\ref{sec:Log-sing}, we recall how the logarithmic singularity of a \psido\ gives rise to a well defined density. We then  
explain its connection with the noncommutative residue trace of Wodzicki and Guillemin. 

In Section~\ref{sec:Conformal-invariance}, 
we show that the conformal invariance of the logarithmic singularities of the Schwartz kernels and Green kernels of conformally invariant \psidos. 

In Section~\ref{sec:Riemannian} we show that the logarithmic singularities of the Schwartz kernels and Green kernels of
Riemannian invariant \psidos\ are local Riemannian invariants. 

In Section~\ref{sec:GJMS} we prove that the logarithmic singularities of the Schwartz kernels and Green kernels of 
conformally invariant \psidos\ are linear combinations of the local conformal invariants in the sense 
of Fefferman's program. In particular, this leads us to invariant expressions for the logarithmic singularities of the Green kernels of the GJMS operators.

In Section~\ref{sec:Heisenberg}, we recall some important facts about \psivdos\ and their logarithmic singularities. 

In Section~\ref{sec:contact-invariance}, we prove the contact invariance of the Schwartz kernels and Green kernels of contact invariant \psivdos.

In Section~\ref{sec:Pseudohermitian}, we define pseudohermitian invariant \psivdos\ and prove that their Schwartz kernels and Green kernels give rise to local 
pseudohermitian invariants. 

In Section~\ref{sec:CR-GJMS}, we prove that the Schwartz kernels and Green kernels of CR invariant \psivdos\ are linear combinations of Weyl CR 
invariants, which allows us to get invariant expressions for the  the Green kernels of the CR GJMS operators. 

\begin{acknowledgements}
    I am very grateful to Pierre Albin, Charles Fefferman, Robin Graham, Kengo Hirachi and Sidney Webster for helpful and stimulating 
    discussions related to the subject matter of this paper. 
    In addition, I wish to thank for their hospitality the University of Tokyo and the University of California at Berkeley where part of this 
    paper was written. 
\end{acknowledgements}

\section{Pseudodifferential Operators and the Logarithmic Singularities of their Schwartz Kernels}
\label{sec:Log-sing}
In this section we recall some definitions and properties about \psidos\ and the logarithmic singularities of the Schwartz kernels of \psidos.  

First, given an open subset $U\subset \R^{n}$ the symbols on $\URn$ are defined as follows. 

\begin{definition}1)  $S_{m}(\URn)$, $m\in\C$, is the space of functions 
    $p(x,\xi)$ contained in $C^{\infty}(U\times\Rno)$ such that $p(x,t\xi)=t^m p(x,\xi)$ for any $t>0$.\smallskip

2) $S^m(\URn)$,  $m\in\C$, consists of functions  $p\in C^{\infty}(\URn)$ with
an asymptotic expansion $ p \sim \sum_{j\geq 0} p_{m-j}$, $p_{k}\in S_{k}(\URn)$, in the sense that, for any integer $N$, 
any compact $K \subset U$ and any multi-orders $\alpha$, $\beta$, there exists a constant $C_{NK\alpha\beta}>0$ such that, 
for any $x\in K$ and any $\xi \in \Rd$ so that $|\xi | \geq 1$, we have
\begin{equation}
    | \partial^\alpha_{x}\partial^\beta_{\xi}(p-\sum_{j<N}p_{m-j})(x,\xi)| \leq 
    C_{NK\alpha\beta}|\xi|^{\Re m-\brak\beta -N}.
    \label{eq:PsiDOs.asymptotic-expansion-symbols}
\end{equation}
\end{definition}

Given a symbol $p\in S^{m}(\URn)$ we let $p(x,D)$ be the 
continuous linear operator from $C^{\infty}_{c}(U)$ to $C^{\infty}(U)$ such that 
    \begin{equation}
          p(x,D)u(x)= (2\pi)^{-n} \int e^{ix.\xi} p(x,\xi)\hat{u}(\xi)d\xi
    \qquad \forall u\in C^{\infty}_{c}(U).
\end{equation}

Let $M^{n}$ be a manifold and let $\cE$ be a vector bundle over $M$. 
We define \psidos\ on $M$ acting on the sections of $\cE$ as follows. 

\begin{definition}
 $\Psi^{m}(M,\cE)$, $m\in \C$, consists of continuous operators  $P$ from 
$C^{\infty}_{c}(M,\cE)$ to $C^{\infty}(M,\cE)$ such that:\smallskip

(i) The Schwartz kernel of $P$ is smooth off the diagonal;\smallskip 

(ii) In any trivializing local coordinates the operator $P$ can be written as 
\begin{equation}
    P=p(x,D)+R,
%     \label
\end{equation}
where $p$ is a symbol of order $m$ and $R$ is a smoothing operator. 
\end{definition}

We can give a precise description of the singularity of the Schwartz kernel of a \psido\ near the diagonal and, in fact, the general form of these 
singularities can be used to characterize \psidos~(see, e.g., \cite{Ho:ALPDO3}, \cite{Me:APSIT}, 
\cite{BG:CHM}). In particular, if $P:C^{\infty}_{c}(M,\cE)\rightarrow C^{\infty}(M,\cE)$ if a \psido\ of integer order $m\geq -n$,
then in local coordinates its Schwartz kernel $k_{P}(x,y)$ has a behavior near the 
diagonal $y=x$ of the form
   \begin{equation}
        k_{P}(x,y)=\sum_{-(m+n)\leq j\leq -1}a_{j}(x,x-y) - c_{P}(x)\log |x-y| + 
        \op{O}(1), 
         \label{eq:Log-sing.log-sing}
    \end{equation}
    where $a_{j}(x,y)\in C^{\infty}(U\times (\Rno))$ is homogeneous of degree $j$ in $y$ and we have
%     $c_{P}(x)$ is given by the formula: 
 \begin{equation}
    c_{P}(x)=\frac{1}{(2\pi)^{n}}\int_{S^{n-1}}p_{-n}(x,\xi)d^{n-1}\xi,
     \label{eq:Log-sing.formula-cP}
\end{equation}
where $p_{-n}(x,\xi)$ is the symbol of degree $-n$ of $P$. 

It seems to have been first observed by Connes-Moscovici~\cite{CM:LIFNCG} (see~\cite{GVF:ENCG}, \cite{Po:NCR}) for detailed proofs) that the coefficient 
$c_{P}(x)$ makes sense globally on $M$ as a 1-density with values in $\End \cE$, i.e., it defines an element of $C^{\infty}(M,|\Lambda|(M)\otimes 
\End \cE)$ where $|\Lambda|(M)$ is the bundle of 1-densities on $M$. 

In the sequel we refer to the density $c_{P}(x)$ as the \emph{logarithmic singularity} of the Schwartz kernel of $P$. 

If $P$ is elliptic, then we shall call \emph{Green kernel for $P$} the Schwartz kernel of a parametrix $Q\in \Psi^{-m}(M,\cE)$ for $P$. 
Such a parametrix is uniquely defined only modulo smoothing operators, but the singularity near the diagonal 
of the Schwartz kernel of $Q$, including the logarithmic singularity 
$c_{Q}(x)$, does not depend on the choice of $Q$. 

\begin{definition}
    If $P\in \Psi^{m}(M,\cE)$, $m \in \Z$, is elliptic, then the Green kernel logarithmic singularity of $P$ is the density
    \begin{equation}
        \gamma_{P}(x):=c_{Q}(x),
%         \label
    \end{equation}
    where $Q\in \Psi^{-m}(M,\cE)$ is any given parametrix for $P$. 
\end{definition}

Next, because of~(\ref{eq:Log-sing.formula-cP}) the density $c_{P}(x)$ is 
related to the noncommutative residue trace of Wodzicki~(\cite{Wo:LISA}, \cite{Wo:NCRF}) and Guillemin~\cite{Gu:NPWF} as follows. 

Let $\pdoi(M,\cE) = \cup_{\Re m < -n}\pdo^{m}(M,\cE)$ denote the class of \psidos\ whose symbols are integrable with respect to the $\xi$-variable. If 
$P$ is a \psido\ in this class then the restriction to the diagonal of its Schwartz kernel $k_{P}(x,y)$ defines a smooth $\End \cE$-valued density $k_{P}(x,x)$. 
Therefore, if $M$ is compact then $P$ is trace-class on $L^{2}(M,\cE)$ and we have
\begin{equation}
    \Tra P= \int_{M}k_{P}(x,x).
    \label{eq:Log-sing.trace}
 \end{equation}

In fact, the map $P\rightarrow k_{P}(x,x)$ admits an analytic continuation $P\rightarrow t_{P}(x)$ to the class  $\pdocz(M,\cE)$ of non-integer 
\psidos, where analyticity is meant with respect to holomorphic families of \psidos\ as in~\cite{Gu:GLD} and~\cite{KV:GDEO}. 
Furthermore, if $P \in \pdoz(M,\cE)$ and if $(P(z))_{z\in\C}$ is a holomorphic family of \psidos\ such that $\ord P(z)=\ord P +z$ and $P(0)=P$.
Then the map  $z\rightarrow t_{P(z)}(x)$ has at worst a simple pole 
 singularity at  $z=0$ in such way that
 \begin{equation}
     \Res_{z=0} t_{P(z)}(x)=- c_{P}(x).
     \label{eq:Log-sing.residue-tP(z)}
 \end{equation}

Suppose now that $M$ is compact. Then the \emph{noncommutative residue} is the linear functional on $\pdoz(M,\cE)$ defined by

\begin{equation}
    \Res P := \int_{M}\tr_{\cE}c_{P}(x) \qquad \forall P\in \pdocz(M,\cE).
\end{equation}
Thanks to~(\ref{eq:Log-sing.formula-cP}) this definition agrees with the usual definition of the noncommutative residue. Moreover, 
by using~(\ref{eq:Log-sing.residue-tP(z)}) we see that if $(P(z))_{z\in\C}$ is a holomorphic family of \psidos\ such that 
$\ord P(z)=\ord P +z$ and $P(0)=P$, then the map $z \rightarrow \Tra P(z)$ has an analytic extension to $\CZ$ with at worst a simple pole near $z=0$ 
in such way that
\begin{equation}
    \Res P=-\Res_{z=0} \TR P(z).
\end{equation}
 Using this it is not difficult to see that the noncommutative residue is a trace on $\pdoz(M,\cE)$. Wodzicki~\cite{Wo:PhD} even proved that his is 
 the unique trace up to constant multiple when $M$ is connected. 

Finally, let $P:C^{\infty}(M,\cE) \rightarrow C^{\infty}(M,\cE)$ be a \psido\ of integer order $m\geq 0$ with a positive principal symbol. For $t>0$ 
we let $k_{t}(x,y)$ denote the Schwartz kernel of $e^{-tP}$. Then $k_{t}(x,y)$ is a smooth kernel and  as $t\rightarrow 0^{+}$ we have
\begin{equation}
    k_{t}(x,x)\sim t^{-\frac{n}{m}}\sum_{j \geq 0} t^{\frac{j}{m}}a_{j}(P)(x)+\log t\sum_{j\geq 0}t^{j}b_{j}(P)(x), 
     \label{eq:Log-sing.heat-kernel-asymptotics}
\end{equation}
where we further have $a_{2j+1}(P)(x)=b_{j}(P)(x)=0$ for any $j=0,1,\ldots$ when $P$ is a differential operator 
(see, e.g.,~\cite{Gi:ITHEASIT}, \cite{Gr:AEHE}). 

By making use of the Mellin Formula we can explicitly relate the coefficients of the above heat kernel asymptotics to the singularities of the local 
zeta function $t_{P^{-s}}(x)$ (see, e.g.,~\cite[3.23]{Wo:NCRF}). In particular, if for $j=0,\ldots, n-1$ we set $\sigma_{j}=\frac{n-j}{m}$ then we have
\begin{equation}
mc_{P^{-\sigma_{j}}}(x)=\Gamma(\sigma_{j})^{-1}a_{j}(P)(x).
     \label{eq:Log-sing.Log-sing-heat-kernel}
\end{equation}
% (Strictly speaking when $P$ has negative eigenvalues we have to make a choice of continuous determination of the argument in order to define the complex powers 
% $P^{-s}$, $s\in \C$, but the quantities involved in Eq.~(\ref{eq:Log-sing.Log-sing-heat-kernel}) are insensitive to such a choice.)

The above equalities provide us with an immediate connection between the Green kernel logarithmic singularity of $P$ 
and the heat kernel asymptotics~(\ref{eq:Log-sing.heat-kernel-asymptotics}). Indeed, as the partial inverse $P^{-1}$ is a parametrix for $P$ in 
$\Psi^{-m}(M,\cE)$,  setting $j=n-m$ in~(\ref{eq:Log-sing.Log-sing-heat-kernel}) gives
\begin{equation}
    a_{n-m}(P)(x)= mc_{P^{-1}}(x)= m\gamma_{P}(x) .
     \label{eq:Log-sing.Green-heat}
\end{equation}

\section{Conformal Invariance of Logarithmic Singularities of $\Psi$DOs}
\label{sec:Conformal-invariance} 
In this section we will prove that the logarithmic singularities of conformally invariant \psidos\ on a given Riemannian manifold $(M^{n},g)$ 
transform conformally under conformal changes of metric. 
% This will provide us with a generalization of the result of~\cite{PR:ICL} about the conformal invariance of the log singularity of the Green kernel 
% of the conformal Laplacian and will aloo

Throughout this section we let $(M^{n},g)$ be a Riemannian manifold. 
The first historic instances of conformally invariant operator were the 
Dirac and Yamabe operators. 

If $M$ is spin and we let $\sD_{g}$ denote the Dirac operator of $M$ acting on spinors then Hitchin~\cite{Hi:HS} and 
Kosmann-Schwarbach~\cite{Ko:DELOD} proved that under a conformal change of metric $g\rightarrow e^{2f}g$, $f \in C^{\infty}(M,\R)$, we have
\begin{equation}
    \sD_{e^{2f}g}=e^{-\frac{n+1}{2}f}\sD_{g}e^{\frac{n-1}{2}f}.
%     \label
\end{equation}

The Yamabe operator $\Box_{g}:C^{\infty}(M)\rightarrow 
C^{\infty}(M)$ is a perturbation of the  Laplace operator $\Delta_{g}$ in order to get a conformally invariant operator. It is given by 
\begin{gather}
       \Box_{g}=\Delta_{g}+\frac{n-2}{4(n-1)}\kappa_{g},
%     \label
\end{gather}
where $\kappa_{g}$ is the scalar curvature of $M$, and it satisfies
\begin{equation}
     \Box_{e^{2f}g}=e^{-(\frac{n}{2}+1)f} \Box_{g}e^{(\frac{n}{2}-1)f} \qquad \forall f \in C^{\infty}(M,\R).
%    \label
\end{equation}

This construction was generalized by Graham-Jenne-Mason-Sparling~\cite{GJMS:CIPLIE} (see also~\cite{GZ:SMCG})
who produced, for any integer $k\in \N$ when $n$ is odd, and for $k=1,\ldots,\frac{n}{2}$ when $n$ is even, a conformal $k$-th power of $\Delta_{g}$, i.e., 
a selfadjoint differential operator $\Box_{g}^{(k)}:C^{\infty}(M)\rightarrow C^{\infty}(M)$ such that
\begin{gather}
    \Box_{g}^{(k)}=\Delta_{g}^{(k)} + \ \text{lower order terms},
    \label{eq:Conformal.GJMS1}\\ 
          \Box_{e^{2f}g}^{(k)}=e^{-(\frac{n}{2}+k)f} \Box_{g}e^{(\frac{n}{2}-k)f} \qquad \forall f \in C^{\infty}(M,\R).
     \label{eq:Conformal.GJMS2}
\end{gather}
In particular, for $k=1$ we recover the Yamabe operator and for $k=2$ we recover the fourth order operator
of Paneitz~(\cite{Pa:QCCDOAPRM}, \cite{ES:CIMG}). 

There are further generalizations of the GJMS operators. Branson-Gover~\cite{BG:CINLO}
and Peterson~\cite{Pe:CCPDO} constructed families of conformally 
invariant \psidos\ which include the GJMS operators. In this case we have conformal invariance only up to smoothing operators. Recently, 
Alexakis~(\cite{Al:OCIDO}, \cite{Al:OCIDOOD}) and Juhl~\cite{Ju:FCDOQCH} constructed new families of conformally invariant operators. 
Furthermore, Alexakis proved that, under some restrictions, his family of operators exhausts \emph{all} conformally invariant \emph{Riemannian} differential operators. 

In the sequel we let $\cE$ denote a vector bundle over $M$ and we let $\cG$ be the class of Riemannian metrics on $M$ that are conformal multiples of 
$g$. 

Let $(P_{\hat{g}})_{\hat{g}\in\cG}\subset \Psi^{m}(M,\cE)$ be a family of  \psidos\  of integer order $m$ so that  there 
are real numbers $w$ and $w'$ in such way that, for any $f$ in $C^{\infty}(M,\R)$, we have 
\begin{equation}
    P_{e^{f}g}=e^{w'f} P_{g}e^{-wf} \quad \bmod \Psi^{-\infty}(M,\cE).
    \label{eq:Conformal.covariancePg}
\end{equation}

\begin{proposition}\label{prop:Conformal.main} 
1) If $m\geq -n$, then  
\begin{equation}
    c_{P_{e^{f}g}}(x)=e^{-(w-w')f(x)}c_{P_{g}}(x) \qquad \forall f\in C^{\infty}(M,\R).
%      \label{eq:CILS.conformla-invariance-Pg}
\end{equation}

2) Assume that $P_{g}$ is elliptic and we have $0\leq m\leq n$, then 
\begin{equation}
    \gamma_{P_{e^{f}g}}(x)=e^{-(w'-w)f(x)}\gamma_{P_{g}}(x) \qquad \forall f\in C^{\infty}(M,\R).
%     \label
\end{equation}
\end{proposition}
\begin{proof}
 Let $f \in C^{\infty}(M,\R)$, set $\hat{g}=e^{f}g$ and let $k_{P_{g}}(x,y)$ and $k_{P_{\hat{g}}}(x,y)$ denote the respective Schwartz kernels of 
 $P_{g}$ and $P_{\hat{g}}$. It follows from~(\ref{eq:Conformal.covariancePg}) that 
 near the diagonal $y=x$ we have
 \begin{equation}
     k_{P_{\hat{g}}}(x,y)=e^{w'f(x)} k_{P_{g}}(x,y) e^{-wf(y)}+\op{O}(1).
      \label{eq:CILS.kPg-kPhg}
 \end{equation}

Let $U\subset \Rn$ be an open of local coordinates. By~(\ref{eq:Log-sing.log-sing}) the kernel 
$k_{P_{g}}(x,y)$ has a behavior near the diagonal of the form
\begin{equation}
    k_{P_{g}}(x,y)=\sum_{-(m+n) \leq j \leq -1}a_{j}(x,y)-c_{P_{g}}(x)\log|x-y| +\op{O}(1),
%     \label
\end{equation}
where $a_{j}(x,y)\in C^{\infty}(\URno)$ is homogeneous of degree $j$ with respect to $y$. Combining this with~(\ref{eq:CILS.kPg-kPhg}) then gives
\begin{equation}
    k_{P_{\hat{g}}}(x,y)=\sum_{-(m+n) \leq j \leq -1}b(x,y)a_{j}(x,y) -c_{P_{g}}(x)b(x,y)\log|x-y| +\op{O}(1),
     \label{eq:CILS.sing-kPhg}
\end{equation}
where we have set $b(x,y)=e^{-wf(y)+w'f(x)}$. 

The Taylor expansion of $b(x,y)$ near $y=x$ is of the form
\begin{equation}
    b(x,y)= \sum_{|\alpha|< m}(y-x)^{\alpha}b_{\alpha}(x) + \sum_{|\alpha|=m} (x-y)^{\alpha}r_{\alpha}(x,y),
    \label{eq:CILS.Taylor-b}
\end{equation}
where we have set $b_{\alpha}(x)=\frac{1}{\alpha!}\partial_{y}^{\alpha}b(x,x)$, and the functions $r_{\alpha}(x,y)$ are smooth near $y=x$. Using this 
we obtain
\begin{equation}
    b(x,y)a_{j}(x,y)=\sum_{|\alpha|+j \leq -1}b_{\alpha}(x)(y-x)^{\alpha}a_{j}(x,y) +\op{O}(1),
    \label{eq:CILS.sing-kPhg.homogeneous-terms}
\end{equation}
where each term $b_{\alpha}(x)(y-x)^{\alpha}a_{j}(x,y)$ is homogeneous in $y$ of degree $|\alpha|+j\leq -1$. 

Moreover, as we have 
$(x-y)^{\alpha}\log|x-y|=\op{O}(1)$  for any multi-order $\alpha\neq0$, from~(\ref{eq:CILS.Taylor-b}) we also get 
\begin{equation}
    b(x,y)\log|x-y|= b(x,x)\log|x-y|+\op{O}(1)=e^{-(w-w')f(x)}\log|x-y|+\op{O}(1).
     \label{eq:CILS.sing-kPhg.log-term}
\end{equation}

Combining~(\ref{eq:CILS.sing-kPhg}) with~(\ref{eq:CILS.sing-kPhg.homogeneous-terms}) and~(\ref{eq:CILS.sing-kPhg.log-term}) 
shows that $k_{P_{\hat{g}}}(x,y)$ has a behavior near the diagonal  of the form 
\begin{multline}
     k_{P_{\hat{g}}}(x,y)= \\ \sum_{-(m+n) \leq |\alpha|+j \leq -1}b_{\alpha}(x)(y-x)^{\alpha}a_{j}(x,y)  -c_{P_{g}}(x)e^{-(w-w')f(x)}\log|x-y|
     +\op{O}(1).
%     \label
\end{multline}
Comparing this to~(\ref{eq:Log-sing.log-sing}) yields the equality $c_{P_{\hat{g}}}(x)=e^{-(w-w')f(x)}c_{P_{g}}(x)$. 

Now, assume that $P_{g}$ is elliptic and we have $m\leq n$. Let $Q_{g}$ (resp.~$Q_{\hat{g}}$) be a parametrix in $\Psi^{-m}(M,\cE)$ 
for $P_{g}$ (resp.~$P_{\hat{g}}$). Thanks to~(\ref{eq:Conformal.covariancePg}) we have
\begin{equation}
  P_{\hat{g}}e^{wf}Q_{g} e^{-w'f}= e^{w'f}P_{g}Q_{g}e^{-w'f}=1 \qquad \bmod \Psi^{-\infty}(M,\cE).
 \end{equation}
Multiplying the right-hand and left-hand sides by $Q_{\hat{g}}$ gives
\begin{equation}
    Q_{\hat{g}}= Q_{\hat{g}}P_{\hat{g}}e^{wf}Q_{g} e^{-w'f}= e^{wf}Q_{g} e^{-w'f} \qquad \bmod \Psi^{-\infty}(M,\cE).
     \label{eq:CILS.conformal-invariance-parametrix}
\end{equation}
We then can apply the first part of the proof to get $c_{Q_{\hat{g}}}(x)=e^{-(w'-w)f(x)}c_{P_{g}}(x)$. The proof is now complete.
\end{proof}

The above result unifies and extend several previous results of conformal invariance of densities associated to conformally 
invariant operators. 

First, in~\cite{PR:ICL} Parker-Rosenberg proved the conformal invariance on a compact manifold of the Green kernel of the Yamabe operator $\Box_{g}$ 
(i.e.~the Schwartz kernel of $\Box_{g}^{-1}$). In this setting the singularity near the diagonal of the Green kernel is derived from the knowledge of 
the off-diagonal small time asymptotics for the heat kernel of $\Box_{g}$. Moreover, the logarithmic singularity is described in the form 
$-c(x)\log d(x,y)$, where $d(x,y)$ is the Riemannian distance. Since in local coordinates $\log \frac{d(x,y)}{|x-y|}$ is bounded near $y=x$ this 
description of the logarithmic singularity is the same as that provided by~(\ref{eq:Log-sing.log-sing}). 
Therefore, we see that Proposition~\ref{prop:Conformal.main} 
allows us to recover 
Parker-Rosenberg's result. 

In fact, in~\cite{PR:ICL} the coefficient $c(x)$ in the logarithmic singularity $-c(x)\log d(x,y)$ was identified with the coefficient 
$a_{n-2}(\Box_{g})(x)$ of $t^{-1}$ in the heat kernel asymptotics~(\ref{eq:Log-sing.heat-kernel-asymptotics}) for $\Box_{g}$. 
This allowed Parker-Rosenberg to prove the conformal 
invariance of $a_{n-2}(\Box_{g})(x)$. Subsequently, Gilkey~\cite[Thm.~1.9.4]{Gi:ITHEASIT} proved the conformal 
invariance of the coefficient $a_{n-m}(P_{g})(x)$ of $t^{-1}$ in the heat kernel asymptotics for a conformally invariant selfadjoint elliptic differential operator 
$P_{g}$ of order $m$ with positive principal symbol on a compact Riemannian manifold. Thanks to~(\ref{eq:Log-sing.Green-heat}) 
we have $a_{n-m}(P_{g})(x)=m\gamma_{P_{g}}(x)$, 
so Proposition~\ref{prop:Conformal.main} also allows us to recover Gilkey's result. 

Recently Paycha-Rosenberg~\cite{PR:CACT} extended Gilkey's result to \psidos\ and proved the conformal invariance of noncommutative residue densities 
of conformally 
invariant \psidos. The arguments were based on variational formulas for  zeta functions of elliptic \psidos, so the result was stated for compact 
manifold and  for an elliptic conformally invariant \psidos\ such that there is a spectral cut independent of the metric for both the operator and its 
principal symbol. 

Since the density $c_{P}(x)$ agrees with the noncommutative residue density of $P$ it follows that the results of Paycha-Rosenberg 
are encapsulated by Proposition~\ref{prop:Conformal.main} and hold in full generality on noncompact manifold and for non-elliptic conformally invariant \psidos. 

Notice that it is important to be able to remove the ellipticity assumptions from the results of Gilkey and Paycha-Rosenberg, because we can construct examples 
of \emph{non-elliptic} conformally invariant \psidos. For instance, let $Q^{(k)}_{g}\in \Psi^{-2k}(M)$ be a parametrix for 
the GJMS operator $\boxdot^{(k)}_{g}$. Then by~(\ref{eq:Conformal.GJMS2}) and~(\ref{eq:CILS.conformal-invariance-parametrix}) we have 
\begin{equation}
    Q_{e^{2f}g}^{(k)}=e^{-(\frac{n}{2}-k)f}Q_{g}^{(k)}e^{(\frac{n}{2}+k)f} \qquad \forall f \in C^{\infty}(M,\R).
%     \label
\end{equation}

Let $L_{g}:C^{\infty}(M)\rightarrow C^{\infty}(M)$ be a Weyl differential operator as constructed by Alexakis~(\cite{Al:OCIDO}, \cite{Al:OCIDOOD}) 
such that, for some $w'\in \Z$ we have $L_{e^{2f}}=e^{-2w'f}L_{g}e^{(\frac{n}{2}-k)f}$. Alexakis' construction shows that there is a handful of such operators. 
In addition, these operators need not be elliptic. Then the operator $L_{g}Q_{g}^{(k)}$ satisfies
\begin{equation}
    L_{e^{2f}g}Q_{e^{2f}g}^{(k)}= e^{-2w'f}L_{g}Q_{g}^{(k)}e^{(\frac{n}{2}+k)f} \qquad \forall f \in C^{\infty}(M,\R). 
%     \label
\end{equation}
Furthermore, if we choose $L_{g}$ to be non-elliptic, then $L_{g}Q_{g}^{(k)}$ is not elliptic and we really need to use Proposition~\ref{prop:Conformal.main} to 
prove that 
\begin{equation}
    c_{L_{e^{f}g}Q_{e^{f}g}^{(k)}}(x)= e^{(\frac{n}{2}+k-2w')f} c_{L_{g}Q_{g}^{(k)}}(x)\qquad \forall f \in C^{\infty}(M,\R). 
     \label{eq:CILS.conformal-invariance-cLgQgk}
\end{equation}

\section{Logarithmic Singularities of Riemannian Invariant \psidos}
\label{sec:Riemannian}
In this section we shall prove that the logarithmic singularities of Riemannian invariant \psidos\ are local Riemannian invariants. 

Let $M_{n}(\R)_{+}$ denote the open subset of $M_{n}(\R)$ consisting of positive definite matrices. 
Following~\cite{ABP:OHEIT} we call \emph{scalar local Riemannian invariant} of weight $w$, $w\in \Z$, 
datum on any Riemannian manifold $(M^{n},g)$ of a function $\cI_{g}\in C^{\infty}(M)$ such that:\smallskip 

- There exist finitely many functions $a_{\alpha\beta}\in C^{\infty}(M_{n}(\R)_{+})$ such that in \emph{any} local coordinates we can write 
$\cI_{g}(x)=\sum a_{\alpha\beta}(g(x))(\partial^{\alpha}g(x))^{\beta}$.\smallskip

- We have $\cI_{t g}(x)= t^{-w}\cI_{g}(x)$ for any $t>0$.\smallskip

It follows from the invariant theory developed by Atiyah-Bott-Patodi~\cite{ABP:OHEIT} (see also~\cite{Gi:ITHEASIT}) that any local Riemannian invariant is a 
linear combination of complete contractions of the covariant derivatives of the curvature tensor. 

Notice also that the above definition continue to make sense for manifolds equipped with a nondegenerate metric of nonpositive signature, provided we replace 
$M_{n}(\R)_{+}$ by the subset of nondegenerate selfadjoint matrix of the corresponding signature. Following the convention of~\cite{FG:CI} we shall continue to
call local Riemannian invariants such invariants. 

Let $R_{ijkl}=\acou{R(\partial_{i},\partial_{j})\partial_{k}}{\partial_{l}}$ denote the components of the curvature tensor of $(M,g)$. We will use the metric 
$g=(g_{ij})$ and its inverse $g^{-1}=(g^{ij})$ to lower and raise indices. For instance, the Ricci tensor is 
$\rho_{jk}:=R^{\mbox{~}\mbox{~}\mbox{~}i}_{ijk}=g^{il}R_{ijkl}$ 
and the scalar curvature is $\kappa_{g}:=\rho^{~j}_{j}=g^{ji}\rho_{ij}$. 

All the scalar local Riemannian invariants of weight 1 are constant multiples of $\kappa_{g}$, and 
those of weight $2$ are linear combinations of the following invariants:
\begin{equation}
    |R|^{2}_{g}:=R^{ijkl}R_{ijkl}, \qquad |\rho|_{g}:=\rho^{ij}\rho_{jk}, \qquad |\kappa_{g}|^{2}_{g}, \qquad \Delta_{g}\kappa_{g}.
     \label{eq:GJMS.weight2-Riemannian-invariants}
\end{equation}

Next, for $m \in \C$ we let $S_{m}(M_{n}(\R)_{+}\times\Rn)$ denote the space of functions $a(g,\xi)$ in $C^{\infty}(M_{n}(\R)_{+}\times 
    (\Rno))$ such that we have $a(g,t\xi)=t^{m}a(g,\xi)$  for any $t>0$. 

\begin{definition}\label{def:GJMS.Riemannian-invariant-PsiDO}
    A Riemannian invariant \psido\ of order $m$ and weight $w$ is the datum on any Riemannian manifold $(M^{n},g)$ of an operator $P_{g}\in 
    \Psi^{m}(M)$ so that:\smallskip
    
    (i) For $j=0,1,\ldots$ there exist finitely  many symbols $a_{j\alpha\beta}\in S_{m-j}(M_{n}(\R)_{+}\times\Rn)$ such that in any local 
    coordinates $P_{g}$ has symbol $p_{g}(x,\xi)\sim \sum_{j\geq 0} p_{g,m-j}(x,\xi)$, where 
   \begin{equation}
        p_{g,m-j}(x,\xi)=\sum_{\alpha,\beta}(\partial^{\alpha}g(x))^{\beta}a_{j\alpha\beta}(g(x),\xi);
     \label{eq:Conformal.Riemannian-invariant-symbol}
   \end{equation}
    
    (ii) For any $t>0$ we have $P_{tg}=t^{-w}P_{g}$ modulo $\Psi^{-\infty}(M)$.\smallskip 

\noindent In addition, we say that $P$ is admissible if in~(\ref{eq:Conformal.Riemannian-invariant-symbol}) we can take $a_{0\alpha\beta}$ to be zero 
for $(\alpha,\beta)\neq 0$.
\end{definition}

\begin{remark}
    In~(ii) we require to have $P_{tg}=t^{-w}P_{g}$ modulo smoothing operators, rather than to have an actual 
    equality, so that if we replace $P_{g}$ by a properly supported \psido\ that agrees with $P_{g}$ modulo a smoothing operator, then we get a 
    Riemannian invariant \psido\ with same symbol. This way we can compose Riemannian invariant \psidos. This is totally innocuous when we consider 
    differential operators, because two differential operators that differ by a smoothing operator agree.
\end{remark}

\begin{proposition}\label{prop:GJMS.Riemannian-invariance-properties1}
  Let $P_{g}$ be a Riemannian invariant \psido\ of order $m$ and weight $w$, let $Q_{g}$ be a Riemannian \psido\ of order $m'$ and weight $w'$, and 
  suppose that $P_{g}$ or $Q_{g}$ is properly supported. Then $P_{g}Q_{g}$ is a Riemannian invariant 
  \psido\ of order $m+m'$ and weight $w+w'$.
\end{proposition}
\begin{proof}
  First, the operator $P_{g}Q_{g}$ is a \psido\ of order $m+m'$ and for 
   any $t>0$ we have $P_{tg}Q_{tg}=t^{-(w+w')}P_{g}Q_{g}$ modulo $\Psi^{-\infty}(M)$. 

   Next, let $p_{g}(x,\xi)\sim \sum p_{g,m-j}(x,\xi)$ and  let $q_{g}(x,\xi)\sim \sum q_{g,m'-j}(x,\xi)$ be the respective symbols of $P_{g}$ and 
   $Q_{g}$ in local 
   coordinates. Then it is well-known (see, e,g.,~\cite{Ho:ALPDO3}) that the symbol $r_{g}(x,\xi)\sim \sum r_{g,m'-j}(x,\xi)$ of $P_{g}Q_{g}$ is such that 
   we have $r_{g}(x,\xi) \sim \sum \frac{1}{\alpha!}\partial_{\xi}^{\alpha}p_{g}(x,\xi)D_{x}^{\alpha}q_{g}(x,\xi)$. Thus,
   \begin{equation}
       r_{g, m+m'-j}(x,\xi)= \sum_{|\alpha|+k+l=j} \frac{1}{\alpha!}\partial_{\xi}^{\alpha}p_{g,m-k}(x,\xi)D_{x}^{\alpha}q_{g,m'-l}(x,\xi).
       \label{eq:GJMS.product-symbols}
   \end{equation}
   By assumption $p_{g}(x,\xi)$ and $q_{g}(x,\xi)$ satisfy the condition (i) of Definition~\ref{def:GJMS.Riemannian-invariant-PsiDO}. Therefore, 
   using~(\ref{eq:GJMS.product-symbols}) it is not difficult to check that so does $r_{g}(x,\xi)$. Hence $P_{g}Q_{g}$ is 
   a Riemannian invariant \psido\ of weight $w+w'$.
\end{proof}

\begin{proposition}\label{prop:GJMS.Riemannian-invariance-properties2}
  Let $P_{g}$ be Riemannian invariant \psido\ of order $m$ and weight $w$ which is elliptic and is admissible in the sense of 
  Definition~\ref{def:GJMS.Riemannian-invariant-PsiDO}.  For 
  each Riemannian manifold $(M^{n},g)$ let $Q_{g}\in \Psi^{-m}(M,\cE)$ be a parametrix for $P$. Then $Q_{g}$ is a Riemannian invariant \psido\ of 
  weight $-w$. 
\end{proposition}
\begin{proof}
  First, without any loss of generality we may assume $Q_{g}$ to be properly supported. Let $t>0$. As $P_{tg}=t^{-w}P_{g}$ modulo $\Psi^{-\infty}(M)$ 
  we see that $t^{w}Q_{g}$ is a parametrix for $P_{tg}$, hence it agrees with
  $Q_{tg}$ modulo $\Psi^{-\infty}(M)$. 
  
 Next, since $P_{g}$ is admissible there exists $a_{m}\in S_{m}(M_{n}(\Rn)_{+}\times \Rn)$ such that in any given local coordinates the principal symbol of 
  $P_{g}$ is $p_{m}(x,\xi)=a_{m}(g(x),\xi)$. The fact that $P_{g}$ is elliptic then implies that, for any Riemannian manifold $(M^{n},g)$ and for 
  $x$ in the range of the given local coordinates, we have $a_{m}(g(x),\xi)\neq 0$ for any $\xi\neq 0$. Since any matrix $g \in M_{n}(\Rn)_{+}$ defines a 
  Riemannian metric on $\Rn$, we see that $a_{m}(g,\xi) $ is an invertible symbol in $S_{m}(M_{n}(\Rn)_{+}\times \Rn)$. 
  
  Now, let $p\sim \sum p_{g,m-j}(x,\xi)$ and $q(x,\xi) \sim \sum_{j\geq 0}q_{-m-j}(x,\xi)$ be the respective symbols of $P_{g}$ and  
  $Q_{g}$ in local coordinates. As we have $Q_{g}P_{g}=1$ modulo 
  $\Psi^{-\infty}(M)$, using~(\ref{eq:GJMS.product-symbols}) we get
  \begin{equation}
      q_{-m}p_{g,m}=1, \qquad \sum_{|\alpha|+k+l=j} \frac{1}{\alpha!}\partial_{\xi}^{\alpha}q_{-m-k}D_{x}^{\alpha}p_{g,m-l}=0 \quad 
      j\geq 1.
      \label{eq:GJMS.symbol-parametrix1}
  \end{equation}
  Therefore, we obtain
  \begin{gather}
      q_{-m}(x,\xi)=p_{g,m}(x,\xi)^{-1}=a_{m}(g(x),\xi)^{-1},\\
%       \label  
   q_{-m-j}(x,\xi)= a_{m}(g(x),\xi)^{-1}  \!\!\!\!
    \sum_{\substack{|\alpha|+k+l=j\\ 
    k<j}}\!\!\!\!\frac{1}{\alpha!}\partial_{\xi}^{\alpha}q_{-m-k}(x,\xi)D_{x}^{\alpha}p_{m-l}(x,\xi)  \quad j\geq 1.
       \label{eq:GJMS.symbol-parametrix3}
  \end{gather}
  By induction we then can show that for $j=0,1,\ldots$ the symbol $q_{-m-j}(x,\xi)$ can be expressed as a universal expression of the 
  form~(\ref{eq:Conformal.Riemannian-invariant-symbol}).  This completes the proof that $Q_{g}$ is a Riemannian invariant \psido\ of weight $-w$.
 \end{proof}

 In the sequel for any top-degree form $\eta$ on $M$ we let $|\eta|$ denote the corresponding 1-density (or measure) defined by $\eta$. For instance, if 
 $v_{g}(x):=\sqrt{g(x)}dx^{1}\wedge \ldots \wedge dx^{n}$ is the Riemannian volume form, then the Riemannian density is 
 $|v_{g}(x)|$. In local coordinates we have $|v_{g}(x)|=\sqrt{g(x)}dx$, where $dx=|dx^{0}\wedge \ldots \wedge dx^{n}|$ is the Lebesgue measure of 
 $\R^{n}$. 

\begin{proposition}\label{prop:GJMS.Riemannian-invariant-log-sing}
   Let $P_{g}$ be a Riemannian invariant \psido\ of order $m$ and weight~$w$.\smallskip 
   
   1) The logarithmic singularity $c_{P_{g}}(x)$  is of the form
   \begin{equation}
       c_{P_{g}}(x)=\cI_{P_{g}}(x)|v_{g}(x)|,
%        \label{¥}
   \end{equation}
   where $\cI_{P_{g}}(x)$ is a local Riemannian invariant of weight $\frac{n}{2}+w$.\smallskip
   
   2) Assume that $P_{g}$ is elliptic and is admissible in the sense of 
  Definition~\ref{def:GJMS.Riemannian-invariant-PsiDO}. Then the Green kernel logarithmic singularity of 
   $P$ takes the form 
   \begin{equation}
\gamma_{P_{g}}(x)=\cJ_{P_{g}}(x)|v_{g}(x)|,
%        \label{¥}
   \end{equation}
   where $\cJ_{P_{g}}(x)$ is a local Riemannian invariant of weight $\frac{n}{2}-w$.  
\end{proposition}
\begin{proof}
 Let us write $c_{P_{g}}(x)=\cI_{g}(x)|v_{g}(x)|$. Let $t>0$. Since $P_{g}$ and $t^{-w}P_{g}$ agree up to a smoothing operator we have 
 $t^{-w}c_{P_{g}}(x)=c_{P_{tg}}(x)$. As $dv_{tg}(x)=t^{\frac{n}{2}}|v_{g}(x)|$ we see that  $\cI_{P_{tg}}(x)=t^{-(\frac{n}{2}+w)}\cI_{P_{g}}(x)$.
 
 On the other hand, since $P_{g}$ is a Riemannian invariant \psido\ there exist finitely many symbols 
 $a_{\alpha\beta}\in S_{m-j}(M_{n}(\R)_{+}\times\Rn)$ such that in any local coordinates the symbol of degree $-n$ of $P_{g}$ is 
 $p_{-n}(x,\xi)=\sum (\partial^{\alpha}g(x))^{\beta}a_{\alpha\beta}(g(x),\xi)$. By~(\ref{eq:Log-sing.formula-cP}) in local coordinates we have 
 $c_{P_{g}}(x)=\cI_{g}(x)\sqrt{g(x)}dx=(2\pi)^{-n}(\int_{S^{n-1}}p_{-n}(x,\xi)d^{n-1}\xi )dx$. Thus, 
 \begin{equation}
     \cI_{g}(x)= \frac{1}{\sqrt{g(x)}}\sum  (\partial^{\alpha}g(x))^{\beta} A_{\alpha\beta}(g(x)),
           \label{eq:GJMS.local-expression-cIg(x)}
 \end{equation}
where $A_{\alpha\beta}$ is the smooth function on $M_{n}(\R)_{+}$ defined by 
\begin{equation}
    A_{\alpha\beta}(g):= (2\pi)^{-n}\int_{S^{n-1}}a_{\alpha\beta}(g,\xi)d^{n-1}\xi \qquad \forall g\in M_{n}(\R)_{+} .
%     \label{¥}
\end{equation}
Since the expression~(\ref{eq:GJMS.local-expression-cIg(x)}) 
of $\cI_{g}(x)$ holds in \emph{any} local coordinates this proves that $\cI_{g}(x)$ is a local Riemannian invariant. 

Finally, suppose that $P_{g}$ is elliptic  and is admissible. For each Riemannian manifold 
  $(M^{n},g)$ let $Q_{g}\in \Psi^{-m}(M)$ be a parametrix for $P_{g}$. Then the Green kernel logarithmic singularity $\gamma_{P_{g}}(x)$ agrees with 
   $c_{Q_{g}}(x)$ and Proposition~\ref{prop:GJMS.Riemannian-invariance-properties2} tells us that 
  $Q_{g}$ is a Riemannian invariant \psido\ of weight $-w$. Therefore, it follows from the first part of the proposition that $\gamma_{P_{g}}(x)$ is of 
  the form $\gamma_{P_{g}}(x)=\cJ_{P_{g}}(x)|v_{g}(x)|$, where $\cJ_{P_{g}}(x)$ is a local Riemannian invariant of weight $\frac{n}{2}-w$. 
\end{proof}

\section{Logarithmic Singularities and Local Conformal Invariants}
\label{sec:GJMS}
In this section we shall make use of the program of Fefferman in conformal geometry to give a precise form of the logarithmic singularities of 
conformally invariant Riemannian \psidos. 

\subsection{Conformal invariants and Fefferman's program}
Motivated by the analysis of the singularity of the Bergman kernel of a strictly pseudoconvex domain $D\subset \C^{n+1}$ Fefferman~\cite{Fe:PITCA} 
launched the program of determining all local invariants of a strictly pseudoconvex CR structure. This  was subsequently 
extended to conformal geometry and to more general parabolic geometries (see, e.g.,~\cite{FG:CI}). 

A  \emph{scalar local conformal invariant} of weight $w$ is a scalar local Riemannian invariant $\cI_{g}(x)$ such that
\begin{equation}
     \cI_{e^{f} g}(x)= e^{-wf(x)}\cI_{g}(x) \qquad \forall f\in C^{\infty}(M,\R).
%     \label
\end{equation}

The most important conformally invariant tensor is the Weyl curvature,
\begin{equation}
    W_{ijkl}=R_{ijkl}-(P_{jk}g_{il}+P_{il}g_{jk}-P_{jl}g_{ik}-P_{ik}g_{jl}), 
\end{equation}
where $P_{jk}=\frac{1}{n-2}(\rho_{jk}-\frac{\kappa_{g}}{2(n-1)}g_{jk})$ denotes the Schouten tensor. The Weyl tensor is conformally invariant of weight 1, 
so we get scalar conformal invariants by taking complete tensorial contractions. For instance as invariant of weight $2$ we get
\begin{equation}
    |W|^{2}=W^{ijkl}W_{ijkl},
\end{equation}
and as invariants of weight 3 we have 
\begin{equation}
    W_{ij}^{\mbox{~}\mbox{~}kl}W_{lk}^{\mbox{~}\mbox{~}pq}W_{pq}^{\mbox{~}\mbox{~}ij}\qquad \text{and} \qquad 
    W_{i\mbox{~}\mbox{~}l}^{~jk}W^{i\mbox{~}\mbox{~}q}_{~pk}W_{j\mbox{~}\mbox{~}q}^{~pl}.
     \label{eq:GJMS.weight3-conformal-invariants}
\end{equation}

The aim of the program of Fefferman in conformal geometry is to exhibit a basis of local conformal invariants. It was initially conjectured that 
such a basis should involve the Weyl conformal invariants defined in terms of the Lorentz ambient metric of Fefferman-Graham~(\cite{FG:CI}, \cite{FG:AM}) as follows. 

Let $\cG$ be the $\R_{+}$-bundle of metrics defined by the conformal class of $g$. We identify $\cG$ 
with the hypersurface $\cG_{0}=\cG\times\{0\}$ in $\tilde{\cG}=\cG\times (-1,1)$. The ambient metric then is a Ricci-flat Lorentzian metric $\tilde{g}$ on 
$\tilde{\cG}$ defined formally near $\cG_{0}$. In odd dimension the jets of the ambient metric are defined at any order near $\cG_{0}$, but in even dimension 
there is an obstruction for defining them at order~$\geq \frac{n}{2}$. In any case, the local Riemannian invariants of $\tilde{g}$ on $\tilde{\cG}$ 
push down to conformal invariants of $g$ on $M$. The latter are the \emph{Weyl conformal invariants}. 

For instance the Weyl curvature corresponds to the ambient curvature tensor $\tilde{R}$. Moreover, the Ricci flatness of $\tilde{g}$ and the Bianchi identities 
imply that complete tensorial contractions covariant derivatives of $\tilde{R}$ involving internal traces must vanish. 
For instance, there is no scalar Weyl conformal invariant of weight 1 
(in fact there is no scalar conformal invariant of weight 1 at all) and the only non-zero scalar Weyl conformal invariant of weight 2 is $|W|^{2}$, 
which arises from the ambient invariant $|\tilde{R}|^{2}$ (all the other invariants~(\ref{eq:GJMS.weight2-Riemannian-invariants}) 
associated to the ambient metric are zero).  

In addition, the scalar  Weyl conformal invariants of weight $3$ consist of the invariants~(\ref{eq:GJMS.weight3-conformal-invariants}) together with the invariant 
$\Phi_{g}$ exhibited by Fefferman-Graham~(\cite{FG:CI}, \cite{FG:AM}). The latter is the conformal invariant arising from the ambient invariant 
$|\nabla \tilde{R}|^{2}$ and is explicitly 
given by the formulas:
\begin{equation}
    \Phi_{g}=|V|^{2}+16\acou{W}{U}+16|C|^{2},
    \label{eq:GJMS.Phi-invariant}
\end{equation}
where $C_{jkl}=\nabla_{l}A_{jk}-\nabla_{k}A_{jl}$ is the Cotton tensor and $V$ and $U$ are the tensors
\begin{gather}
    V_{sijkl}=\nabla_{s}W_{ijkl}-g_{is}C_{jkl}+g_{js}C_{ikl}-g_{ks}C_{lij}+g_{ls}C_{kij}, \\
  U_{sjkl}=\nabla_{s}C_{jkl}+g^{pq}A_{sp}W_{qjkl}.
%     \label
\end{gather}

Next, a very important result is: 

\begin{proposition}[{\cite[Thm.~11.1]{BEG:ITCCRG}}]\label{prop:GJMS.BEG} 
   1)  In odd dimension every scalar local conformal invariant  is a linear combination of Weyl conformal invariants.\smallskip
   
   2) In even dimension every scalar local conformal invariant for weight $w \leq 
    \frac{n}{2}-1$ is a linear combination of Weyl conformal invariants.
\end{proposition}

In even dimension a description of the scalar local conformal invariants of weight $w\geq \frac{n}{2}+1$ was recently presented by 
Graham-Hirachi~\cite{GH:IAM}.  More precisely, they modified the construction of the ambient metric in such way to obtain a metric on the ambient 
space $\tilde{\cG}$ which is smooth at any order near $\cG_{0}$. There is an ambiguity on the choice of a smooth ambient metric, but such a 
metric agrees with the ambient metric of Fefferman-Graham up to order~$< \frac{n}{2}$ near $\cG_{0}$.

Using a smooth ambient metric we can construct Weyl conformal invariants in the same way as we do by using  the ambient metric of Fefferman-Graham. If such an 
invariant does not depend on the choice of the smooth ambient metric we then say that it is a  \emph{ambiguity-independent} Weyl conformal invariant. 
Not every conformal invariant arises this way, since in dimension $n =4m$ this construction does not encapsulate the exceptional 
local conformal invariants of~\cite{BG:EIPITCG}. 
However, we have:

\begin{proposition}[\cite{GH:IAM}]\label{prop:CI.GH} Let $w$ be an integer~$\geq \frac{n}{2}$.\smallskip 
    
    1) If $n =2 \ \bmod 4$, and if $n =0 \ \bmod 4$ and $w$ is odd, then every scalar local conformal of weight $w$ is a linear combination 
    of ambiguity-independent Weyl conformal invariants.\smallskip
    
    2)  If $n =0 \ \bmod 4$ and $w$ is odd, then every scalar local conformal of weight $w$ is a linear combination 
    of ambiguity-independent Weyl conformal invariants and of exceptional conformal invariants.
\end{proposition}

\subsection{Logarithmic singularities of conformally invariant Riemannian $\mathbf{\Psi}$DOs}
Let us now look at the logarithmic singularities of conformally invariant Riemannian \psidos. The latter are defined are follows. 

\begin{definition}
    A conformally invariant Riemannian \psido\ of order $m$ and biweight $(w,w')$ is a Riemannian invariant $m$'th order \psido\ $P_{g}$ such that, 
    for any $f \in C^{\infty}(M,\R)$, we have 
    \begin{equation}
      P_{e^{f}g}=e^{w'f} P_{g}e^{-wf} \quad \bmod \Psi^{-\infty}(M).
         \label{eq:GJMS.conformal-invariance-biweight}
    \end{equation}
\end{definition}

\begin{remark}
    It follows from~(\ref{eq:GJMS.conformal-invariance-biweight}) 
    that a conformally invariant Riemannian \psido\ of biweight $(w,w')$ is a Riemannian invariant \psido\ of weight $w'-w$. 
\end{remark}

The main result of this section is: 

\begin{theorem}\label{thm:GJMS.main-result}
    Let $P_{g}$ be a conformally invariant Riemannian \psido\ of integer order $m$ and biweight $(w,w')$.\smallskip  

    1) In odd dimension, as well as in even dimension when $w'>w$, the logarithmic singularity $c_{P_{g}}(x)$  is of the form
   \begin{equation}
       c_{P_{g}}(x)=\cI_{P_{g}}(x)|v_{g}(x)|,
%        \label{¥}
   \end{equation}
   where $\cI_{P_{g}}(x)$ is a universal linear combination of Weyl conformal invariants of weight $\frac{n}{2}+w-w'$. If $n$ is even and we have $w'\leq 
   w$, then $c_{P_{g}}(x)$ still is of a similar form, but in this case $\cI_{P_{g}}(x)$ is a local conformal  invariant of weight $\frac{n}{2}+w-w'$ of the 
   type described in Proposition~\ref{prop:CI.GH}.\smallskip
   
   2) Suppose that $P_{g}$ is elliptic and is admissible in the sense of Definition~\ref{def:GJMS.Riemannian-invariant-PsiDO}. 
   Then in odd dimension, as well as in even dimension when $w'<w$, the Green kernel logarithmic singularity of $P$ takes the form
   \begin{equation}
\gamma_{P_{g}}(x)=\cJ_{P_{g}}(x)|v_{g}(x)|,
%        \label{¥}
   \end{equation}
   where $\cJ_{P_{g}}(x)$ is a universal linear combination of Weyl conformal invariants of weight $\frac{n}{2}-w+w'$.  If $n$ is even and we have 
   $w'\geq w$, then $\gamma_{P_{g}}(x)$ still is of a similar form, but in this case $\cJ_{P_{g}}(x)$ is a local conformal  invariant of weight 
   $\frac{n}{2}-w+w'$ of the form described in Proposition~\ref{prop:CI.GH}.
\end{theorem}

\begin{proof}
  First, since $P_{g}$ is a Riemannian invariant \psido\ of weight $w-w'$ we see from Proposition~\ref{prop:GJMS.Riemannian-invariant-log-sing} 
  that $c_{P_{g}}(x)$  is of the form 
  $c_{P_{g}}(x)=\cI_{P_{g}}(x)|v_{g}(x)|$, where $\cI_{P_{g}}(x)$ is a local Riemannian invariant of weight $w-w'$. 
  
  Let $f\in C^{\infty}(M,\R)$. 
  As $P_{g}$ is conformally invariant of biweight $(w,w')$, it follows from Proposition~\ref{prop:Conformal.main} that  
  $c_{P_{e^{f}g}}(x)=e^{-(w-w')f}c_{P_{g}}(x)$. Since $|v_{e^{f}g}(x)|=e^{\frac{n}{2}f(x)}|v_{g}(x)|$ we see that $\cI_{P_{e^{f}g}}(x)=
  e^{-(\frac{n}{2}-w+w')f(x)}\cI_{P_{g}}(x)$. Thus $\cI_{P_{g}}(x)$ is a local conformal invariant of weight $\frac{n}{2}+w-w'$. It then follows from 
  Proposition~\ref{prop:GJMS.BEG}  that in odd dimension, and in even dimension when $w< w'$, the invariant $\cI_{P_{g}}(x)$ is a linear combination of 
  Weyl conformal invariants of weight $\frac{n}{2}+w-w'$. When $n$ is even and we have $w'\leq w$ the invariant $\cI_{P_{g}}(x)$ is of the form described 
  in Proposition~\ref{prop:CI.GH}.  
  
   Suppose now that $P_{g}$ is elliptic and is admissible. In the same way as above, it follows from 
   Proposition~\ref{prop:Conformal.main} and Proposition~\ref{prop:GJMS.Riemannian-invariant-log-sing} 
   that $\gamma_{P_{g}}(x)$ takes the form $\gamma_{P_{g}}(x)=\cJ_{P_{g}}(x)|v_{g}(x)|$, 
   where $\cJ_{P_{g}}(x)$ is a local conformal invariant of weight $\frac{n}{2}-w+w'$. We then can apply Proposition~\ref{prop:GJMS.BEG}  
   to deduce that in odd dimension, as well as in 
   even dimension when $w\geq w'$, the invariant $\cJ_{P_{g}}(x)$ is  a linear combination of Weyl conformal invariants of weight $\frac{n}{2}-w+w'$.  
   When $n$ is even and we have $w'\geq w$ the invariant $\cJ_{P_{g}}(x)$ is of the form described 
  in Proposition~\ref{prop:CI.GH}.  
\end{proof}

We shall now make use of Theorem~\ref{thm:GJMS.main-result} 
to  get a precise geometric description of the Green kernel logarithmic singularities of the GJMS operators $\Box_{g}^{(k)}$.

\begin{theorem}\label{thm:GJMS.GJMS}
    1) In odd dimension the Green kernel logarithmic singularity $\gamma_{\Box_{g}^{(k)}}(x)$ is always zero.\smallskip
    
    2) In even dimension and for $k=1,\ldots,\frac{n}{2}$ we have
    \begin{equation}
        \gamma_{\Box_{g}^{(k)}}(x)=c_{g}^{(k)}(x)d\nu_{g}(x),
%         \label{¥}
    \end{equation}
    where $c_{g}^{(k)}(x)$ is a linear combination of Weyl conformal invariants of weight $\frac{n}{2}-k$. In particular, we have
    \begin{gather}
      c_{g}^{(\frac{n}{2})}(x)= (4\pi)^{-\frac{n}{2}}\frac{n}{(n/2)!}, \qquad c_{g}^{(\frac{n}{2}-1)}(x)=0 , \qquad 
        c_{g}^{(\frac{n}{2}-2)}(x)=\alpha_{n}|W(x)|_{g}^{2} , 
        \label{eq:GJMS.cgk02}\\
        c_{g}^{(\frac{n}{2}-3)}(x)= \beta_{n} W_{ij}^{\mbox{~}\mbox{~}kl}W_{lk}^{\mbox{~}\mbox{~}pq}W_{pq}^{\mbox{~}\mbox{~}ij} 
        + \gamma_{n} W_{i\mbox{~}\mbox{~}l}^{~jk}W^{i\mbox{~}\mbox{~}q}_{~pk}W_{j\mbox{~}\mbox{~}q}^{~pl} + \delta_{n}\Phi_{g},
           \label{eq:GJMS.cgk3}
    \end{gather}
    where $W$ is the Weyl curvature tensor, $\Phi_{g}$ is the Fefferman-Graham invariant~(\ref{eq:GJMS.Phi-invariant}) and 
    $\alpha_{n}$, $\beta_{n}$, $\gamma_{n}$ and $\delta_{n}$ are universal constants depending only on $n$.
\end{theorem}
\begin{proof}
  Let $Q_{g}^{(k)}\in \Psi^{-2k}(M)$ be a parametrix for $\Box_{g}^{(k)}$. Since $\Box_{g}^{(k)}$ is a differential operator, 
  using~(\ref{eq:GJMS.symbol-parametrix1})--(\ref{eq:GJMS.symbol-parametrix3}) one 
  can check that  if $q^{(k)}\sim \sum q^{(k)}_{-2k-j}$ denotes the symbol of $Q_{g}^{(k)}$ in local coordinates then we have 
  $q^{(k)}_{-2k-j}(x,-\xi)=(-1)^{-2k-j}q^{(k)}_{-2k-j}(x,\xi)$ for all $j\geq 0$. Combining this with~(\ref{eq:Log-sing.formula-cP}) then gives
 \begin{equation}
      c_{Q_{g}^{(k)}}(x)= (2\pi)^{-n}\int_{S^{n-1}}q_{-n}^{(k)}(x,-\xi)d^{n-1}\xi=(-1)^{n}  c_{Q_{g}^{(k)}}(x).
% %     \label
 \end{equation}
 Hence $c_{Q_{g}^{(k)}}(x)$ must vanish when $n$ is odd. Since by definition $\gamma_{\Box_{g}^{(k)}}(x)= c_{Q_{g}^{(k)}}(x)$ this shows that 
 $\gamma_{\Box_{g}^{(k)}}(x)$ is always zero in odd dimension. 

  Next, suppose that $n$ is even and $k$ is 
  between $1$ and $\frac{n}{2}$. It follows from the construction in~\cite{GJMS:CIPLIE} that $\Box_{g}^{(k)}$ is a Riemannian invariant operator, so 
  by combining this with~(\ref{eq:Conformal.GJMS2}) we see that $\Box_{g}^{(k)}$ is a conformally invariant Riemannian 
  operator of biweight $(\frac{2k-n}{4},-\frac{n+2k}{4})$. Furthermore, by~(\ref{eq:Conformal.GJMS1}) 
  the principal symbol of $\Box_{g}^{(k)}$ agrees with that of $\Delta_{g}^{(k)}$, so $\Box_{g}^{(k)}$ is admissible in the sense of 
  Definition~\ref{def:GJMS.Riemannian-invariant-PsiDO}. 
 We then can apply Theorem~\ref{thm:GJMS.main-result} 
  to deduce that  $\gamma_{\Box_{g}^{(k)}}(x)$ 
  is of the form $\gamma_{\Box_{g}^{(k)}}(x)=c_{g}^{(k)}(x)|v_{g}(x)|$, where $c_{g}^{(k)}(x)$ is a linear combination of Weyl conformal 
  invariants of weight $\frac{n}{2}-k$. 
   
As mentioned earlier there are no scalar Weyl conformal invariants of weight 1,  the only invariant of weight 2 is $|W|^{2}$, and the 
only Weyl invariants of  weight $3$ 
    are $W_{ij}^{\mbox{~}\mbox{~}kl}W_{lk}^{\mbox{~}\mbox{~}pq}W_{pq}^{\mbox{~}\mbox{~}ij}$ and 
    $W_{i\mbox{~}\mbox{~}l}^{~jk}W^{i\mbox{~}\mbox{~}q}_{~pk}W_{j\mbox{~}\mbox{~}q}^{~pl}$ and 
    the invariant $\Phi_{g}$. From this we get the formulas~(\ref{eq:GJMS.cgk02}) and~(\ref{eq:GJMS.cgk3}) for $c_{g}^{(k)}(x)$ when $k=1,2,3$. 
    
 The formula for $c_{n}^{(\frac{n}{2})}(x)$ follows from a direct computation. More precisely, 
 as $Q_{g}^{(\frac{n}{2})}$ has order $-n$ its symbol of degree $-n$ agrees with its principal symbol, which is the inverse of that of  
 $\Box_{g}^{(\frac{n}{2})}$. By~(\ref{eq:Conformal.GJMS1}) 
 the latter agrees with the principal symbol of $\Delta_{g}^{\frac{n}{2}}$. Therefore, in local coordinates 
 the principal symbol of $\Box_{g}^{(\frac{n}{2})}$ is $p_{n}^{(\frac{n}{2})}(x,\xi)=|\xi|^{n}_{g}$, where $|\xi|^{2}_{g}:=g^{ij}(x)\xi_{i}\xi_{j}$, and that of 
 $Q_{g}^{(\frac{n}{2})}$ is $q_{-n}^{(\frac{n}{2})}(x,\xi)=|\xi|^{-n}_{g}$. As 
 $c_{g}^{(\frac{n}{2})}\sqrt{g(x)}dx=\gamma_{\Box_{g}^{(\frac{n}{2})}}(x)=c_{Q_{g}^{(\frac{n}{2}})}(x)$, using~(\ref{eq:Log-sing.formula-cP}) we see
 that $c_{g}^{(\frac{n}{2})}(x)$ is equal to
  \begin{equation}
      \frac{(2\pi)^{-n}}{\sqrt{g(x)}} \int_{S^{n-1}}|\xi|_{g}^{-n}d^{n-1}\xi
     = (2\pi)^{-n} \int_{S^{n-1}}|\xi|^{-n}d^{n-1}\xi = (2\pi)^{-n}|S^{n-1}|.
%      \label{¥}
 \end{equation}
 Since $|S^{n-1}|=\frac{2\pi^{\frac{n}{2}}}{\Gamma(\frac{n}{2})}=\frac{n \pi^{\frac{n}{2}}}{(n/2)!}$ it follows that
 $c_{g}^{(\frac{n}{2})}(x)=\frac{n(4\pi)^{-\frac{n}{2}}}{(n/2)!}$ as desired.
\end{proof}

Finally, we can get an explicit expression for $c_{g}^{(1)}(x)$ in dimension 6 and 8 by making use the explicit computations  by  
Parker-Rosenberg~\cite{PR:ICL} in these dimensions of the coefficient $a_{n-2}(\Box_{g})(x)$ of $t^{-1}$ 
in the heat kernel asymptotics~(\ref{eq:Log-sing.heat-kernel-asymptotics}) for the Yamabe operator. 

Assume first that $M$ compact. Then by~(\ref{eq:Log-sing.Green-heat}) we have $2\gamma_{\Box_{g}}(x)=a_{n-2}(\Box_{g})(x)$, 
so by using~\cite[Prop.~4.2]{PR:ICL} we see that in dimension 6 we have 
    \begin{equation}
        c_{g}^{(1)}(x)=\frac{1}{360}|W(x)|^{2},
    \end{equation}
while in dimension $8$ we get
     \begin{equation}
        c_{g}^{(1)}(x)=\frac{1}{90720}(81 \Phi_{g} + 352W_{ij}^{\mbox{~}\mbox{~}kl}W_{lk}^{\mbox{~}\mbox{~}pq}W_{pq}^{\mbox{~}\mbox{~}ij} 
        +64 W_{i\mbox{~}\mbox{~}l}^{~jk}W^{i\mbox{~}\mbox{~}q}_{~pk}W_{j\mbox{~}\mbox{~}q}^{~pl}).
         \label{eq:Conformal.w81}
    \end{equation}

  In fact, as $c_{g}^{(1)}(x)$ is a local Riemannian invariant its expression in local coordinates is independent of whether $M$ is compact 
or not. Therefore, the above formulas continue to hold when $M$ is not compact.

\section{Heisenberg calculus and noncommutative residue}
\label{sec:Heisenberg}
The relevant pseudodifferential calculus to study  the main geometric operators on a CR manifold is the Heisenberg calculus of 
Beals-Greiner~\cite{BG:CHM} and Taylor~\cite{Ta:NCMA}. In this section we recall the main definitions and properties of this calculus.

\subsection{Heisenberg manifolds}
The Heisenberg calculus holds in full generality on Heisenberg manifolds. Such a manifold consists of a pair $(M,H)$ where $M$ is a manifold and $H$ 
is a distinguished hyperplane bundle of $TM$. This definition covers many examples: Heisenberg group, CR manifolds, contact manifolds, as well as
(codimension 1) foliations. In addition, given another Heisenberg manifold $(M',H')$ we say that a diffeomorphism 
$\phi:M\rightarrow M'$ is a Heisenberg  diffeomorphism when $\phi_{*}H=H'$. 

The terminology Heisenberg manifold stems from the fact that the relevant tangent structure in this setting is that of a bundle $GM$ of graded nilpotent Lie 
groups (see, e.g.,~\cite{BG:CHM}, \cite{EMM:RLSPD}, \cite{Gr:CCSSW}, \cite{Po:Pacific1}, \cite{Ro:INA}). 
This tangent Lie group bundle can be described as follows. 

First, there is an intrinsic Levi form as the 2-form $\cL:H\times H\rightarrow TM/H$ such that, for any point $a 
\in M$ and any sections $X$ and $Y$ of $H$ near $a$, we have 
\begin{equation}
    \cL_{a}(X(a),Y(a))=[X,Y](a) \qquad \bmod H_{a}.
     \label{eq:NCRP.Levi-form}
\end{equation}
In other words the class of $[X,Y](a)$ modulo $H_{a}$ depends only on $X(a)$ and $Y(a)$, not on the germs of $X$ and $Y$ near $a$ (see~\cite{Po:Pacific1}). 

We define the tangent Lie algebra bundle $\fg M$ as the graded Lie algebra bundle consisting of $(TM/H)\oplus H$ together with the 
fields of Lie bracket and dilations such that, for sections $X_{0}$, $Y_{0}$ of $TM/H$ and $X'$, $Y'$ 
of $H$ and for $t\in \R$, we have 
\begin{equation}
    [X_{0}+X',Y_{0}+Y']=\cL(X',Y'), \qquad  t.(X_{0}+X')=t^{2}X_{0}+t X' .
    \label{eq:NCRP.Heisenberg-dilations} 
%     \label{eq:¥}
\end{equation}

Each fiber $\fg_{a}M$ is a two-step nilpotent Lie algebra so, by requiring the exponential map to be the identity, 
the associated tangent Lie group bundle $GM$ appears as $(TM/H)\oplus H$ together with the grading above and the product law such that, 
for sections $X_{0}$, $Y_{0}$ of $TM/H$ and $X'$, $Y'$ of $H$, we have 
 \begin{equation}
     (X_{0}+X').(Y_{0}+Y')=X_{0}+Y_{0}+\frac{1}{2}\cL(X',Y')+X'+Y'.
%      \label{eq:}
 \end{equation}
 
 Moreover, if $\phi$ is a Heisenberg diffeomorphism from $(M,H)$ onto a Heisenberg manifold $(M',H')$ then, as we have $\phi_{*}H=H'$, we get linear isomorphisms 
 from $TM/H$ onto $TM'/H'$ and from $H$ onto $H'$, which can be combined together to give rise to a linear isomorphism 
 $\phi_{H}':(TM/H)\oplus H\rightarrow (TM'/H')\oplus H'$. In fact $\phi_{H}'$ is a graded Lie group isomorphism from $GM$ onto $GM'$ 
 (see~\cite{Po:Pacific1}). 

\subsection{Heisenberg calculus}
 The initial idea in the Heisenberg calculus, which goes back to Stein, is to construct a 
 class of operators on a Heisenberg manifold $(M^{d+1},H)$, called \psivdos, which at any point $a \in M$ are modeled on homogeneous left-invariant 
 convolution operators on the tangent group $G_{a}M$. 
 
 Locally the \psivdos\ can be described  as follows. Let $U \subset \Rd$ be an open of local coordinates together with a frame $X_{0},\ldots,X_{d}$ of $TU$ such that 
 $X_{1},\ldots,X_{d}$ span $H$. Such a frame is called a \emph{$H$-frame}. Moreover, on $\Rd$ we introduce the dilations and the pseudonorm,
\begin{gather}
     t.\xi=(t^{2}\xi_{0},t\xi_{1},\ldots,t\xi), \qquad t>0, 
     \label{eq:HC.Heisenberg-dilations}\\ 
     \|\xi\|=(\xi_{0}^{2}+\xi_{1}^{4}+\ldots+\xi_{d}^{4})^{1/4}.
\end{gather}
In addition, for any multi-order $\alpha\in \N^{d+1}$ we set $\brak\beta=2\beta_{0}+\beta_{1}+\ldots+\beta_{d}$. 

The Heisenberg symbols are defined as follows. 

\begin{definition}1)  $S_{m}(\URd)$, $m\in\C$, is the space of functions %is the space of symbols 
    $p(x,\xi)$ in $C^{\infty}(U\times\Rdo)$ such that $p(x,t.\xi)=t^m p(x,\xi)$ for any $t>0$.\smallskip

2) $S^m(\URd)$,  $m\in\C$, consists of functions  $p\in C^{\infty}(\URd)$ with
an asymptotic expansion $ p \sim \sum_{j\geq 0} p_{m-j}$, $p_{k}\in S_{k}(\URd)$, in the sense that, for any integer $N$, 
any compact $K \subset U$ and any multi-orders $\alpha$, $\beta$, there exists a constant $C_{NK\alpha\beta}>0$ such that, 
for any $x\in K$ and any $\xi \in \Rd$ so that $\|\xi \| \geq 1$, we have
\begin{equation}
    | \partial^\alpha_{x}\partial^\beta_{\xi}(p-\sum_{j<N}p_{m-j})(x,\xi)| \leq 
    C_{NK\alpha\beta}\|\xi\|^{\Re m-\brak\beta -N}.
    \label{eq:NCRP.asymptotic-expansion-symbols}
\end{equation}
\end{definition}

Next, for $j=0,\ldots,d$ let  $\sigma_{j}(x,\xi)$ denote the symbol (in the 
classical sense) of the vector field $\frac{1}{i}X_{j}$  and set  $\sigma=(\sigma_{0},\ldots,\sigma_{d})$. Then for $p \in S^{m}(\URd)$ we let $p(x,-iX)$ be the 
continuous linear operator from $C^{\infty}_{c}(U)$ to $C^{\infty}(U)$ such that 
    \begin{equation}
          p(x,-iX)u(x)= (2\pi)^{-(d+1)} \int e^{ix.\xi} p(x,\sigma(x,\xi))\hat{u}(\xi)d\xi
    \qquad \forall u\in C^{\infty}_{c}(U).
\end{equation}

Let $(M^{d+1},H)$ be a Heisenberg manifold and let $\cE$ be a vector bundle over $M$. 
We define \psivdos\ on $M$ acting on the sections of $\cE$ as follows. 

\begin{definition}
 $\pvdo^{m}(M,\cE)$, $m\in \C$, consists of continuous operators  $P$ from 
$C^{\infty}_{c}(M,\cE)$ to $C^{\infty}(M,\cE)$ such that:\smallskip

(i) The Schwartz kernel of $P$ is smooth off the diagonal;\smallskip 

(ii) In any trivializing local coordinates equipped with a $H$-frame $X_{0},\ldots,X_{d}$ the operator $P$ can be written as 
\begin{equation}
    P=p(x,-iX)+R,
%     \label
\end{equation}
where $p(x,\xi)$ is a Heisenberg symbol of order $m$ and $R$ is a smoothing operator. 
\end{definition}

 Let $\fg^{*}M$ denote the (linear) dual of the Lie algebra bundle $\fg M$ of $GM$ 
 with canonical projection  $\text{pr}: \fg^{*}M \rightarrow M$. As shown 
 in~\cite{Po:MAMS1} (see also~\cite{EM:HAITH}) 
 the principal symbol of $P\in \pvdo^{m}(M,\cE)$ can be intrinsically defined as a symbol $\sigma_{m}(P)$ of the class below. 
 
\begin{definition}
  $S_{m}(\fg^{*}M,\cE)$, $m\in \C$, consists of sections $p\in C^{\infty}(\fg^{*}M\setminus 0, \End \textup{pr}^{*}\cE)$ which are homogeneous of 
    degree $m$ with respect to the dilations in~(\ref{eq:NCRP.Heisenberg-dilations}), i.e., we have 
   $p(x,\lambda.\xi)=\lambda^{m}p(x,\xi)$ for any $\lambda>0$. 
\end{definition}

%  All these results are obtained by making use of the convolution on the tangent Lie groups $G_{a}M$. 
 For any $a \in M$ the convolution 
 on $G_{a}M$ gives rise under the (linear) Fourier transform to a bilinear product for homogeneous symbols,
 \begin{equation}
     *^{a}: S_{m_{1}}(\fg^{*}_{a}M,\cE_{a})\times S_{m_{2}}(\fg^{*}_{a}M,\cE_{a}) \longrightarrow S_{m_{1}+m_{2}}(\fg^{*}_{a}M,\cE_{a}),
%      \label{eq:¥}
 \end{equation}
 This product depends smoothly on $a$ as much so it gives rise to the product,
 \begin{gather}
     *: S_{m_{1}}(\fg^{*}M,\cE)\times S_{m_{2}}(\fg^{*}M,\cE) \longrightarrow S_{m_{1}+m_{2}}(\fg^{*}M,\cE),\\
     p_{m_{1}}*p_{m_{2}}(a,\xi)=[p_{m_{1}}(a,.)*^{a}p_{m_{2}}(a,.)](\xi).
     \label{eq:NCRP.product-symbols}
 \end{gather}
 This provides us with the right composition for principal symbols, since for any operators $P_{1}\in \pvdo^{m_{1}}(M,\cE)$ and $P_{2}\in 
 \pvdo^{m_{2}}(M,\cE)$ such that $P_{1}$ or $P_{2}$ is properly supported we have
 \begin{equation}
     \sigma_{m_{1}+m_{2}}(P_{1}P_{2})=\sigma_{m_{1}}(P_{1})*\sigma_{m_{2}}(P_{2}). %\qquad \forall P_{j}\in \pvdo^{m_{j}}(M,\cE).
%      \label{eq:}
 \end{equation}

 Notice that when $G_{a}M$ is not commutative, i.e., when $\cL_{a}\neq 0$, the product $*^{a}$ is not anymore the pointwise product of symbols and, in 
 particular, it is not commutative. 
 As a consequence, unless when $H$ is integrable, the product for Heisenberg symbols, while local, it is not microlocal 
 (see~\cite{BG:CHM}). 
 
  When the principal symbol of $P\in \pvdo^{m}(M,\cE)$ is invertible with respect to the product $*$, the symbolic calculus 
 of~\cite{BG:CHM} allows us to construct a parametrix for $P$ in $\pvdo^{-m}(M,\cE)$. In particular, although not elliptic, 
 $P$ is hypoelliptic with a controlled loss/gain of derivatives (see~\cite{BG:CHM}). 
 
 In general, it may be difficult to determine whether the principal symbol of a given operator $P\in \pvdo^{m}(M,\cE)$ is invertible with respect to the product 
 $*$, but  this can be completely determined in terms of a representation theoretic criterion on each tangent 
 group $G_{a}M$, the so-called Rockland condition (see~\cite{Po:MAMS1}, Thm.~3.3.19).  In particular, 
 if $\sigma_{m}(P)(a,.)$ is \emph{pointwise} invertible with respect to the product $*^{a}$ for any 
 $a \in M$ then  $\sigma_{m}(P)$ is \emph{globally} invertible with respect to $*$.

\subsection{Logarithmic singularity and noncommutative residue}
It is possible to characterize the \psivdos\ in terms of their Schwartz kernels (see~\cite{BG:CHM}). As a consequence we get the following 
description of the singularity near the diagonal of the Schwartz kernel of a \psivdo. 

In the sequel, given an open of local coordinates $U\subset \Rd$ equipped with a $H$-frame $X_{0},\ldots,X_{d}$ of $TU$, 
for any $a\in U$ we let $\psi_{a}$ denote the unique affine change of variables such that $\psi_{a}(a)=0$ 
 and $(\psi_{a*}X_{j})(0)=\frac{\partial}{\partial x_{j}}$ for $j=0,1,\ldots,d+1$. 
 
\begin{definition}\label{def:Heisenberg.privileged-coordinates}
    The local coordinates provided by $\psi_{a}$ are called privileged coordinates centered at $a$.
\end{definition}

Throughout the rest of the paper the notion of homogeneity refers to homogeneity with respect to the anisotropic dilations~(\ref{eq:HC.Heisenberg-dilations}). 
 
\begin{proposition}[{\cite[Prop.~3.11]{Po:JFA1}}]\label{prop:Heisenberg.Log-sing}
    Let $\pvdo^{m}(M,\cE)$, $m \in \Z$.\smallskip 
    
    1) In local coordinates equipped with a $H$-frame the kernel $k_{P}(x,y)$ has a behavior near the diagonal $y=x$ of the form
    \begin{equation}
        k_{P}(x,y)=\sum_{-(m+d+2)\leq j\leq -1}a_{j}(x,-\psi_{x}(y)) - c_{P}(x)\log \|\psi_{x}(y)\| + 
        \op{O}(1), 
         \label{eq:Heisenberg.log-sing}
    \end{equation}
    where $a_{j}(x,y)\in C^{\infty}(U\times (\Rno))$ is homogeneous of degree $j$ in $y$, and we have
%     $c_{P}(x)$ is given by the formula: 
 \begin{equation}
    c_{P}(x)=(2\pi)^{-(d+1)}\int_{\|\xi\|=1}p_{-(d+2)}(x,\xi)\iota_{E}d\xi,
     \label{eq:Heisenberg.formula-cP}
\end{equation}
where $p_{-(d+2)}(x,\xi)$ is the symbol of degree $-(d+2)$ of $P$ and $E$ denotes the anisotropic radial vector 
$2x^{0}\partial_{x^{0}}+x^{1}\partial_{x^{1}}+\ldots +x^{d}\partial_{x^{d}}$.\smallskip    
    
    2) The coefficient $c_{P}(x)$ makes  sense globally on $M$ as an $\END\cE$-valued density.
\end{proposition}

Let $P\in \pvdo^{m}(M,\cE)$ be such that its principal symbol is invertible in the Heisenberg calculus sense and let $Q\in \pvdo^{-m}(M,\cE)$ be a 
parametrix for $P$. Then $Q$ is uniquely defined modulo smoothing operators, so the logarithmic singularity $c_{Q}(x)$ does not depend on the 
particular choice of $Q$. 

\begin{definition}
  If $P\in \pvdo^{m}(M,\cE)$, $m \in \Z$, has an invertible principal symbol, then its Green kernel logarithmic singularity is the density
  \begin{equation}
      \gamma_{P}(x):=c_{Q}(x),
%       \label{¥}
  \end{equation}
  where $Q\in \pvdo^{-m}(M,\cE)$ is any given parametrix for $P$.
\end{definition}

In the same way as for classical \psidos\ the logarithmic singularity densities are related to the construction of the noncommutative residue trace 
for the Heisenberg calculus (see~\cite{Po:JFA1}). 

Let $\pvdoi(M,\cE) = \cup_{\Re m < -(d+2)}\pdo^{m}(M,\cE)$ be the class of \psivdos\ whose symbols are integrable with respect to the $\xi$-variable. If 
$P$ is an operator in this class, then the restriction of its Schwartz kernel $k_{P}(x,y)$ to the diagonal defines a smooth $\End \cE$-valued density 
$k_{P}(x,x)$. In particular, if $M$ is compact, then $P$ is trace-class and its trace is given by~(\ref{eq:Log-sing.trace}). 

The map $P\rightarrow k_{P}(x,x)$ admits an analytic continuation $P\rightarrow t_{P}(x)$ to the class 
$\pvdocz(M,\cE)$ of non-integer order \psivdos, where is analyticity is meant with respect to holomorphic families of \psivdos\ as defined in ~\cite{Po:MAMS1}.
Moreover, if $P\in \pvdoz(M,\cE)$ and if $(P(z))_{z\in\C}$ is a holomorphic family of 
\psivdos\  such that $\ord P(z)=\ord P +z$ and $P(0)=P$, then the map  $z\rightarrow t_{P(z)}(x)$ has at worst a simple pole 
 singularity at  $z=0$ in such way that
\begin{equation}
    \Res_{z=0} t_{P_{z}}(x)=- c_{P}(x).
    \label{eq:NCRHC.residue-densities}
\end{equation}

Assume now that $M$ is compact. Then the \emph{noncommutative residue} for the Heisenberg calculus is the linear functional $\Res$ on $\pvdoz(M,\cE)$ 
defined by 
\begin{equation}
     \Res P :=\int_{M}\tr_{\cE} c_{P}(x) \qquad \forall P\in \pvdoz(M,\cE).
 \end{equation}

 It follows from~(\ref{eq:NCRHC.residue-densities}) that if $(P(z))_{z\in\C}$ is a holomorphic family of 
\psivdos\ such that $\ord P(z)=\ord P +z$ and $P(0)=P$, then the map $z\rightarrow \Tra P(z)$ has an analytic extension to $\CZ$ with at worst a 
simple at $z=0$ in such way that 
\begin{equation}
    \Res_{z=0}\Tra P(z)=-\Res P.
%     \label
\end{equation}
Using this it is not difficult to check that the above noncommutative residue is a trace on $\pvdoz(M,\cE)$.
This is even the unique trace up to constant multiple when $M$ is connected (see~\cite{Po:JFA1}).

Finally, suppose that $M$ is endowed with a positive density and $\cE$ is endowed with a Hermitian metric. Let
 $P:C^{\infty}(M,\cE)\rightarrow C^{\infty}(M,\cE)$ be a selfadjoint \psivdo\ of integer order $m\geq 1$ such that the union set $\theta(P)$ of the principal cuts of 
 its principal symbol agrees with $\C\setminus [0,\infty)$ (see~\cite{Po:CPDE1} for the precise definition of a principal cut). This implies that the 
 principal symbol of $P$ is invertible in the Heisenberg calculus sense. This also implies that $P$ is bounded from 
below, hence gives rise to a heat semigroup $e^{-tP}$, $t\geq 0$.

For any $t>0$ the operator $e^{-tP}$  has a smooth Schwartz kernel $k_{t}(x,y)$ in 
$C^{\infty}(M,\cE)\hotimes C^{\infty}(M,\cE^{*}\otimes |\Lambda|(M))$, and as 
$t\rightarrow 0^{+}$ we have the heat kernel asymptotics, 
\begin{equation}
    k_{t}(x,x)\sim t^{-\frac{d+2}{m}}\sum_{j\geq 0} t^{\frac{j}{m}}a_{j}(P)(x) + \log t\sum_{k\geq 0}t^{k}b_{k}(P)(x),
     \label{eq:Heisenberg.heat-kernel-asymptotics}
\end{equation}
where the asymptotics takes place in $C^{\infty}(M,\End \cE \otimes |\Lambda|(M))$, and when $P$ is a differential operator we have 
$a_{2j-1}(P)(x)=b_{j}(P)(x)=0$ for all $j\in \N$ (see~\cite{BGS:HECRM}, \cite{Po:MAMS1}, \cite{Po:CPDE1}).

As in~(\ref{eq:Log-sing.Log-sing-heat-kernel}) if for $j=0,\ldots, n-1$ we set $\sigma_{j}=\frac{d+2-2j}{m}$, then we have
\begin{equation}
mc_{P^{-\sigma_{j}}}(x)=\Res_{s=\sigma_{j}}t_{P^{-s}}(x)=\Gamma(\sigma_{j})^{-1}a_{2j}(P)(x).
\end{equation}
In particular, we get
\begin{equation}
     m\gamma_{P}(x)= a_{d+2-m}(P)(x).
     \label{eq:Heisenberg.Green-heat}
\end{equation}

\section{Logarithmic singularities of contact invariant operators}
\label{sec:contact-invariance}
The aim of this section is to prove an analogue of Proposition~\ref{prop:Conformal.main} in the setting of contact geometry. 

Let $(M^{2n+1},H)$ be an orientable contact manifold. This means that $(M,H)$ is an orientable Heisenberg manifold such 
such that $H$ can be represented as the annihilator of a 
\emph{globally} defined contact form, that is, a 1-form $\theta$ on $M$ such that $H=\ker \theta$ and $d\theta_{|H}$ is nondegenerate. 
We further assume that $\theta$ is chosen in such way that the top-degree form 
$d\theta^{n}\wedge \theta$ is in the orientation class of $M$. This uniquely determines the contact form $\theta$ up to a conformal factor.

 As we will recall in Section~\ref{sec:CR-GJMS}, the CR GJMS of Gover-Graham~\cite{GG:CRIPSL} on a pseudohermitian manifold transform covariantly under a 
 conformal change of contact form. These operators 
 include the CR Yamabe operator of Jerison-Lee~\cite{JL:YPCRM}, 
 for which N.K.~Stanton~\cite[p.~276]{St:SICRM} determined the behavior of the logarithmic singularity of 
 the Green kernel under a conformal change of contact form. 
 
 More generally, let $\Theta$ be the class of contact forms on $M$ that are conformal multiples of $\theta$, and let 
 let $(P_{\hat{\theta}})_{\hat{\theta}\in \Theta} \subset \pvdo^{m}(M,\cE)$ 
 be a family of $m$th order \psivdos\ in such way that there exist real numbers $w$ and 
$w'$ so that, for any $f$ in $C^{\infty}(M,\R)$, we have 
\begin{equation}
    P_{e^{f}\theta}=e^{w'f} P_{\theta}e^{-wf} \quad \bmod \pvdo^{-\infty}(M,\cE).
    \label{eq:Contact.covariancePtheta}
\end{equation}
 Then the following holds.
 
\begin{proposition}\label{prop:Contact.main}
1) We have
\begin{equation}
    c_{P_{e^{f}\theta}}(x)=e^{-(w-w')f(x)}c_{P_{\theta}}(x) \qquad \forall f\in C^{\infty}(M,\R).
     \label{eq:Contact.contact-invariance-Ptheta}
\end{equation}

2) Suppose that the principal symbol of $P_{\theta}$ is invertible in the sense of the Heisenberg calculus. Then we have
\begin{equation}
    \gamma_{P_{e^{f}\theta}}(x)=e^{-(w'-w)f(x)}\gamma_{P_{\theta}}(x) \qquad \forall f\in C^{\infty}(M,\R).
%     \label
\end{equation}

3) Suppose that, for any $\hat{\theta}\in \Theta$, the operator is $P_{\hat{\theta}}$ is selfadjoint with respect to some density on $M$ 
and some Hermitian metric on $\cE$, and the union-set of the principal cuts of the principal symbol of $P_{\hat{\theta}}$ is $\C\setminus [0,\infty)$. Then we have
 \begin{equation}
     a_{2n+2-m}(P_{e^{f}\theta})(x)=e^{-(w'-w)f(x)}a_{2n+2-m}(P_{\theta})(x) \quad \forall f\in C^{\infty}(M,\R),
% %     \label{¥}
 \end{equation}
where $a_{2n+2-m}(P_{\hat{\theta}})(x)$ is the coefficient of $t^{-1}$ in the heat kernel asymptotics~(\ref{eq:Heisenberg.heat-kernel-asymptotics}) 
for~$P_{\hat{\theta}}$. 
\end{proposition}
 \begin{proof}
     The proof is similar to that of Proposition~\ref{prop:Conformal.main}. Let $f \in C^{\infty}(M,\R)$, set $\hat{\theta}=e^{f}\theta$ and let 
     $k_{P_{\theta}}(x,y)$ and $k_{P_{\hat{\theta}}}(x,y)$ denote the respective Schwartz kernels of $P_{\theta}$ and $P_{\hat{\theta}}$. Then 
     it follows from~(\ref{eq:Contact.contact-invariance-Ptheta}) that we have 
  \begin{equation}
     k_{P_{\hat{\theta}}}(x,y)=e^{w'f(x)} k_{P_{\theta}}(x,y) e^{-wf(y)}+\op{O}(1).
      \label{eq:Contact.kPtheta-kPhtheta}
 \end{equation}

Next, let $U\subset \R^{2n+1}$ be an open of local coordinates equipped with a $H$-frame $X_{0},\ldots,X_{d}$. By Proposition~\ref{prop:Heisenberg.Log-sing} the kernel 
$k_{P_{\theta}}(x,y)$ has a behavior near the diagonal of the form
\begin{equation}
    k_{P_{\theta}}(x,y)=
    \!\!\!\sum_{-(m+2n+2) \leq j \leq -1}\!\!\!a_{j}(x,\psi_{x}(y))-c_{P_{\theta}}(x)\log\|\psi_{x}(y)\| +\op{O}(1),
%     \label{¥}
\end{equation}
where $a_{j}(x,y)\in C^{\infty}(\URno)$ is homogeneous of degree $j$ with respect to $y$. Combining this with~(\ref{eq:Contact.kPtheta-kPhtheta}) 
then gives
\begin{multline}
    k_{P_{\hat{\theta}}}(x,y)=\\
    \sum_{-(m+2n+2) \leq j \leq -1}\!\!\! b(x,\psi_{x}(y))a_{j}(x,\psi_{x}(y)) -c_{P_{\theta}}(x)b(x,\psi_{x}(y))\log\|\psi_{x}(y)\| +\op{O}(1),
     \label{eq:Contact.sing-kPhtheta}
\end{multline}
where we have set $b(x,y)=e^{-wf(\psi_{x}^{-1}(y))+w'f(x)}$. 

The Taylor expansion of $b(x,y)$ near $y=0$ can be written in the form
\begin{equation}
    b(x,y)= \!\!\!\sum_{\brak\alpha< m}\!\!\!\frac{1}{\alpha!}\partial_{y}^{\alpha}b(x,0)y^{\alpha} + \sum_{\brak{\alpha}=m}y^{\alpha}r_{\alpha}(x,y),
    \label{eq:Contact.Taylor-b}
\end{equation}
where the functions $r_{\alpha}(x,y)$ are smooth near $y=x$. 
By arguing as in the proof of Proposition~\ref{prop:Conformal.main}  we can show that 
\begin{gather}
     b(x,y)a_{j}(x,y)=\sum_{\brak\alpha|+j \leq -1}\frac{1}{\alpha!}\partial_{y}^{\alpha}b(x,0)y^{\alpha}a_{j}(x,y) +\op{O}(1),
    \label{eq:Contact.sing-kPhtheta.homogeneous-terms}\\
      b(x,y)\log\|y\| = b(x,0)\log\|y\|+\op{O}(1)=e^{-(w-w')f(x)}\log\|y\| +\op{O}(1).
     \label{eq:Contact.sing-kPhg.log-term}
\end{gather}
Combining this with~(\ref{eq:Contact.sing-kPhtheta}) then shows that
\begin{multline}
     k_{P_{\hat{\theta}}}(x,y)= \\ \sum_{-(m+2n+2) \leq |\alpha|+j \leq -1}\frac{1}{\alpha!}\partial_{y}^{\alpha}b(x,0)\psi_{x}(y)^{\alpha}a_{j}(x,y)  
     -c_{P_{\theta}}(x)e^{-(w-w')f(x)}\log\|\psi_{x}(y)\|  
     +\op{O}(1).
%     \label{¥}
\end{multline}
This shows that $c_{P_{\hat{\theta}}}(x)=e^{-2(w-w')f(x)}c_{P_{\theta}}(x)$ as desired, so the 1st part of the proposition is proved. 

Next, suppose that the principal symbol of $P_{\theta}$ is invertible in the Heisenberg calculus sense. Because of~(\ref{eq:Contact.covariancePtheta}) this implies that the 
principal symbol of $P_{\hat{\theta}}$ is invertible as well. Let $Q_{\theta}\in \pvdo^{-m}(M,\cE)$ be a parametrix for $P_{\theta}$ and, similarly, 
let $Q_{\hat{\theta}}\in \pvdo^{-m}(M,\cE)$ be a parametrix for $P_{\hat{\theta}}$.
By arguing as in the proof of Proposition~\ref{prop:Conformal.main} we can show that 
\begin{equation}
    Q_{\hat{\theta}}=  e^{wf}Q_{\theta} e^{-w'f} \qquad \bmod \Psi^{-\infty}(M,\cE).
%     \label{}
\end{equation}
Therefore, it follows from the first part of the proof that
\begin{equation}
    \gamma_{P_{\theta}}(x)=c_{Q_{\theta}}(x)=e^{-(w'-w)f(x)}c_{Q_{\theta}}(x)=e^{-(w'-w)f(x)}\gamma_{P_{\theta}}(x).
%     \label{¥}
\end{equation}
The 2nd part of the proposition is thus proved.

Finally, thanks to~(\ref{eq:Heisenberg.Green-heat}) the third part of the proposition is an immediate consequence of the second one. 
 \end{proof}
  
\begin{remark}
     The third part of Proposition~\ref{prop:Contact.main} 
     has also been obtained by N.K.~Stanton~\cite[Thm.~3.3]{St:SICRM} in the special case of the CR Yamabe operator on a pseudohermitian 
     manifold. 
\end{remark}
 
\section{Pseudohermitian Invariant \psivdos\ and Their Logarithmic Singularities}
\label{sec:Pseudohermitian}
In this section,  after some preliminary work on local pseudohermitian invariants and pseudohermitian invariant \psivdos, we shall prove that
the logarithmic singularities of the Schwartz kernels and Green kernels of pseudohermitian invariant \psivdos give rise local pseudohermitian invariants. 

\subsection{The geometric set-up}
Let $(M^{2n+1},H)$ be a compact orientable CR manifold. Thus $(M^{2n+1},H)$ is a Heisenberg manifold and $H$ is equipped with a complex structure $J\in 
C^{\infty}(M,\End H)$, $J^{2}=-1$, in such way that $T_{1,0}:=\ker (J+i)\subset T_{\C}M$ is a complex rank $n$ subbundle integrable in 
Fr\"obenius' sense (i.e.~$C^{\infty}(M,T_{1,0})$ is closed under the Lie bracket of vector fields). In addition, we 
set $T_{0,1}:=\overline{T_{1,0}}=\ker(J-i)$. 

Since $M$ is orientable and $H$ is orientable by means of its complex structure, there exists a global non-vanishing real 1-form $\theta$ such 
that $H=\ker \theta$. Associated to $\theta$ is its Levi form, i.e., the Hermitian form on $T_{1,0}$ such that
\begin{equation}
    L_{\theta}(Z,W)=-id\theta(Z,\overline{W})=i\theta([Z,W]) \qquad \forall Z,W \in C^{\infty}(M,T_{1,0}).
%     \label{¥}
\end{equation}

We further assume that  $M$ is strictly pseudoconvex, that is, we can choose $\theta$ so that $L_{\theta}$ is 
positive definite at every point. In particular $\theta$ is a contact form on $M$. In the terminology 
of~\cite{We:PHSRH} the datum of such a contact form defines a \emph{pseudohermitian structure} on $M$. 

Since $\theta$ is a contact form there exists a unique vector field $X_{0}$ on $M$, called 
the \emph{Reeb field}, such that $\iota_{X_{0}}\theta=1$ and $\iota_{X_{0}}d\theta=0$. Let $\cN\subset T_{\C}M$ be the complex line bundle spanned by 
$X_{0}$. We then have the splitting
\begin{equation}
    T_{\C}M=\cN\oplus T_{1,0}\oplus T_{0,1}.
    \label{eq:CR.splitting-TcM}
\end{equation}
The Levi metric $h_{\theta}$ is the unique Hermitian metric on $T_{\C}M$ such that: \smallskip

- The splitting~(\ref{eq:CR.splitting-TcM}) is orthogonal with respect to $h_{\theta}$;\smallskip

 - $h_{\theta}$ commutes with complex conjugation;\smallskip

- We have $h(X_{0},X_{0})=1$ and $h_{\theta}$ agrees with $L_{\theta}$ on $T_{1,0}$.\smallskip

\noindent Notice that the volume form of $h_{\theta}$ is $\frac{1}{n!}d\theta^{n}\wedge \theta$.

As proved by Tanaka~\cite{Ta:DGSSPCM} and Webster~\cite{We:PHSRH} the datum of the pseudohermitian contact 
form $\theta$ uniquely defines a connection, the \emph{Tanaka-Webster connection}, which preserves the pseudohermitian structure of $M$, i.e., such that 
$\nabla \theta=0$ and $\nabla J=0$. It can be defined as follows. 

Let $\{Z_{j}\}$ be a frame of $T_{1,0}$. We set $Z_{\bar{j}}=\overline{Z_{j}}$. Then 
$\{X_{0},Z_{j},Z_{\bar{j}}\}$ forms a frame of $T_{\C}M$. In the sequel such a frame will be called an \emph{admissible frame} of $T_{\C}M$. 
Let $\{\theta,\theta^{j}, \theta^{\bar{j}}\}$ be the coframe of $T^{*}_{\C}M$ dual to $\{X_{0},Z_{j},Z_{\bar{j}}\}$. With respect to this coframe
we can write $d\theta =ih_{j\bar{k}}\theta^{j}\wedge \theta^{\bar{k}}$. 

Using the matrix $(h_{j\bar{k}})$ and its inverse $(h^{j\bar{k}})$ to lower and raise indices, the connection 1-form 
$\omega=(\omega_{j}^{~k})$ and the 
 torsion form $\tau_{j}=A_{jk}\theta^{k}$ of the Tanaka-Webster connection  are uniquely determined by the relations
\begin{equation}
     d\theta^{k}=\theta^{j}\wedge \omega_{j}^{~k}+\theta \wedge \tau^{k}, \qquad 
     \omega_{j\bar{k}} + \omega_{\bar{k}j} 
         =dh_{j\bar{k}}, \qquad  A_{jk}=A_{k j}. 
     \label{eq:CR.TW-connection}
\end{equation}
In addition, we have the structure equations
\begin{equation}
  d\omega_{j}^{~k}-\omega_{j}^{~l}\wedge \omega_{l}^{~k}=R_{j~l\bar{m}}^{~k} \theta^{l}\wedge \theta^{\bar{m}} + 
    W_{j\bar{k}l}\theta^{l}\wedge \theta -  W_{\bar{k}j\bar{l}}\theta^{\bar{l}}\wedge \theta 
    +i\theta_{j}\wedge \tau_{\bar{k}}-i\tau_{j}\wedge \theta_{\bar{k}}.
     \label{eq:CR.TW-curvature}
\end{equation}

The \emph{pseudohermitian curvature tensor} of the Tanaka-Webster connection is the tensor with components $R_{j\bar{k} l\bar{m}}$, its \emph{Ricci tensor} is 
$ \rho_{j\bar{k}}:=R_{l~j \bar{k}}^{~l}$ and its \emph{scalar curvature} is $\kappa_{\theta} :=\rho_{j}^{~j}$.

\subsection{Local pseudohermitian invariants}
Let us now define local pseudohermitian invariants. The definition is a bit more complicated than that of a local Riemannian invariants, because:\smallskip 

- The components of the Tanaka-Webster connections and its curvature and torsion tensors are defined with respect to the datum of a local frame 
$Z_{1},\ldots,Z_{n}$ which does not correspond to frame given by derivatives with respect to coordinate functions;\smallskip

- In order to get local pseudohermitian invariants from pseudohermitian invariant \psivdos\ it is important to take into the tangent group bundle of a CR 
manifold, in which the Heisenberg group comes into play.\smallskip
 
This being said, in order to define local pseudohermitian invariants some notation need to be introduced.

Let $U\subset \Rn$ be an open of local coordinates equipped with a frame 
$Z_{1},\ldots,Z_{n}$ of $T_{1,0}$. We set $Z_{j}=X_{j}-iX_{n+j}$, where $X_{j}$ and $X_{n+j}$ are real vector fields. Then $X_{0},\ldots,X_{2n}$ is a 
local $H$-frame of $TM$. We shall call this frame the \emph{$H$-frame associated to $Z_{1},\ldots,Z_{n}$.}

Let $\eta^{0},\ldots,\eta^{2n}$ be the coframe of $T^{*}M$  dual to $X_{0},\ldots,X_{2n}$ (so that $\eta^{0}=\theta$). We set 
$X_{j}=X_{j}^{~k}\partial_{x^{k}}$ and $\eta^{j}=\eta_{~k}^{j}dx^{k}$. We also set $Z_{j}=Z_{j}^{~k}\partial_{x^{k}}$. It will be convenient to 
identify $X_{0}(x)$ with the vector $(X_{0}^{~k}(x))\in \R^{2n+1}$ and $Z(x):=(Z_{1}(x),\ldots,Z_{n}(x))$ with the matrix $(Z_{j}^{~k}(x))$ in 
$M_{n,2n+1}(\C)^{\times}$, where the latter denotes the open subset of $M_{n,2n+1}(\C)$ consisting of regular matrices. 

For $j,\bar{k}=1,\ldots,n$ we set $h_{j\bar{k}}=h_{\theta}(Z_{j},Z_{k})=i\theta([Z_{j},Z_{\bar{k}}])$, and for $j,k=1,\ldots,2n$ we set 
$L_{jk}=\theta([X_{j},X_{k}])$. Let $M_{n}(\C)_{+}$ denote the open cone of positive definite Hermitian $n\times n$ matrices. In the sequel it will 
also be convenient to identify 
$h_{\theta}$ with the matrix $h_{\theta}(x):=(h_{j\bar{k}}(x))\in M_{n}(\C)_{+}$. 

Thanks to the integrability of $T_{1,0}$ we have $\theta([Z_{j},Z_{k}])=0$. 
As we have $[Z_{j},Z_{k}]=[X_{j},X_{k}]-[X_{n+j},X_{n+k}]-i([X_{n+j},X_{k}]+[X_{j},X_{n+k}])$ we see that  
\begin{equation}
     L_{n+j,n+k}=L_{j,k}\qquad  \text{and} \qquad L_{j,n+k}=-L_{n+j,k}. 
%     \label
\end{equation}
Since 
$[Z_{j},Z_{\bar{k}}]=[X_{j},X_{k}]+[X_{n+j},X_{n+k}]+i([X_{n+j},X_{k}]-[X_{j},X_{n+k}])$ we get
\begin{equation}
  h_{j\bar{k}}= i\theta([Z_{j},Z_{\bar{k}}])= 2iL_{jk}+2L_{n+jk}. 
     \label{eq:CR.L-h0}
\end{equation}
In other words, we have 
\begin{equation}
    (L_{jk})= \frac{1}{2}\left(
    \begin{array}{cc}
        \Im h & -\Re h  \\
        \Re h & \Im h
    \end{array}\right). 
     \label{eq:CR.L-h}
\end{equation}

For any $a \in U$ we let  $\psi_{a}$ be the affine change of variables
to the privileged coordinates centered at $a$ (cf.~Definition~\ref{def:Heisenberg.privileged-coordinates}). 
One checks that $\psi_{a}(x)^{j}=\eta^{j}_{~k}(x^{k}-a^{k})$, so we have 
\begin{equation}
    \psi_{a*}X_{j}=X_{j}^{~k}(\psi_{a}(x))\eta^{l}_{~k}(a)\partial_{l}.
%     \label{¥}
\end{equation}
Given a vector field $X$ defined near $x=0$ let us denote  $X(0)_{l}$ the vector field obtained as the part in the Taylor expansion at $x=0$ of 
$X$ which is homogeneous of degree $l$ with respect to the Heisenberg dilations~(\ref{eq:HC.Heisenberg-dilations}). Then the Taylor 
expansions at $x=0$ of the vector fields $\psi_{a*}X_{0}, \ldots,  \psi_{a*}X_{2n}$ take the form 
\begin{gather}
    X_{0}=X_{0}^{(a)}+X_{0}(0)_{(-1)}+\ldots, \label{eq:CR.asymptotics-vector-fields1} \\ 
    X_{j}=X_{j}^{(a)}+X_{j}(0)_{(0)}+\ldots, \quad 1\leq j \leq 2n,
     \label{eq:CR.asymptotics-vector-fields2}
\end{gather}
with
% where the vector fields $X_{0}^{(a)},\ldots,X_{2n}^{(a)}$ are given by the formulas, 
\begin{equation}
    X_{0}^{(a)}=\partial_{x^{0}}, \qquad 
    X_{j}^{(a)}=\partial_{x^{j}} + b_{jk}(a)x^{k}\partial_{x^{0}}, \quad 
     1\leq j\leq 2n,
     \label{eq:CR.model-vector-fields}
\end{equation}
where we have set $b_{jk}(a):=\partial_{k}[X_{j}^{~l}(\psi_{a}(x))]_{|x=0}\eta^{0}_{~l}(a)$. Notice that $X_{0}^{(a)}$ is homogeneous of degree~$-2$, while 
$X_{1}^{(a)},\ldots,X_{2n}^{(a)}$ are homogeneous of degree~$-1$.

The linear span of the vector fields $X_{0}^{(a)},\ldots,X_{2n}^{(a)}$ is a 2-step niloptent Lie algebra under the Lie bracket of 
vector fields. Therefore, this is 
the Lie algebra of left-invariant vector fields on a 2-step nilpotent Lie group $G^{(a)}$. The latter can be realized as $\Rtn$ equipped with the 
product,
\begin{equation}
    x.y=(x^{0}+y^{0}+b_{kj}(a)x^{j}y^{k},x^{1}+y^{1},\ldots,x^{2n}+y^{2n}).
\end{equation}
Notice that $[X_{j}^{(a)},X_{k}^{(a)}]=(b_{kj}(a)-b_{jk}(a))X_{0}^{(a)}$. In addition, we can check that
$[\psi_{a*}X_{j},\psi_{a*}X_{k}](0)=(b_{kj}(a)-b_{jk}(a))\partial_{x^{0}} \bmod H_{0}$. Thus, 
\begin{multline}
    L_{jk}(a)=\theta(X_{j},X_{k})(a)=  (\psi_{a*}\theta)([\psi_{a*}X_{j},\psi_{a*}X_{k}])(0) \\ 
    =\acou{dx^{0}}{[\psi_{a*}X_{j},\psi_{a*}X_{k}](0)}=b_{kj}(a)-b_{jk}(a). 
     \label{eq:CR.Ljk-bjk}
\end{multline}
This shows that $G^{(a)}$ has the same constant structures as the tangent group $G_{a}M$, hence is isomorphic to it (see~\cite{Po:Pacific1}). This also implies that 
$(-\frac{1}{2}L_{jk}(a))$ is the skew-symmetric part of $(b_{jk}(a))$.   
For $j,k=1,\ldots,2n$ set $\mu_{jk}(a)=b_{jk}(a)+\frac{1}{2}L_{jk}(a)$. The matrix $(\mu_{jk}(a))$ is the symmetric part of $(b_{jk}(a))$, so it 
belongs to the space $S_{2n}(\R)$ of symmetric $2n\times 2n$ matrices with real coefficients. 

In the sequel we set 
\begin{equation}
    \Omega=M_{n}(\C)_{+}\times \R^{2n+1}\times M_{n,2n+1}(\C)^{\times}\times 
    S_{2n}(\R).
%     \label
\end{equation}
 This is a manifold, and for any $x \in U$ the quadruple $(h(x),X_{0}(x),Z(x),\mu(x))$ is an element of $\Omega$ depending smoothly on $x$. 

In addition, we let $\cP$ be the set of monomials in the undetermined variables $\partial^{\alpha} X_{0}^{~k}$, $\partial^{\alpha}Z_{j}^{~k}$ and 
$\partial^{\alpha} \overline{Z_{j}^{~k}}$, where the integer $j$ ranges over $\{1,\ldots,n\}$, the integer 
$k$ ranges over $\{0,\ldots,2n\}$, and $\alpha$ ranges over all multi-orders in $\N_{0}^{2n}$. Given the Reeb field $X_{0}$ and a local frame 
$Z_{0},\ldots,Z_{n}$ of $T_{1,0}$ by plugging  $\partial^{\alpha}_{x} X_{0}^{~k}(x)$, $\partial^{\alpha}_{x}Z_{j}^{~k}(x)$ and 
$\partial^{\alpha} \overline{Z_{j}^{~k}}(x)$ into a monomial $\mathfrak{p}\in \cP$ we get a function which we shall denote 
$\fp(X_{0},Z,\overline{Z})(x)$.

 Bearing all this mind we define local pseudohermitian invariants as follows.

\begin{definition}
    A local pseudohermitian invariant of weight $w$ is the datum on each pseudohermitian manifold $(M^{2n+1},\theta)$ of a function 
    $\cI_{\theta}\in C^{\infty}(M)$ such that:\smallskip
    
    (i) There exists a finite family $(a_{\fp})_{\fp \in \cP}\subset C^{\infty}(\Omega)$ 
   such that, in any local coordinates equipped with a frame $Z_{1},\ldots,Z_{n}$ of $T_{1,0}$, we have 
    \begin{equation}
        \cI_{\theta}(x)= \sum_{\fp \in \cP} a_{\fp}(h(x),X_{0}(x),Z(x), \mu(x)) \fp(X_{0},Z,\overline{Z})(x) .
         \label{eq:CR.pseudohermitian-invariants}
    \end{equation}
     
    (ii) We have $\cI_{t \theta}(x)= t^{-w}\cI_{\theta}(x)$ for any $t>0$.
\end{definition}

Any local Riemannian invariant of $h_{\theta}$ is a local pseudohermitian 
invariant. However, the above notion of weight for pseudohermitian invariant is anisotropic with respect to $h_{\theta}$. For instance if we replace $\theta$ by 
$t\theta$ then $h_{\theta}$ is rescaled by $t$ on $T_{1,0}\oplus T_{0,1}$ and by $t^{2}$ on the vertical line bundle $\cN\otimes \C$.

On the other hand, as shown in~\cite[Prop.~2.3]{JL:ICRNCCRYP} by means of parallel translation 
  along parabolic geodesics any orthonormal frame $Z_{1}(a),\ldots,Z_{n}(a)$ of $T_{1,0}$ at a point $a\in M$ can be extended  into a 
  local frame $Z_{1},\ldots,Z_{n}$ of $T_{1,0}$ near $a$. Such a frame is called a \emph{special orthonormal frame}. 
  
  Furthermore, as also shown in~\cite[Prop.~2.3]{JL:ICRNCCRYP} any special orthonormal frame $Z_{1},\ldots,Z_{n}$ near $a$ allows us to 
  construct \emph{pseudohermitian normal coordinates} $x_{0},z^{1}=x^{1}+ix^{n+1},
  \ldots, z^{n}=x^{n}+ix_{2n}$ centered at $a$ in such way that in the notation 
  of~(\ref{eq:CR.asymptotics-vector-fields1})--(\ref{eq:CR.asymptotics-vector-fields2}) we have
  \begin{equation}
      X_{0}(0)_{(-2)}=\partial_{x^{0}}, \qquad Z_{j}(0)_{(-1)}=\partial_{z^{j}}+\frac{i}{2}\bar{z}^{j}\partial_{x^{0}}, \qquad \omega_{j\bar{k}}(0)=0.
       \label{eq:XZw.normal-coordinates}
  \end{equation}

  Set $Z_{j}=X_{j}-iX_{n+j}$,  where $X_{j}$ and $X_{n+j}$ are real vector fields. Then we have
$X_{j}(0)_{(-1)}=\partial_{x^{j}}-\frac{1}{2}x^{n+j}\partial_{x^{0}}$ and $X_{n+j}(0)_{(-1)}=\partial_{x^{n+j}}+\frac{1}{2}x^{j}\partial_{x^{0}}$. In 
particular, we have $X_{j}(0)=\partial_{x^{j}}$ for $j=0,\ldots,2n$. This implies that the affine change of variables $\psi_{0}$ to the privileged coordinates at $0$ is 
just the identity. Moreover, in the notation of~(\ref{eq:CR.model-vector-fields}) for $j=1,\ldots,n$ we have
\begin{equation}
    X_{j}^{(0)}=\partial_{x^{j}}-\frac{1}{2}x^{n+j}\partial_{x^{0}}, \qquad X_{n+j}^{(0)}=\partial_{x^{n+j}}+\frac{1}{2}x^{j}\partial_{x^{0}}.
     \label{eq:X.normal-coordinates}
\end{equation}
Incidentally, this shows that the matrix $(b_{jk}(0))$ is skew-symmetric, so its symmetric part vanishes, i.e., we have $\mu(0)=0$. 

\begin{proposition}
    Assume that on each pseudohermitian manifold $(M^{2n+1},\theta)$  there is the datum of a function $\cI_{\theta}\in C^{\infty}(M)$ such that 
    $\cI_{t \theta}(x)= t^{-w}\cI_{\theta}(x)$ for any $t>0$. Then the following are equivalent:\smallskip
    
    (i) $\cI_{\theta}(x)$ is a local pseudohermitian invariant;\smallskip
 
   (ii) There exists a finite family $(a_{\fp})_{\fp \in \cP}\subset \C$ such that, for any  pseudohermitian manifold $(M^{2n+1},\theta)$ and any point $a 
   \in M$, in any pseudohermitian normal coordinates centered at $a$ associated to any given special orthonormal frame $Z_{1},\ldots,Z_{n}$ of $T_{1,0}$ 
   near $a$, we have 
   \begin{equation}
       \cI_{\theta}(a)= \sum_{\fp \in \cP} a_{\fp}\fp(X_{0},Z,\overline{Z})(x)_{|x=0}.
%        \label{¥}
   \end{equation}
    
   (iii) $\cI_{\theta}(x)$ is a universal linear combination of complete tensorial contractions of covariant derivatives of the pseudohermitian 
   curvature tensor and  of the torsion tensor of the Tanaka-Webster connection.
\end{proposition}
\begin{proof}
The proof will consist in proving the implications $\text{(iii)} \Rightarrow \text{(i)}$, $\text{(i)} \Rightarrow \text{(ii)}$ and $\text{(ii)} 
\Rightarrow (iii)$. We shall prove them in this order. 
    
  First, let $Z_{1},\ldots,Z_{n}$ be a local frame of $T_{1,0}$ and let $\theta^{1},\ldots,\theta^{n}$ be the corresponding coframe of $T_{1,0}$. 
  Then it follows from~(\ref{eq:CR.TW-connection}) and~(\ref{eq:CR.TW-curvature}) that in local coordinates the components  $R_{j\bar{k}l\bar{m}}$ and 
  $A_{jk}$ of its curvature and 
  torsion tensors of the Tanaka-Webster connection with respect to the frame are universal expressions of the form~(\ref{eq:CR.pseudohermitian-invariants}). 
  Therefore, any linear combination of complete tensorial contractions of covariant derivatives of the curvature and torsion tensors yield a local pseudohermitian 
  invariant. This proves the implication $\text{(iii)} \Rightarrow \text{(i)}$. 
  
  Second, let $\cI_{\theta}(x)$ be a local pseudohermitian invariant. Then there 
  exists a finite family $(a_{\fp})_{\fp \in \cP}\subset C^{\infty}(\Omega)$ such that, in any 
    local coordinates equipped with a frame $Z_{1},\ldots,Z_{n}$ of $T_{1,0}$, we have 
\begin{equation}
       \cI_{\theta}(x)= \sum_{\fp \in \cP} a_{\fp}(h(x),X_{0}(x),Z(x),\mu(x)) \fp(X_{0},Z,\overline{Z})(x).
     \label{eq:CR.pseudohermitian-invariants-bis}
\end{equation}
 
  Let $a \in M$ and let us work in normal pseudohermitian coordinates centered at $a$ and associated to a special orthonormal  frame 
  $Z_{1},\ldots,Z_{n}$ of $T_{1,0}$. Since $Z_{1},\ldots,Z_{n}$ is an orthonormal frame we have $h_{j\bar{k}}=\delta_{j\bar{k}}$. Moreover, 
  by~(\ref{eq:XZw.normal-coordinates}) we have $X_{0}(0)=\partial_{x^{0}}$ and $Z_{j}(0)=\partial_{z^{j}}$, i.e., $X_{0}(0)=(\delta_{0}^{~k})$ and
  $Z(0)=(\delta_{j}^{~k}-i\delta_{n+j}^{\mbox{~}\mbox{~}\mbox{~}\mbox{~}\mbox{~}k})$. In addition, by~(\ref{eq:X.normal-coordinates}) we have $\mu(0)=0$. 
  Set 
  $a_{\fp}=a_{\fp}((\delta_{j\bar{k}}),(\delta_{0}^{~k}), (\delta_{j}^{~k}-i\delta_{n+j}^{\mbox{~}\mbox{~}\mbox{~}\mbox{~}\mbox{~}k}),0)$. 
  Then by~(\ref{eq:CR.pseudohermitian-invariants-bis}) we have  
   \begin{equation}
      \cI_{\theta}(a)=\sum_{\fp \in \cP} a_{\fp} \fp(X_{0},Z,\overline{Z})(0).
%       \label
  \end{equation}
  Since the  $a_{\fp}$'s are universal constants this shows that $\cI_{\theta}$ satisfies~(ii). The implication $\text{(i)} \Rightarrow 
  \text{(ii)}$ is thus proved. 
  
  It remains to prove the implication $\text{(ii)} \Rightarrow \text{(iii)}$. To this end assume that there exists a finite family $(a_{\fp})_{\fp 
  \in \cP}\subset \C$ such that, for any  pseudohermitian manifold $(M^{2n+1},\theta)$ and any point $a 
   \in M$, in any pseudohermitian normal coordinates centered at $a$ and associated to any given special orthonormal frame $Z_{1},\ldots,Z_{n}$ of $T_{1,0}$ 
   near $a$, we have 
   \begin{equation}
       \cI_{\theta}(a)= \sum_{\fp \in \cP} a_{\fp} \fp(X_{0},Z,\overline{Z})(x)_{|x=0}.
%        \label{¥}
   \end{equation}
  
  In the sequel by order of a differential operator we mean order in the Heisenberg calculus sense, and 
  by polynomial in partial or covariant derivatives of components of some tensors or forms we mean a polynomial in these quantities \emph{and} their complex 
  conjugates. Bearing this in mind, by~\cite[Prop.~2.5]{JL:ICRNCCRYP} in normal pseudohermitian coordinates associated to any given special orthonormal frame 
  $Z_{1},\ldots,Z_{n}$ of 
  $T_{1,0}$ with dual coframe $\theta^{1},\ldots,\theta^{n}$ the following holds:\smallskip
  
  (a) At $x=0$ the partial derivatives of order $\leq N$  of the components of the contact form $\theta$ are universal polynomials in partial 
  derivatives of order $\leq N$ of the components of the forms $\theta^{j}$;\smallskip 
  
  (b) At $x=0$ the partial derivatives of order $\leq N$ of the components of the forms $\theta^{j}$ are universal polynomials in partial 
  derivatives of order $\leq N$ of the components $\omega_{j\bar{k}}$ and $A_{jk}$ of the connection 1-form and torsion tensor of the 
Tanaka-Webster connection;\smallskip
  
  (c) At $x=0$ the partial derivatives of order $\leq N$ of the components $\omega_{j\bar{k}}$ are universal polynomials in partial 
  derivatives of order $\leq N$ of the components $R_{j\bar{k}l\bar{m}}$ and $A_{jk}$ of the pseudohermitian curvature tensor and torsion tensor 
  of the Tanaka-Webster connection.\smallskip
  
  It follows from this that at $x=0$ the partial derivatives of order $\leq N$  of the components of the vector fields $X_{0},Z_{1},\ldots,Z_{n}$ 
  are universal polynomials in partial derivatives of order $\leq N$ of the components $R_{j\bar{k}l\bar{m}}$ and $A_{jk}$ of the pseudohermitian curvature tensor 
  and torsion tensor of the Tanaka-Webster connection.. Therefore $\cI_{\theta}(0)$ is a universal polynomial 
  in partial derivatives at $x=0$ of these components.
  
  Next, by definition the pseudohermitian cuvarture tensor $R_{j\bar{k}l\bar{m}}$  is a section of the bundle 
  $\cT:=\Lambda^{1,0}\otimes \Lambda^{0,1}\otimes \Lambda^{1,0}\otimes  \Lambda^{0,1}$.  Let $\nabla^{\cT}$ be the lift of $\nabla$ to $\cT$, so that with respect to 
  the local  frame $\{\theta^{j_{1}}\otimes\theta^{\bar{j}_{2}}\otimes \theta^{j_{3}}\otimes \theta^{\bar{j}_{4}}\}$ of $\cT$ we have 
  \begin{equation}
      \nabla^{\cT}=d+\omega^{j_{1}}_{~k_{1}} \otimes 1\otimes 1\otimes 1 + 1\otimes  
      \omega^{\bar{j}_{2}}_{~\bar{k}_{2}} \otimes 1\otimes 1 +\ldots \,.
%       \label{}
  \end{equation}
 
  For $j=1,..,n$ set $Z_{j}=X_{j}-iX_{n+j}$ where $X_{j}=X_{j}^{~k}\frac{\partial}{\partial x^{k}}$ and $X_{n+j}=X_{n+j}^{~k}\frac{\partial}{\partial 
  x^{k}}$  are
  real-valued vector fields. An induction then shows that, for any ordered subset 
  $I=(i_{1},\ldots,i_{N})\subset \{0,\ldots,2n\}$, we have 
  \begin{equation}
      \nabla_{X_{i_{1}}}^{\cT}\cdots \nabla_{X_{i_{m}}}^{\cT}=X_{i_{1}}^{~j_{1}}\cdots X_{i_{N}}^{~j_{N}} \partial_{x^{j_{1}}} \cdots \partial_{x^{j_{N}}} +
      \sum_{|\alpha|\leq N-1}a_{I,\alpha}\partial_{x}^{\alpha},
%       \label{¥}
  \end{equation}
  where the components of $a_{I,\alpha}=(a^{j_{1}~\bar{j}_{2}~j_{3}~\bar{j}_{4}}_{~k_{1}~\bar{k}_{2}~k_{3}~\bar{k}_{4}})$ with respect to the frame 
  $\{\theta^{j_{1}}\otimes\theta^{\bar{j}_{2}}\otimes \theta^{j_{3}}\otimes \theta^{\bar{j}_{4}}\}$ are universal polynomials in the partial derivatives of order~$\leq 
  N-1$ of the components $X_{i_{l}}^{~j_{l}}$ and $\omega_{j\bar{k}}(X_{i_{l}})$. 
  
  We know that at $x=0$ the partial derivatives of order~$\leq 
  N-1$ of the components $X_{i_{l}}^{~j_{l}}$ and $\omega_{j\bar{k}}(X_{i_{l}})$ are universal polynomials in partial derivatives of order~$\leq N-1$ 
  of the curvature and torsion components $R_{j\bar{k}l\bar{m}}$ and $A_{jk}$. %of pseudohermitian curvature and torsion tensor of the Tanaka-Webster connection. 
  Moreover~(\ref{eq:XZw.normal-coordinates}) implies that
  $X_{i_{1}}^{~j_{1}}\cdots X_{i_{N}}^{~j_{N}} \partial_{x^{j_{1}}} \cdots \partial_{x^{j_{N}}}(0)=\partial_{x^{i_{1}}} \cdots \partial_{x^{i_{N}}}$. 
  Therefore, for any multi-order $\alpha$ in $ \N_{0}^{2n+1}$ such that $|\alpha|=N$, we have 
  \begin{equation}
     \partial_{x}^{\alpha} R_{j\bar{k}l\bar{m}}(0)= \left((\nabla_{X_{0}}^{\cT})^{\alpha_{0}}\cdots  
     (\nabla_{X_{2n}}^{\cT})^{\alpha_{2n}}R\right)_{j\bar{k}l\bar{m}}(0)+
       P_{\alpha}(R,\tau),
       \label{eq:PsH.partial-R}
  \end{equation}
  where $P_{\alpha}(R,\tau)$ is a universal polynomial in the partial derivatives at $x=0$ of order~$\leq N-1$ of the components of the pseudohermitian 
  cuvarture tensor and that of the torsion tensor.
  
%   $R_{j\bar{k}l\bar{m}}$ and $A_{jk}$.
%   $R_{j\bar{k}l\bar{m}}$ and $A_{jk}$ of the 
  
  On the other hand, as $Z_{1},\ldots,Z_{n}$ is an orthonormal frame we have
  \begin{equation}
      \theta([Z_{j},Z_{\bar{k}}])=-id\theta(Z_{j},Z_{\bar{k}})=-i\delta_{j\bar{k}}.
       \label{eq:PsH.bracket-1}
  \end{equation}
  Furthermore, from~(\ref{eq:CR.TW-connection}) we get 
  \begin{equation}
       \theta^{l}([Z_{j},Z_{\bar{k}}])=-d\theta^{l}(Z_{j},Z_{\bar{k}})=-\omega_{j}^{~l}(Z_{\bar{k}}).
       \label{eq:PsH.bracket-2}
  \end{equation}
  Together~(\ref{eq:PsH.bracket-1}) and~(\ref{eq:PsH.bracket-2}) show that 
  \begin{equation}
      [Z_{j},Z_{\bar{k}}]=-i\delta_{j\bar{k}}X_{0}-\omega_{j}^{~l}(Z_{\bar{k}})Z_{l}+\omega_{~\bar{k}}^{\bar{l}}(Z_{j})Z_{\bar{l}}.
%       \label{¥}
  \end{equation}
  Combining this with the fact that $[Z_{j},Z_{\bar{j}}]=2i[X_{j},X_{n+j}]$ we deduce that
  \begin{equation}
      X_{0}=\frac{1}{n}\sum_{j=1}^{n}\left\{ 2[X_{j},X_{n+j}]+i\omega_{j}^{~k}(Z_{\bar{j}})Z_{k}-i\omega_{~\bar{j}}^{\bar{k}}(Z_{j})Z_{\bar{k}}\right\}.
%       \label{¥}
  \end{equation}
  Thus, 
  \begin{equation}
      \nabla^{\cT}_{X_{0}}=\frac{1}{n}\sum_{j=1}^{n}\left\{2 \nabla^{\cT}_{[X_{j},X_{n+j}]}+ 
      i\omega_{j}^{~k}(Z_{\bar{j}}) \nabla^{\cT}_{Z_{k}}-i\omega_{~\bar{j}}^{\bar{k}}(Z_{j}) \nabla^{\cT}_{Z_{\bar{k}}}\right\}. 
       \label{eq:PsH.nablaX0a}
  \end{equation}
  
  Let $R^{\cT}$ be the pseudohermitian curvature of $\cT$. Its components with respect to the orthonormal
  frame $\{\theta^{j_{1}}\otimes\theta^{\bar{j}_{2}}\otimes \theta^{j_{3}}\otimes \theta^{\bar{j}_{4}}\}$ are
  \begin{equation}
      R^{j_{1}~\bar{j}_{2}~j_{3}~\bar{j}_{4}\mbox{~}\mbox{~}l\bar{m}}_{~k_{1}~\bar{k}_{2}~k_{3}~\bar{k}_{4}}= 
      -\overline{R_{~\bar{k}_{1}m\bar{l}}^{\bar{j}_{1}}}\otimes 1 \otimes 1\otimes 1 
      + 1\otimes R_{~\bar{k}_{2}l\bar{m}}^{\bar{j}_{2}}\otimes 1 \otimes 1+ \ldots \, .
%       \label{}
  \end{equation}
   As $R^{\cT}(X_{j},X_{n+j})= [\nabla^{\cT}_{X_{j}},\nabla^{\cT}_{X_{n+j}}] 
  -\nabla^{\cT}_{[X_{j},X_{n+j}]}$ it follows from~(\ref{eq:PsH.nablaX0a}) that $\nabla^{\cT}_{X_{0}}$ is equal to
  \begin{equation}
       \frac{1}{n}\sum_{j=1}^{n}\left\{2[\nabla^{\cT}_{X_{j}},\nabla^{\cT}_{X_{n+j}}]  + 
      i\omega_{j}^{~k}(Z_{\bar{j}}) \nabla^{\cT}_{Z_{k}}-i\omega_{~\bar{j}}^{\bar{k}}(Z_{j}) \nabla^{\cT}_{Z_{\bar{k}}} - 2R^{\cT}(X_{j},X_{n+j})\right\}. 
%        \label{¥}
  \end{equation}
  By combining this with~(\ref{eq:PsH.partial-R}) we then can show that, for any multi-order $\alpha$ in $\N_{0}^{2n+1}$ such that $|\alpha|=N$, we 
  have
  \begin{equation}
  \partial_{x}^{\alpha} R_{j\bar{k}l\bar{m}}(0)=  \left(\left(\frac{2}{n}\sum_{j=1}^{n}[\nabla^{\cT}_{X_{j}},\nabla^{\cT}_{X_{n+j}}]\right)^{\alpha_{0}} \!\!\!
      (\nabla_{X_{1}}^{\cT})^{\alpha_{1}}\cdots  
     (\nabla_{X_{2n}}^{\cT})^{\alpha_{2n}}R\right)_{\!\!\!j\bar{k}l\bar{m}}\!\!\!\!\!\!\!\!(0)+ P_{\alpha}(R,\tau),
       \label{eq:PsH.partial-R2}
  \end{equation}
  where $P_{\alpha}(R,\tau)$ is a universal polynomial in the partial derivatives at $x=0$ of order~$\leq N-1$ of the components of the components of the pseudohermitian 
  curvature tensor and that of the torsion tensor. 
%   $R_{j\bar{k}l\bar{m}}$ and $A_{jk}$. 
  
  The tensor $A_{jk}$ is a section of the bundle $\cT':=\Lambda^{1,0}\otimes \Lambda^{1,0}$. If we let $\nabla^{\cT'}$ denote the lift 
  to $\cT'$ of the Tanaka-Wesbter connection then, in the same way as above, we can show that, for any multi-order $\alpha \in \N_{0}^{2n+1}$ such that $|\alpha|=N$, 
  we have
  \begin{equation}
      \partial_{x}^{\alpha} A_{jk}(0)= \left(\left(\frac{2}{n}\sum_{j=1}^{n}[\nabla^{\cT'}_{X_{j}},\nabla^{\cT'}_{X_{n+j}}]\right)^{\alpha_{0}} \!\!\!
      (\nabla_{X_{1}}^{\cT'})^{\alpha_{1}}\cdots  
     (\nabla_{X_{2n}}^{\cT'})^{\alpha_{2n}}A\right)_{\!\!\!jk}\!\!\!\!(0)+ Q_{\alpha}(R,\tau),
       \label{eq:PsH.partial-A}
  \end{equation}
  where $Q_{\alpha}(R,\tau)$ is a universal polynomial in the partial derivatives at $x=0$ of order~$\leq N-1$ of the components of the pseudohermitian 
  curvature tensor and that of the torsion tensor. By combining~(\ref{eq:PsH.partial-R2}) and~(\ref{eq:PsH.partial-A}) and by arguing by induction we 
  then can see
  that the value at $x=0$ of any partial derivative of order~$N$  of these components agrees with the value of an universal polynomial in their \emph{covariant 
  derivatives} of order~$\leq N$ with respect to the vector fields $X_{1},\ldots,X_{2n}$.
  
  It follows from all this that $\cI_{\theta}(0)$ agrees with value at $x=0$ of a universal polynomial in covariant derivatives with respect to the vector fields 
  $Z_{1},\ldots,Z_{n}, Z_{\bar{1}}, \ldots, Z_{\bar{n}}$ of the components of the pseudohermitian 
  curvature tensor and that of the torsion tensor. 
  We then can make use of $U(n)$-invariant theory as in~\cite[pp.~380--382]{BGS:HECRM} to deduce that $\cI_{\theta}(x)$ is a linear combination 
  of complete tensorial contractions of covariant derivatives of these tensors, 
  i.e.,~$\cI_{\theta}(x)$ satisfies (iii). This proves that (ii) implies (iii). The proof is thus achieved. 
\end{proof}

\subsection{Pseudohermitian invariants $\mathbf{\Psi_{H}}$DOs}
We define homogeneous symbols on $\Omega\times\Rtn$ as follows. 

\begin{definition}
    $S_{m}(\Omega\times \R^{2n+1})$, $m\in \C$, consists of be functions $a(h,X_{0},Z,\xi)$ in 
$C^{\infty}(\Omega\times (\R^{2n+1}\setminus 0))$ such that  $a(\theta,Z,t\xi)=t^{m}a(\theta,Z,\xi)$ $\forall t>0$. 
\end{definition}

In addition, recall that if $Z_{1},\ldots,Z_{n}$ is a local frame of $T_{1,0}$ then its associated $H$-frame is the frame $X_{0},\ldots,X_{2n}$ of $TM$ 
such that $Z_{j}=X_{j}-iX_{n+j}$ for $j=1,\ldots,n$. 

\begin{definition}
A pseudohermitian invariant \psivdo\ of order $m$ and weight $w$ is the datum on each pseudohermitian manifold  $(M^{2n+1},\theta)$ of an operator 
$P_{\theta}$ in $\pvdo^{m}(M)$ such that:\smallskip

(i) For $j=0,1,\ldots$ there exists a finite family $(a_{j\fp})_{\fp \in \cP}\subset S_{m-j}(\Omega\times 
\R^{2n+1})$ such that, in any local coordinates equipped with the $H$-frame associated to a frame $Z_{1},\ldots,Z_{n}$ of $T_{1,0}$, the operator 
$P_{\theta}$ has symbol $p_{\theta}\sim \sum p_{\theta,m-j}$ with 
\begin{equation}
    p_{\theta,m-j}(x,\xi)= \sum_{\fp \in \cP} \fp(X_{0},Z,\overline{Z})(x) a_{j\fp}(h(x),X_{0}(x),Z(x),\mu(x),\xi).
     \label{eq:CR.CR-invariant-symbol}
\end{equation}

(ii) For any $t>0$ we have $P_{t\theta}=t^{-w}P_{\theta}$ modulo $\Psi^{-\infty}(M)$.\smallskip

\noindent In addition, we will say that $P_{\theta}$ is \emph{admissible} if in~(\ref{eq:CR.CR-invariant-symbol})  
we can take $a_{0\fp}(h,X_{0},Z,\mu,\xi)$ to be zero for $\fp\neq 1$.
\end{definition}

Before proving the analogues in pseudohermitian geometry of Proposition~\ref{prop:GJMS.Riemannian-invariance-properties1} 
and Proposition~\ref{prop:GJMS.Riemannian-invariance-properties2}
we need to recall some results about the symbolic calculus for \psivdos.

Given a matrix $b=(b_{jk})\in M_{2n}(\R)$ we can endow $\R^{2n+1}$ with a structure of 2-step nilpotent Lie group by means of the product,
\begin{equation}
    x.y=(x^{0}+y^{0}+b_{kj}x^{j}y^{k},x^{1}+y^{1},\ldots,x^{2n}+y^{2n}).
    \label{eq:CR.product-Rtn}
\end{equation}
It follows from~\cite{BG:CHM} that 
under the inverse Fourier transform the convolution for distributions with respect to this group gives rise to a product for homogeneous 
symbols, 
\begin{equation}
    *^{b}:S_{m_{1}}(\Rtn)\times S_{m_{2}}(\Rtn)\longrightarrow S_{m_{1}+m_{2}}(\Rtn).
     \label{eq:CR.product-Homogeneous-symbols-Rtn}
\end{equation}
Furthermore, this product depends smoothly on $b$.

Let $U\subset \R^{2n+1}$ be an open of local coordinates equipped with a $H$-frame $X_{0},\ldots,X_{d}$. We let 
$\eta^{0},\ldots,\eta^{2n}$ be the dual coframe and we set $X_{j}=X_{j}^{~k}\partial_{x^{k}}$ and $\eta^{j}=\eta_{j}^{~k}dx^{k}$. 

For any $a\in U$ we let $\psi_{a}$ be the affine change of variables to the privileged coordinates centered $a$, 
and we let $X_{0}^{(a)},\ldots,X_{d}^{(a)}$ be the model vector fields as defined 
in~(\ref{eq:CR.model-vector-fields}), that is, we have $X_{0}^{(a)}=\partial_{x^{0}}$ and $X_{j}^{(a)}=\partial_{x^{j}}+b_{jk}(a)x^{k}\partial_{x^{0}}$ where 
$b_{jk}(a):=L_{jk}(a)+\mu_{jk}(a)$. As alluded to before the linear span of the vector fields $X_{0}^{(a)},\ldots,X_{d}^{(a)}$ is a 2-step nilpotent 
Lie algebra whose corresponding Lie group is isomorphic to the tangent group $G_{a}M$ and can be realized as $\R^{2n}$ equipped with the group 
law~(\ref{eq:CR.product-Rtn}) with $b_{jk}=b_{jk}(a)$. 
Since the product~(\ref{eq:CR.product-Homogeneous-symbols-Rtn}) 
for homogenous symbols on $\R^{2n}$ depends smoothly on $b$ and $b(a):=(b_{jk}(a))$ depends smoothly on $a$, 
we get the following product for homogeneous symbols on $\URtn$, 
\begin{gather}
    *:S_{m_{1}}(\URtn)\times S_{m_{2}}(\URtn) \longrightarrow S_{m_{1}+m_{2}}(\URtn),\\
     p_{m_{1}}*p_{m_{2}}(a,\xi):=[p_{m_{1}}(a,.)*^{b(a)}p_{m_{2}}(a,.)](\xi) \qquad \forall p_{j} \in S_{m_{j}}(\URtn).
     \label{eq:CR.product-symbols-URn}
 \end{gather}

 We also can define a product for homogeneous symbols on $\Omega\times \Rtn$ as follows. For any $(h,\mu)$ in 
 $M_{n}(\C)_{+}\times S_{2n}(\R)$ we let
 \begin{equation}
     b(h,\mu):=\frac{1}{2}L+\mu, \qquad L=  \frac{1}{2}\left(
    \begin{array}{cc}
        \Im h & -\Re h  \\
        \Re h & \Im h
    \end{array}\right). 
      \label{eq:CR.b(h,mu)}
 \end{equation}
 As $b(h,\mu)$ depends smoothly on $(h,\mu)$ we obtain the bilinear product, 
 \begin{equation}
    *:S_{m_{1}}(\Omega\times \Rtn)\times S_{m_{2}}(\Omega\times \Rtn) \longrightarrow S_{m_{1}+m_{2}}(\Omega\times\Rtn),
 \end{equation}
 such that, for any symbols $p_{1} \in S_{m_{1}}(\Omega\times\Rtn)$ and $p_{2} \in S_{m_{2}}(\Omega\times\Rtn)$, on $\Omega\times \Rtn$ we have
 \begin{equation}
     p_{m_{1}}*p_{m_{2}}(h,X_{0},Z,\mu,\xi)=  [p_{m_{1}}(h,X_{0},Z,\mu,.)*^{b(h,\mu)}p_{m_{2}}(h,X_{0},Z,\mu,.)](\xi). 
    \label{eq:CR.product-symbols-OmegaRtn}
 \end{equation}

 Observe that it follows from~(\ref{eq:CR.L-h}) and from the very definition of $\mu(a)$ that we have $b(x)=\frac{1}{2}L(x)+\mu(x)=b(h(x),\mu(x))$. Therefore, 
 we see that, for any symbols $p_{1} \in S_{m_{1}}(\Omega\times\Rtn)$ and $p_{2} \in S_{m_{2}}(\Omega\times\Rtn)$,
  on $\URtn$ we have  
\begin{multline}
    [p_{m_{1}}(h(x),X_{0}(x),Z(x),\mu(x),\xi)]*[p_{m_{2}}(h(x),X_{0}(x),Z(x),\mu(x),\xi)] \\ = 
    (p_{m_{1}}*p_{m_{2}})(h(x),X_{0}(x),Z(x),\mu(x),\xi),
     \label{eq:CR.comptatibility-products-symbols}
\end{multline}
 where the product $*$ on the l.h.s.~is that for homogenous symbols on $\URtn$ and the other product $*$ is that for 
 homogeneous symbols on $\Omega\times\Rtn$. 
 
Next, let $\sigma_{j}(x,\xi)=X_{j}^{~k}(x)\xi_{k}$ be the classical symbol of $\frac{1}{i}X_{j}$. Then the symbol of $\psi_{a*}X_{j}$ is 
$\psi_{a*}\sigma_{j}(x,\xi):=X_{j}^{~k}(\psi_{a}(x))\eta^{l}_{~k}(a)\xi_{l}$. We set $\sigma(x,\xi)=(\sigma_{0}(x,\xi),\ldots,\sigma_{2n}(x,\xi))$. 
Similarly, we let $\sigma_{j}^{(a)}(x,\xi)$ be the classical symbol of $\frac{1}{i}X_{j}^{(a)}$ and we set 
$\sigma^{(a)}(x,\xi)=(\sigma_{0}(x,\xi),\ldots,\sigma_{2n}(x,\xi))$. Notice that $\sigma_{0}^{(a)}(x,\xi)=\xi_{0}$, while for $j=1,\ldots,2n$ we 
have $\sigma_{j}^{(a)}(x,\xi)=\xi_{j}+ b_{jk}(a)x^{k}\xi_{0}$. For any multi-order $\beta\in \N_{0}^{2n}$ we can write 
\begin{equation}
    [\psi_{a*}\sigma(x,\xi)-\sigma^{(a)}(x,\xi)]^{\beta}=\sum_{|\gamma|=|\beta|}e_{\beta\gamma}(a,x)\sigma^{(a)}(x,\xi)^{\gamma},
%     \label{¥}
\end{equation}
where the coefficients $e_{\beta\gamma}(a,x)$ are smooth functions on $\URtn$. We then let 
$h_{\alpha\beta\gamma\delta}(a)$ be the smooth function on $U$ given by
\begin{equation}
    h_{\alpha\beta\gamma\delta}(a) =\frac{1}{\alpha!\beta!\delta!}\partial_{x}^{\delta}[\psi^{`-1}_{a}(x)^{\alpha}e_{\beta\gamma}(a,x)]_{|x=0}.
%     \label{¥}
\end{equation}

\begin{proposition}[{\cite[Thm.~14.7]{BG:CHM}}]\label{prop:CR.product-psivdos}
Let $P\in \pvdo^{m}(U)$ have symbol $p\sim \sum p_{m-j}$, let $Q\in \pvdo^{m'}(U)$ have symbol $q\sim \sum q_{m'-j}$, and assume that 
$P$ or $Q$ is properly supported. Then $PQ$ belongs to $\pvdo^{m+m'}(U)$ and has symbol $r\sim \sum r_{m+m'-j}$ with
\begin{equation}
    r_{m+m'-j}={\sum}_{(j)} h_{\alpha\beta\gamma\delta}(D_{\xi}^{\delta}p_{m-k})*(\xi^{\gamma}\partial_{x}^{\alpha}D_{\xi}^{\beta}q_{m'-l}),
   \label{eq:CR.symbol-product}
\end{equation}
where ${\sum}_{(j)}$ denotes the summation over all indices such that $|\gamma|=|\beta|$ and 
$|\beta|+|\alpha|\leq \brak\delta +\brak\beta-\brak\gamma=j-k-l$. 
\end{proposition}

Bearing all this in mind we are now ready to prove: 

\begin{proposition}\label{prop:CR.pseudohermitian-invariant-PsiHDOs1}
     Let $P_{\theta}$ be a pseudohermitian invariant \psivdo\ of order $m$ and weight~$w$, let $Q_{\theta}$ be a pseudohermitian invariant \psivdo\ 
     of order $m'$ and weight $w'$, and assume that $P_{\theta}$ or $Q_{\theta}$ is 
    uniformly properly supported. Then:\smallskip
    
    1) $P_{\theta}Q_{\theta}$ is a pseudohermitian invariant \psivdo\ of order $m+m'$ and weight $w+w'$.\smallskip
    
    2) If $P_{\theta}$ and $Q_{\theta}$ are admissible, then $P_{\theta}Q_{\theta}$ is admissible as well.
 \end{proposition}
\begin{proof}
    Since $P_{\theta}$ and $Q_{\theta}$ are pseudohermitian invariant \psivdos, for $j=0,1,\ldots$ there exist finite families $(a_{j\fp})_{\fp \in \cP}\subset 
    S_{m-j}(\Omega\times \Rtn)$ and $(b_{j\fq})_{\fq \in \cP}\subset S_{m'-j}(\Omega\times \Rtn)$ such that, in any given local coordinates equipped 
    with the $H$-frame associated to a frame $Z_{1},\ldots,Z_{n}$ of $T_{1,0}$, the respective symbols of $P_{\theta}$ and $Q_{\theta}$ 
    are $p\sim \sum p_{\theta,m-j}$ and $q\sim \sum q_{\theta,m'-j}$ with
    \begin{gather}
         p_{\theta,m-j}(x,\xi)= \sum_{\fp \in \cP} \fp(X_{0},Z,\overline{Z})(x) a_{j\fp}(h(x),X_{0}(x),Z(x),\mu(x),\xi),\\
         q_{\theta,m'-j}(x,\xi)= \sum_{\fq \in \cP}\fq(X_{0},Z,\overline{Z})(x) b_{j\fq}(h(x),X_{0}(x),Z(x),\mu(x),\xi) .
%         \label{¥}
    \end{gather}
    Therefore, by Proposition~\ref{prop:CR.product-psivdos} the operator $P_{\theta}Q_{\theta}$ has symbol $r\sim \sum r_{m+m'-j}$, where $r_{m+m'-j}(x,\xi)$ is equal to
\begin{multline}
   \sum_{\fp, \fq \in \cP}{\sum}_{(j)} h_{\alpha\beta\gamma\delta}(x) \fp(X_{0},Z,\overline{Z})(x) 
    [D_{\xi}^{\delta}a_{k\fp}(h(x),X_{0}(x),Z(x),\mu(x),\xi)]* \\ [\xi^{\gamma}D_{\xi}^{\beta}
    \partial_{x}^{\alpha}(\fq(X_{0},Z,\overline{Z})(x) b_{l\fq}(h(x),X_{0}(x),Z(x),\mu(x),\xi))].
     \label{eq:CR.symbol-product-invariant0}
\end{multline}
    
Notice that,  given a multi-order $\alpha \in \N_{0}^{2n+1}$, for any monomial $\fq \in \cP$ and any symbol $b\in S_{m}(\Omega\times\Rtn)$ there exists 
a universal finite family $(C_{\alpha \tilde{\fq}}(\fq, b))_{\tilde{\fq} \in \cP}$ contained in $S_{m}(\Omega\times\Rtn)$ such that 
\begin{multline}
    \partial_{x}^{\alpha}[\fq(X_{0},Z,\overline{Z})(x) b(h(x),X_{0}(x),Z(x),\mu(x),\xi)] =  \\
    \sum_{\tilde{\fq} \in \cP}\tilde{\fq}(X_{0},Z,\overline{Z})(x) C_{\alpha \tilde{\fq}}(\fq, b)(h(x),X_{0}(x),Z(x),\mu(x),\xi). 
     \label{eq:CR.Leibniz-formula}
\end{multline}
In addition, it follows from the very definition of the function $h_{\alpha\beta\gamma\delta}(x)$ that there exists a universal finite family 
$(h_{\alpha \beta \gamma \delta \fr})_{\fr \in \cP}\subset C^{\infty}(\Omega)$ such that %we have %for any $x \in U$, %we have 
\begin{equation}
    h_{\alpha \beta \gamma \delta}(x)=\sum_{\fr \in \cP} h_{\alpha \beta \gamma \delta \fr}(h(x),X_{0}(x),Z(x),\mu(x)) \fr(X_{0},Z,\overline{Z})(x).
%     \qquad \forall x \in U.
     \label{eq:CR.invariant-form-halbetgadel}
\end{equation}

Now, by combining (\ref{eq:CR.comptatibility-products-symbols}), 
(\ref{eq:CR.Leibniz-formula}),~(\ref{eq:CR.symbol-product-invariant0}) and~(\ref{eq:CR.invariant-form-halbetgadel}) together we deduce that 
\begin{equation}
     r_{m+m'-j}(x,\xi)= \sum_{\fs \in \cP} \fs(X_{0},Z,\overline{Z})(x) c_{j\fs}(h(x),X_{0}(x),Z(x),\mu(x),\xi),
%     \label{¥}
\end{equation}
where $(c_{j\fs})_{\fs \in \cP}$ is the finite family with values in  $S_{m+m'-j}(\Omega\times\Rtn)$ given by  
\begin{equation}
    c_{j\fs}
    =\sum_{\substack{\fp, \fq, \tilde{\fq}, \fr \in \cP\\ \fp\tilde{\fq}\fr=\fs}} {\sum}_{(j)} h_{\alpha \beta \gamma \delta \fr} 
    [D_{\xi}^{\delta}a_{k\fp}]*[\xi^{\gamma}D_{\xi}^{\beta} C_{\alpha \tilde{\fq}}(\fq, b_{l\fq})] .
     \label{eq:CR.invariant-symbol-product}
\end{equation}
Since the family $(c_{j\fs})_{\fs \in \cP}$ is independent of the choice of the local coordinates and of the local frame $Z_{1},\ldots,Z_{n}$ this 
proves that $P_{\theta}Q_{\theta}$ is pseudohermitian invariant. Furthermore, as for any $t>0$ we have 
$P_{t\theta}Q_{t\theta}=t^{-(w+w')}P_{\theta}Q_{\theta}$ modulo $\Psi^{-\infty}(M)$, we see that $P_{\theta}Q_{\theta}$ is a pseudohermitian 
invariant \psivdo\ of weight $w$. 

Finally, assume further that $P_{\theta}$ and $Q_{\theta}$ are admissible, that is, there exist symbols $a_{m}\in S_{m}(\ORn)$ and $b_{m'}\in S_{m'}(\ORn)$ such that, in 
any given local coordinates equipped with the $H$-frame associated to a frame $Z_{1},\ldots,Z_{n}$, the principal symbol of $P_{\theta}$ is 
$p_{m}(x,\xi)=a_{m}(h(x),X_{0}(x),Z(x),\mu(x),\xi)$ and the principal symbol of $Q_{\theta}$ is $q_{m'}(x,\xi)=b_{m'}(h(x),X_{0}(x),Z(x),\mu(x),\xi)$.
It then follows from Proposition~\ref{prop:CR.product-psivdos} and~(\ref{eq:CR.comptatibility-products-symbols}) 
that in these local coordinates the  principal symbol of $P_{\theta}Q_{\theta}$ is equal to
\begin{multline*}
    p_{m}*q_{m'}(x,\xi)=[a_{m}(h(x),X_{0}(x),Z(x),\mu(x),\xi)]*[b_{m'}(h(x),X_{0}(x),Z(x),\mu(x),\xi)] \\ =
    [a_{m}*b_{m'}](h(x),X_{0}(x),Z(x),\mu(x),\xi). 
%     \label
\end{multline*}
Since the symbol $a_{m}*b_{m'}\in S_{m+m'}(\ORn)$ does not depend on the choices of the local coordinates and of the local frame $Z_{1},\ldots,Z_{n}$, 
this shows that $P_{\theta}Q_{\theta}$ is admissible. The proof  is now complete.
\end{proof}

In order to deal with parametrices of pseudohermitian invariant \psivdos\ we need the following lemma. 

\begin{lemma}\label{lem:Pseudohermitian.onto-lemma}
 Let $(h,\mathcal{X}_{0},\mathcal{Z},\mu)\in \Omega$. Then we can endow $\Rtn$ with a pseudohermitian structure and a global  frame $Z_{1},\ldots,Z_{n}$ 
 of $T_{1,0}$ with  respect to which we have $(h(0),X_{0}(0),Z(0),\mu(0))=(h,\mathcal{X}_{0},\mathcal{Z},\mu)$. 
\end{lemma}
\begin{proof}
Let $L=(L_{jk})\in M_{2n}(\R)$ be the skew-symmetric matrix given by~(\ref{eq:CR.b(h,mu)}), and let us endow $\Rtn$ with the group law~(\ref{eq:CR.product-Rtn}) 
corresponding to $b:=b(h,\mu)=\mu+\frac{1}{2}L$. Set $X_{0}^{(0)}=\partial_{x^{0}}$ and 
$X_{j}^{(0)}=\partial_{x^{j}}+b_{jk}x^{k}\partial_{x^{0}}$, $j=1,..,2n$. Let $H^{(0)}\subset T\Rtn$ be the hyperplane bundle spanned by the vector fields 
$X_{1}^{(0)},\ldots,X_{2n}^{(0)}$ and let us endow it with the almost complex structure $J^{(0)}\in C^{\infty}(\Rtn,\End H)$ such that for $j=1,\ldots,n$ we have 
$J^{(0)}X_{j}^{(0)}=X_{n+j}^{(0)}$ and $J^{(0)}X_{n+j}^{(0)}=-X_{j}^{(0)}$. 

Observe that the subbundle $T_{1,0}^{(0)}:=\ker(J^{(0)}+i)$ is spanned by the vector fields $Z_{j}^{(0)}:=X_{j}^{(0)}-iX_{n+j}^{(0)}$, $j=1,..,n$. 
By~(\ref{eq:CR.Ljk-bjk}) we have 
\begin{equation}
    [X_{j}^{(0)},X_{k}^{(0)}]=(b_{kj}-b_{jk})X_{0}^{(0)}=L_{kj}X_{0}^{(0)}.
     \label{eq:CR.Xj(0)Xk(0)}
\end{equation}
Moreover, as the definition~(\ref{eq:CR.b(h,mu)}) of $L$ implies that it satisfies~(\ref{eq:CR.L-h0}), we get 
\begin{equation}
    [Z_{j}^{(0)},Z_{k}^{(0)}]=[(L_{jk}-L_{n+jn+k})-i(L_{jn+k}+L_{n+jk})]X_{0}^{(0)}=0.
%     \label{¥}
\end{equation}
This implies that $T_{1,0}^{(0)}$ is integrable in Fr\"obenius' sense, so $(H^{(0)},J^{(0)})$ defines a CR structure on $\Rtn$. 

Set $\theta^{(0)}=dx^{0}-b_{jk}x^{k}dx^{j}$. Then we have $\theta^{(0)}(X_{0})=1$ and for $j=1,\ldots,2n$, we have $\theta^{(0)}(X_{j})=0$, 
so $\theta^{(0)}$ is a non-vanishing 1-form 
annihilating $H^{(0)}$. Moreover, it follows from~(\ref{eq:CR.b(h,mu)}) and~(\ref{eq:CR.Xj(0)Xk(0)}) that we have 
\begin{equation}
    i\theta^{(0)}([Z_{j}^{(0)},Z_{\bar{k}}^{(0)}])=(L_{jk}+L_{n+jn+k})+i(-L_{jn+k}+L_{n+jk})=h_{j\bar{k}}.
     \label{eq:GJMS.levi-form-theta0}
\end{equation}
Since $h$ is positive definite this shows that the Levi form associated to $\theta^{(0)}$ is positive definite everywhere. Therefore, the CR structure of $\Rtn$ is 
strictly pseudoconvex and $\theta^{(0)}$ is a pseudohermitian contact form. In addition, for $j=0,\ldots,2n$ we have $[X_{0}^{(0)},X_{j}^{(0)}]=0$, so we have 
$\iota_{X_{0}}d\theta^{(0)}(X_{j})=-\theta^{(0)}([X_{0},X_{j}])=0$. As we know that $\theta^{(0)}(X_{0})=1$ it follows that 
$X_{0}$ is the Reeb field of the contact form $\theta^{(0)}$. 

Next, let us write $\mathcal{X}_{0}=(\mathcal{X}_{0}^{~k})$ and $\mathcal{Z}=(\mathcal{Z}_{j}^{~k})$, where $\mathcal{X}_{0}$ and $\mathcal{Z}$ are 
the 2nd and 3rd components of our initial element $(h,\mathcal{X}_{0},\mathcal{Z},\mu)$ in $\Omega$. 
Set $\mathcal{Z}_{j}^{~k}=\mathcal{X}_{j}^{~k}-i\mathcal{X}_{n+j}^{~k}$ with 
$\mathcal{X}_{j}^{~k}$ and $\mathcal{X}_{n+j}^{~k}$ in $\R$, and let $\psi$ be the unique linear change of variables such that for $j=0,\ldots,2n$ 
the tangent map $\psi'(0):T_{0}\Rtn \rightarrow 
T_{0}\Rtn$ maps $\mathcal{X}_{j}^{~k}\partial_{x^{k}}$ to $\partial_{x^{j}}$. Set $H=\psi^{*}H^{(0)}$ and $J=\psi^{*}J^{^{(0)}}$. Then $(H,J)$ 
defines a strictly pseudoconvex CR structure on $\Rtn$ with respect to which $\theta:=\psi^{*}\theta^{(0)}$ is a pseudohermitian contact form with
Reeb field $X_{0}:=\psi^{*}X_{0}^{(0)}$. Moreover, as we have $X^{(0)}=\partial_{x^{0}}$ we see that 
$X_{0}(0)=\psi'(0)^{-1}(\partial_{x^{0}})=\mathcal{X}_{0}$. 

The corresponding bundle of $(1,0)$-vectors is $T_{1,0}:=\psi^{*}T_{1,0}$. A global frame for this bundle is provided by the vector fields 
$Z_{j}:=\psi^{*}Z_{j}^{(0)}$. Moreover, it follows from~(\ref{eq:GJMS.levi-form-theta0}) that with respect to this frame 
the matrix of the Levi form associated to $\theta$ is $(h_{j\bar{k}})$. In particular, we have $h(0)=h$. In 
addition, as $Z_{j}^{(0)}(0)=X_{j}^{(0)}(0)-iX_{n+j}^{(0)}=\partial_{x^{j}}-i\partial_{x^{n+j}}$ we also see that 
$Z_{j}(0)=\psi'(0)^{-1}(\partial_{x^{j}})-i\psi'(0)^{-1}(\partial_{x^{n+j}})=\mathcal{X}_{j}-i\mathcal{X}_{n+j}=\mathcal{Z}_{j}$. Thus 
$Z(0)=\mathcal{Z}$. 

In order to complete the proof it remains to check that $\mu(0)=\mu$. For $j=1,\ldots,2n$ set $X_{j}=\psi^{*}X_{j}^{(0)}$. Then 
$X_{0},\ldots,X_{2n}$ is a global 
$H$-frame of $\Rtn$. Moreover, as $\psi$ is a linear map and for $j=0,\ldots,2n$ we have $\psi_{*}X_{j}(0)=X_{j}^{(0)}(0)=\partial_{x^{j}}$, 
we see that $\psi$ is the affine 
change of variables to the privileged coordinates centered at the origin. In addition, since we have $\psi_{*}X_{j}=X_{j}^{(0)}$ and the vector 
fields $X_{j}^{(0)}$ are homogeneous, we deduce that $X_{j}^{(0)}$ is the model vector field of $X_{j}$ in the sense 
of~(\ref{eq:CR.asymptotics-vector-fields1})--(\ref{eq:CR.model-vector-fields}). As we have 
$X_{j}^{(0)}= \partial_{x^{j}}+b_{jk}x^{k}\partial_{x^{0}}$ we see that $b(0)=(b_{jk})=b(h,\mu)$. Since by definition $\mu(0)$ is the symmetric part of 
$b(0)$ and $b(h,\mu)$ has $\mu$ as symmetric part, it follows that $\mu(0)=\mu$ as desired. The proof is thus achieved.
\end{proof}

\begin{proposition}\label{prop:CR.pseudohermitian-invariant-PsiHDOs2}
     Let $P_{\theta}$ be a pseudohermitian invariant \psivdo\ of order $m$ and weight~$w$ such that $P_{\theta}$ is admissible and 
     its principal symbol is invertible in the Heisenberg calculus sense. 
    For each pseudohermitian manifold $(M^{2n+1},\theta)$ let $Q_{\theta}$ be a parametrix for $P_{\theta}$ in $\pvdo^{-m}(M)$. Then $Q_{\theta}$ is a 
    pseudohermitian invariant \psivdo\ of order $-m$ and weight $-w$.
\end{proposition}
\begin{proof}
 First, as $P_{\theta}$ is an admissible pseudohermitian invariant \psivdo\ there exists a symbol $a_{m}\in S_{m}(\Omega\times\Rtn)$ such that, 
 in any local coordinates equipped with the $H$-frame associated to a frame $Z_{1},\ldots,Z_{n}$ of $T_{1,0}$, the principal symbol of $P_{\theta}$ in these 
 local coordinates is $p_{\theta,m}(x,\xi):=a_{m}(h(x),X_{0}(x),Z(x),\mu(x),\xi)$. The fact that the principal symbol of $P_{\theta}$ is invertible in the 
 Heisenberg calculus sense means that $p_{\theta,m}$ is invertible with respect to the product~(\ref{eq:CR.product-symbols-URn}). 
 Therefore, we see that, for any local coordinates equipped 
with a frame $Z_{1},\ldots,Z_{n}$ of $T_{1,0}$ and for any $x$ in their range, the symbol $a_{m}(h(x),X_{0}(x),Z(x),\mu(x),.)$ is invertible with 
respect to the product $*^{b(x)}=*^{b(h(x),\mu(x))}$. We then can make use of Lemma~\ref{lem:Pseudohermitian.onto-lemma} to conclude that for any 
$(h,X_{0},Z,\mu)\in \Omega$ the symbol $a_{m}(h,X,Z,\mu,.)$ 
is invertible with respect to the product $*^{b(h,\mu)}$. Thus,  for any $(h,X_{0},Z,\mu)\in \Omega$ there exists a symbol $b_{-m}^{(h,X_{0},Z,\mu)}(\xi)$ in 
$S_{-m}(\Rtn)$ such that 
\begin{multline}
    a_{m}(h,X_{0},Z,\mu,.)*^{b(h,\mu)}b_{-m}(h,X_{0},Z,\mu)=\\ b_{-m}(h,X_{0},Z,\mu)*^{b(h,\mu)}a_{m}(h,X_{0},Z,\mu,.)=1. 
%     \label
\end{multline}

Since $a_{m}(h,X_{0},Z,\mu,.)$ depends smoothly on $((h,X_{0},Z,\mu)$, it follows from~\cite[Prop.~3.3.22]{Po:MAMS1} that 
$b_{-m}^{(h,X_{0},Z,\mu)}$ depends smoothly on $(h,X_{0},Z,\mu)$ as well. Therefore, we define a symbol 
$b_{-m}\in S_{-m}(\Omega\times \Rtn)$ by letting 
\begin{equation}
    b_{-m}(h,X_{0},Z,\mu,\xi):=b_{-m}^{(h,X_{0},Z,\mu)}(\xi) \qquad \forall(h,X_{0},Z,\mu,\xi)\in \Omega\times \Rtn.
%     \label
\end{equation}
In view of the definition of the product~(\ref{eq:CR.product-symbols-OmegaRtn}) 
we have $a_{m}*b_{-m}=b_{-m}*a_{m}=1$. By combining this with~(\ref{eq:CR.comptatibility-products-symbols}) 
we then see that, in any  local coordinates equipped 
with the $H$-frame associated to a frame $Z_{1},\ldots,Z_{n}$ of $T_{1,0}$, the symbol $q_{-m}(x,\xi):=b_{-m}(h(x),X_{0}(x),Z(x),\mu(x),\xi)$ is the 
inverse of $p_{\theta,m}$ with respect to the product~(\ref{eq:CR.product-symbols-URn}). 

Next, without any loss of generality we may assume that $Q_{\theta}$ is properly supported. Let $p(x,\xi)\sim \sum 
p_{\theta,m-j}(x,\xi)$ and $q(x,\xi) \sim \sum q_{-m-j}(x,\xi)$ be the respective symbols of $P_{\theta}$ and $Q_{\theta}$ with respect to local coordinates equipped 
with the $H$-frame associated to a frame $Z_{1},\ldots,Z_{n}$ of $T_{1,0}$. As $P_{\theta}Q_{\theta}=1\bmod \Psi^{-\infty}(M)$, from~(\ref{eq:CR.symbol-product}) we get
\begin{gather}
    p_{\theta,m}q_{-m}=1,\\ p_{\theta,m}*q_{-m-j}+ {\sum_{l<j}}^{(j)} h_{\alpha\beta\gamma\delta}
    (D_{\xi}^{\delta}p_{\theta,m-k})*(\xi^{\gamma}\partial_{x}^{\alpha}D_{\xi}^{\beta}q_{-m-l})=0 \quad j\geq 1. 
%      \label
\end{gather}
Therefore, we obtain 
\begin{gather}
    q_{-m}(x,\xi)=b_{-m}(h(x),X_{0}(x),Z(x),\mu(x),\xi), %\label  
    \\
    q_{-m-j}=-q_{-m} *\left[ {\sum_{l<j}}^{(j)} h_{\alpha\beta\gamma\delta}
    (D_{\xi}^{\delta}p_{\theta,m-k})*(\xi^{\gamma}\partial_{x}^{\alpha}D_{\xi}^{\beta}q_{-m-l})\right]
%     \label
\end{gather}

Now, as $P_{\theta}$ is a pseudohermitian invariant \psivdo\ for $j=1,2,\ldots$ there exists a finite family $(a_{j\fp})_{\fp \in \cP}\subset 
S_{m-j}(\Omega\times \Rtn)$ such that, in any  local coordinates equipped 
with the $H$-frame associated to a frame $Z_{1},\ldots,Z_{n}$ of $T_{1,0}$, we have 
\begin{equation}
    p_{\theta,m-j}(x,\xi)=\sum_{\fp \in \cP} p(X_{0},Z,\overline{Z})(x)a_{j\fp}(h(x),X_{0}(x),Z(x),\mu(x),\xi).
%     \label{¥}
\end{equation}
Then by using similar arguments as that of the proof of Proposition~\ref{prop:CR.pseudohermitian-invariant-PsiHDOs1} we can show by induction
that, for $j=0,1,\ldots$ there exists a finite family 
$(\tilde{c}_{j\fs})_{\fs \in \cP}$ contained in $S_{-m-j}(\Omega\times \Rtn)$ such that 
\begin{equation}
    q_{-m-j}=\sum_{\fs \in \cP} s(X_{0},Z,\overline{Z})(x)\tilde{c}_{j\fs}(h(x),X_{0}(x),Z(x),\mu(x),\xi),
%     \label{¥}
\end{equation}
where, with the notation of~(\ref{eq:CR.invariant-symbol-product}), the families $(\tilde{c}_{j\fs})_{\fs \in \cP}$ are given by the recursive formulas, 
\begin{gather}
    \tilde{c}_{01}=b_{-m}, \qquad \tilde{c}_{0\fs}=0 \quad\text{for $\fs\neq 1$},\\
        \tilde{c}_{j\fs}
    = - b_{-m}*\left[\sum_{\substack{\fp, \fq, \tilde{\fq}, \fr \in \cP\\ \fp\tilde{\fq}\fr=\fs}} {\sum_{l<j}}^{(j)} h_{\alpha \beta \gamma \delta \fr} 
    (D_{\xi}^{\delta}a_{k\fp})*(\xi^{\gamma}D_{\xi}^{\beta} C_{\alpha \tilde{\fq}}(\fq, \tilde{c}_{l\fq}))\right] \quad j\geq 0.
%     \label{¥}
\end{gather}
Since the families $(\tilde{c}_{j\fs})_{\fs \in \cP}$ don't depend on the local coordinates, this shows that $Q_{\theta}$ is a pseudohermitian 
invariant \psivdo. 

Finally, let $t>0$.  As $P_{t\theta}=t^{-w}P_{\theta}$ modulo $\Psi^{-\infty}(M)$ 
  we see that $t^{w}Q_{\theta}$ is a parametrix for $P_{t\theta}$, and 
  so we have $Q_{t\theta}=t^{w}Q_{\theta}$ modulo $\Psi^{-\infty}(M)$. This completes the proof that $Q_{\theta}$ is a pseudohermitian 
invariant \psivdo\ of weight $w$.
 \end{proof}

We are now ready to prove the main result of this section.
\begin{proposition}\label{prop:CR.log-sing-pseudohermitian-invariants}
Let $P_{\theta}$ be a pseudohermitian invariant \psivdo\ of order $m$ and weight~$w$\smallskip
    
1) The logarithmic singularity $c_{P_{\theta}}(x)$ takes the form
\begin{equation}
    c_{P_{\theta}}(x)=\cI_{P_{\theta}}(x)|d\theta^{n}\wedge \theta|,
%     \label
\end{equation}
where $\cI_{\theta}(x)$ is a local pseudohermitian invariant of weight $n+1+w$.\smallskip

2) Assume that $P_{\theta}$ is admissible and its principal symbol is invertible in the Heisenberg calculus sense.  Then the Green kernel
    logarithmic singularity of $P_{\theta}$ takes the form
\begin{equation}
    \gamma_{P_{\theta}}(x)=\cJ_{P_{\theta}}(x)|d\theta^{n}\wedge \theta|,
%     \label
\end{equation}
where $\cJ_{P_{\theta}}(x)$ is a local pseudohermitian invariant of weight $n+1-w$.
\end{proposition}
\begin{proof}
 Set $c_{P_{\theta}}(x)=\cI_{P_{\theta}}(x)|d\theta^{n}\wedge \theta|$, so that $\cI_{P_{\theta}}(x)$ is a smooth function on $M$. 
 For any $t>0$ we have 
 $c_{P_{t\theta}}(x)=c_{t^{-w}P_{\theta}}(x)=t^{-w}c_{P_{\theta}}(x)$ and $d(t \theta)^{n}\wedge (t 
 \theta)=t^{n+1}d\theta^{n}\wedge \theta$, so we see that
 \begin{equation}
     \cI_{P_{t \theta}}(x)=t^{-(w+n+1)}\cI_{P_{\theta}}(x) \qquad \forall t>0.
      \label{eq:CR.scaling-cPtheta}
 \end{equation}
 
 Next, by~(\ref{eq:Heisenberg.formula-cP}) in local coordinates equipped with the $H$-frame $X_{0},\ldots,X_{2n}$ associated to a local frame 
 $Z_{1},\ldots,Z_{n}$ of 
 $T_{1,0}$ we have 
 \begin{equation}
     c_{P_{\theta}}(x)=|\psi_{x}'|(2\pi)^{-(2n+1)}\left(\int_{\|\xi\|=1}p_{\theta,-(2n+2)}(x,\xi)\iota_{E}d\xi\right) dx,
%      \label
 \end{equation}
 where $p_{\theta,-(2n+2)}$ is the symbol of degree $-(2n+2)$ of $P_{\theta}$ in these local coordinates. 
 
 Furthermore, since $P_{\theta}$ is a pseudohermitian invariant \psivdo\ there exists a finite family $(a_{\fp})_{\fp \in \cP}\subset 
 S_{-(2n+2)}(\Omega\times \Rtn)$ such that
 \begin{equation}
     p_{\theta,-(2n+2)}(x,\xi)=\sum_{\fp \in \cP} \fp(X_{0},Z,\bar{Z})(x) a_{\fp}(h(x),X_{0}(x),Z(x), \mu(x), \xi).
%      \label{}
 \end{equation}
 Therefore,  we see that 
 \begin{equation}
   c_{P_{\theta}}(x)= \left( \sum_{\fp \in \cP} 
     \fp(X_{0},Z,\bar{Z})(x)A_{\fp}(h(x),X_{0}(x),Z(x),\mu(x))\right)|\psi_{x}'|dx,
      \label{eq:CR.form-cP-Aalphabeta1}
 \end{equation}
 where $A_{\fp}$ is the function in $C^{\infty}(\Omega)$ defined by 
 \begin{equation}
     A_{\fp}(h,X_{0},Z, \mu)=(2\pi)^{-(2n+1)}\int_{\|\xi\|=1}a_{\fp}(h,X_{0},Z,\mu,\xi)\iota_{E}d\xi.
      \label{eq:CR.form-cP-Aalphabeta2}
 \end{equation}
 
 Let $\theta^{1},\ldots,\theta^{n}$ be the coframe of $\Lambda^{1,0}$ dual to $Z_{1},\ldots,Z_{n}$, and let 
 $\eta^{0},\ldots,\eta^{2n}$ be the coframe of $T^{*}M$ dual to $X_{0},\ldots,X_{2n}$. Notice that $\eta^{0}=\theta$ and 
 $\theta^{j}=\frac{1}{2}(\eta^{j}+i\eta^{n+j})$. Moreover, we have $d\theta=ih_{jk}\theta^{j}\wedge \theta^{\bar{k}}$. Thus,  
 \begin{multline}
     d\theta^{n}\wedge \theta=i^{n}n!\det (h_{j\bar{k}})\theta^{1}\wedge \theta^{\bar{1}}\wedge \ldots \wedge \theta^{1}\wedge \theta^{\bar{1}}\wedge 
     \theta\\ 
     =n! \det (h_{j\bar{k}})\eta^{1}\wedge \eta^{n+1}\wedge\ldots \wedge \eta^{n}\wedge \eta^{2n}\wedge \eta^{0}\\
     =(-1)^{n}n! \det (h_{j\bar{k}})\eta^{0}\wedge\eta^{1}\wedge \ldots \wedge \eta^{2n}.
      \label{eq:CR.dtethantheta-omega}
 \end{multline}

 On the other hand, by its very definition $\psi_{a}$ is the unique affine change of variable such that $\psi_{a}(a)=0$ and 
 $(\psi_{a*}X_{j})(0)=\partial_{x^{j}}$. Therefore, if we set $X_{j}=X_{j}^{~k}\partial_{k}$ and $\eta^{j}=\eta^{j}_{~k}dx^{k}$, then we can check
 that $\psi_{a}(x)^{j}=\eta^{j}_{~k}(x^{k}-a^{k})$. Incidentally, we see that $ |\psi_{x}'| dx =|\det (\eta^{j}_{~k}) 
 dx^{0}\wedge \ldots \wedge dx^{2n}|=|\eta^{0}\wedge \ldots \wedge \eta^{2n}|$. Combining this with~(\ref{eq:CR.dtethantheta-omega}) then shows that 
 \begin{equation}
     |\psi_{x}'| dx = \frac{(-1)^{n}}{n!\det(h_{j\bar{k}})}|d\theta^{n}\wedge \theta|.
      \label{eq:CR.psidx-dtethantheta}
 \end{equation}
 
 Now, it follows from~(\ref{eq:CR.form-cP-Aalphabeta1}), (\ref{eq:CR.form-cP-Aalphabeta2}) and~(\ref{eq:CR.psidx-dtethantheta}) that, 
 in any local coordinates equipped with a frame $Z_{1},\ldots,Z_{n}$ of $T_{1,0}$, the function $\cI_{P_{\theta}}(x)$ is equal to
 \begin{equation}
   \sum_{\fp \in \cP} \frac1{n!} 
    \fp(X_{0},Z,\bar{Z})(x) \det{}^{-1}(h_{j\bar{k}}(x))A_{\fp}(h(x),X_{0}(x),Z(x),\mu(x)).
%      \label{¥}
 \end{equation}
 Together with~(\ref{eq:CR.scaling-cPtheta}) this shows that $\cI_{P_{\theta}}(x)$ is a local pseudohermitian invariant of weight $n+1+w$. 
 
 Finally, suppose that $P_{\theta}$ is admissible and its principal symbol is invertible in the Heisenberg calculus sense. 
    For each pseudohermitian manifold $(M^{2n+1},\theta)$ let $Q_{\theta}$ be a parametrix for $P_{\theta}$ in $\pvdo^{-m}(M)$. By definition the Green kernel 
    logarithmic singularity $\gamma_{P_{\theta}}(x)$ is equal to $c_{Q_{\theta}}(x)$, and we know from 
    Proposition~\ref{prop:CR.pseudohermitian-invariant-PsiHDOs2} 
    that $Q_{\theta}$ is a pseudohermitian invariant \psivdo\ of order $-m$ and weight $-w$. Therefore, it follows from the first part 
     that $\gamma_{P_{\theta}}(x)=\cJ_{P_{\theta}}(x)|d\theta^{n}\wedge \theta|$, 
    where $\cJ_{P_{\theta}}(x)$ is a local pseudohermitian invariant of weight $n+1-w$. The proof is thus achieved.
 \end{proof}

\section{Logarithmic singularities of CR invariants \psivdos}
\label{sec:CR-GJMS}
In this section we shall make use of the program of Fefferman in CR geometry to give a geometric description of the logarithmic singularities of CR 
invariant \psivdos.
% , including the Green kernel logarithmic singularities of the CR GJMS operators. 

\subsection{Local CR invariants and Fefferman's program}

The local CR invariants can be defined as follows. 

\begin{definition}
  A local scalar CR invariant of weight $w$  is a local scalar pseudohermitian invariant $\mathcal{I}_{\theta}(x)$ such that
 \begin{equation}
         \mathcal{I}_{e^{f}\theta}(x)=e^{-wf(x)}\mathcal{I}_{\theta}(x) \qquad \forall f \in C^{\infty}(M,\R).
     \label{eq:CR.local-CR-invariant}
 \end{equation}    
\end{definition}

When $M$ is a real hypersurface the  above definition of a local CR invariant agrees with the definition in~\cite{Fe:PITCA} in terms of 
 Chern-Moser invariants (with our convention about weight a local CR invariant that has weight $w$ in the sense of~(\ref{eq:CR.local-CR-invariant}) 
 has weight $2w$ in~\cite{Fe:PITCA}). 
 
 The analogue of the Weyl curvature in CR geometry is the Chern-Moser tensor~(\cite{CM:RHSCM}, \cite{We:PHSRH}). Its components with respect to any 
 local frame $Z_{1},\ldots,Z_{n}$ of $T_{1,0}$ are 
 \begin{equation}
     S_{j\bar{k}l\bar{m}}=R_{j\bar{k}l\bar{m}}- 
     (P_{j\bar{k}}h_{l\bar{m}}+P_{l\bar{k}}h_{j\bar{m}}+P_{l\bar{m}}h_{j\bar{k}}+P_{j\bar{m}}h_{l\bar{k}}), 
%      \label
 \end{equation}
 where $P_{j\bar{k}}=\frac{1}{n+2}(\rho_{j\bar{k}}- \frac{\kappa}{2(n+1)}h_{j\bar{k}})$ is the CR Schouten tensor. 
The Chern-Moser tensor is CR invariant of weight $1$, so we get scalar local CR invariants by taking complete tensorial contractions. 
 For instance, as scalar invariant of weight 2 we have
 \begin{equation}
     |S|_{\theta}^{2}=S^{\bar{j}k\bar{l}m}S_{j\bar{k}l\bar{m}},
%      \label
 \end{equation}
 and as scalar invariants of weight 3 we get 
 \begin{equation}
   S_{i\bar{j}}^{\mbox{~}\mbox{~}\bar{k}l} S_{k\bar{l}}^{\mbox{~}\mbox{~}\bar{p}q} S_{p\bar{q}}^{\mbox{~}\mbox{~}\bar{i}j} \quad \text{and} \quad  
 S_{i\mbox{~}\mbox{~}\bar{l}}^{\mbox{~}j\bar{k}}S^{\bar{i}\mbox{~}\mbox{~}q}_{\mbox{~}\bar{j}p} S^{\bar{p}\mbox{~}\mbox{~}l}_{\mbox{~}\bar{q}k}.
      \label{eq:CR.weight3-invariants}
 \end{equation}

 More generally, the Weyl CR invariants are obtained as follows. Let $\cK$ be the canonical line bundle of $M$, i.e., the annihilator of 
 $T_{1,0}\wedge \Lambda^{n}T^{*}_{\C}M$ in $\Lambda^{n+1}T_{\C}^{*}M$. The Fefferman bundle is the total space of the circle bundle,
\begin{equation}
     \cF:=(\cK\setminus 
     0)/\R_{+}^{*}.
%     \label{}
\end{equation}
 It carries a natural $S^{1}$-invariant Lorentzian metric $g_{\theta}$ whose conformal class depends only the CR structure of $M$, 
 for we have $g_{e^{f}\theta}=e^{f}g_{\theta}$ for any $f\in C^{\infty}(M,\R)$ (see~\cite{Fe:MAEBKGPCD}, \cite{Le:FMPHI}). 
 Notice also that the Levi metric defines a Hermitian metric 
 $h_{\theta}^{*}$ on $\cK$, so we have a natural natural isomorphism of circle bundles
 $\iota_{\theta}:\cF\rightarrow \Sigma_{\theta}$, where $\Sigma_{\theta}\subset \cK$ denotes the unit sphere bundle of $\cK$. 
 
 \begin{lemma}[\cite{Fe:PITCA}]\label{lem:CR.Riemannian-CR-invariant}
   Any local scalar conformal invariant $\cI_{g}(x)$ of weight $w$ uniquely defines a local scalar CR invariant of weight $w$.
 \end{lemma}
 \begin{proof}
 As $g_{\theta}$ is $S^{1}$-invariant the function $\cI_{g_{\theta}}(x)$ is $S^{1}$-invariant as well. Thus, if $\zeta$ is a local section of $\cF$ then we have 
 \begin{equation}
     \cI_{g_{\theta}}(\zeta(x))=\cI_{g_{\theta}}(e^{i\omega}\zeta(x)) \qquad \forall \omega \in \R.
%      \label{}
 \end{equation}
 This means that the value of $ \cI_{g_{\theta}}(\zeta(x))$ at $x$ does not depend on the choice of the local section $\zeta$ near $x$. Therefore, we 
 define a smooth function $\cI_{\theta}(x)$ on $M$ by letting
 \begin{equation}
     \cI_{\theta}(x):=\cI_{g_{\theta}}(\zeta(x)) \qquad \forall x\in M,
%      \label{¥}
 \end{equation}
 where $\zeta$ is any given local section of $\cF$ defined near $x$. 
 
 The fact that $\cI_{\theta}(x)$ is a local pseudohermitian invariant can be seen as follows. Let $Z_{1},\ldots,Z_{n}$ be a local frame of $T_{1,0}$ 
 near a point $a \in M$ and let $\{\theta,\theta^{j},\theta^{\bar{j}}\}$ be the dual coframe of the frame $\{X_{0},Z_{j},Z_{\bar{j}}\}$. By 
 standard multilinear algebra $\zeta_{\theta}=\det h_{j\bar{k}}\theta\wedge \theta^{1}\wedge \ldots \wedge \theta^{n}$ 
 is a local section of $\Sigma_{\theta}$. Therefore, it defines a local fiber coordinate $\gamma\in \cF$ such that $\iota_{\theta} = e^{i\gamma}\zeta$. 
 Then by~\cite[Thm.~5.1]{Le:FMPHI} the Fefferman metric is given by 
 \begin{equation}
     g_{\theta}=h_{j\bar{k}}\theta^{j}\theta^{\bar{k}}+2\theta\sigma, \quad 
     \sigma =\frac{1}{n+2}(d\gamma+i\omega_{j}^{~j}-\frac{i}{2}h^{j\bar{k}}dh_{j\bar{k}}-\frac{1}{2(n+1)}\kappa_{\theta}\theta).
%      \label{¥}
 \end{equation}
 Therefore, if $x_{0},x_{1},\ldots,x_{2n}$ are local coordinates for $M$ near $a$, then one can check that the components in the local coordinates 
 $x_{0},x_{1},\ldots,x_{2n},\gamma$  of the Fefferman metric 
 $g_{\theta}$ and of its inverse are universal expressions of the form~(\ref{eq:CR.pseudohermitian-invariants}). It then follows that 
 $\cI_{g_{\theta}}(\iota^{*}_{\theta}\zeta_{\theta}(x))$ is a universal expression of the form~(\ref{eq:CR.pseudohermitian-invariants}) as well, 
 so $\cI_{\theta}(x)$ is a local pseudohermitian invariant. 
 
  Finally, let $f \in C^{\infty}(M,\R)$. As $\cI_{g}(x)$ is a conformal invariant of weight $w$, we have 
 \begin{equation}
     \cI_{g_{e^{f}\theta}}(\zeta(x))=\cI_{e^{f}g_{\theta}}(\zeta(x))=e^{-wf(x)}\cI_{g_{\theta}}(\zeta(x)).
%      \label{¥}
 \end{equation}
 Hence $\cI_{e^{f}\theta}(x)=e^{-wf(x)}\cI_{\theta}(\zeta(x))$. This completes the proof that $\cI_{\theta}(x)$ is a local CR invariant of weight~$w$.
\end{proof}
 
 Now, the \emph{Weyl CR invariant} are the local CR invariants that are obtained from the Weyl 
 conformal invariants by the process described in the proof of Lemma~\ref{lem:CR.Riemannian-CR-invariant}. Notice that for the Fefferman bundle the ambient 
  metric was constructed by Fefferman~\cite{Fe:PITCA} as a K\"ahler-Lorentz metric. Therefore, the Weyl CR invariants are the local CR invariants that arise 
  from complete tensorial contractions of covariant derivatives of the curvature tensor of Fefferman's ambient K\"ahler-Lorentz metric. 
  
  Bearing this in mind the CR analogue of Proposition~\ref{prop:GJMS.BEG} is: 
  
  \begin{proposition}[{\cite[Thm.~2]{Fe:PITCA}}, {\cite[Thm.~10.1]{BEG:ITCCRG}}]\label{prop:CR.Fe-BEG} 
      Every local CR invariant of weight~$\leq n+1$ is a linear combination of local Weyl CR invariants. 
  \end{proposition}
  
In particular, we recover the fact that there is no local CR invariant of weight 1. Furthermore, we see that every local CR invariant of weight 2 is a 
constant multiple of $|S|_{\theta}$. Similarly, the local CR invariants of weight 3 are linear combinations of the invariants~(\ref{eq:CR.weight3-invariants}) 
and of the invariant $\Phi_{\theta}$ that arises from the Fefferman-Graham invariant $\Phi_{g_{\theta}}$ of the Fefferman Lorentzian space $\cF$. 

\subsection{Logarithmic singularities of CR invariant $\mathbf{\Psi_{H}}$DOs} The CR invariant \psivdos\ are defined as follows.
\begin{definition}
    A CR invariant \psivdo\ of order $m$ and biweight $(w,w')$ is a pseudohermitian invariant \psivdo\ $P_{\theta}$ such that
    \begin{equation}
        P_{e^{f}\theta}=e^{w'f}P_{\theta}e^{-wf} \qquad \forall f \in C^{\infty}(M,\R).
%         \label
    \end{equation}
\end{definition}

We actually have plenty of CR invariant operators thanks to:

\begin{proposition}[\cite{JL:YPCRM}, \cite{GG:CRIPSL}]
    Any conformally invariant Riemannian differential operator $L_{g}$ of weight $w$ uniquely defines a CR invariant differential operator 
    $L_{\theta}$ of same weight. 
\end{proposition}
\begin{proof}
  Since the Fefferman metric is $S^{1}$-invariant, the circle $S^{1}$ acts by isometries on $\cF$. Therefore, the operator
  $L_{g_{\theta}}$ is $S^{1}$-invariant, i.e., for any $\omega \in \R$ we have
\begin{equation}
    L_{g_{\theta}}(v\circ e^{i\omega})=(L_{g_{\theta}}v)\circ e^{i\omega}\qquad \forall v\in 
      C^{\infty}(\cF).
%     \label{}
\end{equation}

Let $\pi:\cF\rightarrow M$ be the canonical projection of $\cF$ and let $u \in C^{\infty}(M)$. Then 
  $\pi^{*}u$ is a $S^{1}$-invariant function on $\cF$, so for any $x \in M$ and any $\zeta\in \pi^{-1}(x)$ we have
  \begin{equation}
      L_{g_{\theta}}(\pi^{*}u)(\zeta)=L_{g_{\theta}}(\pi^{*}u)(e^{i\omega}\zeta)\qquad \forall \omega \in \R.
%       \label{¥}
  \end{equation}
  This means that $L_{g_{\theta}}(\pi^{*}u)(\zeta)$ does not depend on the choice of $\zeta$. Thus, we define a function $L_{\theta}(u)$ on $M$ by letting
\begin{equation}
      L_{\theta}(u)(x):=L_{g_{\theta}}(\pi^{*}u)(\zeta), \qquad \zeta \in \pi^{-1}(x).
%       \label{¥}
\end{equation}

Let us now consider local coordinates $x_{0},\ldots,x_{2n}$ for $M$ equipped with the $H$-frame $X_{0},\ldots,X_{2n}$ associated to a 
frame $Z_{1},..,Z_{2n}$ of $T_{1,0}$. Let 
$\theta^{1},..,\theta^{n}$ be the associated coframe of $\Lambda^{1,0}$, so that $\zeta:=\det^{\frac{1}{2}}(h_{j\bar{k}})\theta\wedge \theta^{1}\ldots 
\theta^{n}$ is a local section of $\Sigma_{\theta}$. Let $\gamma$ be the corresponding local fiber coordinate of $\cF$ in such way that 
$\iota_{\theta}=e^{i\gamma}\zeta$. Since $L_{g}$ is a Riemannian invariant differential operator there exist finitely many universal functions 
$a_{\alpha\beta\delta k}(g)$ in $C^{\infty}(M_{2n+2}(\R^{n})_{+})$ such that, in the local coordinates $x_{0},\ldots,x_{2n},\gamma$ of $\cF$, we have 
\begin{equation}
    L_{g_{\theta}}=\sum a_{\alpha\beta\delta k}(g_{\theta}(x))(\partial^{\alpha}g_{\theta}(x))^{\beta}\partial_{x}^{\delta}\partial_{\gamma}^{k}.
%     \label
\end{equation}
Notice that $S^{1}$-invariance corresponds to translation-invariance with respect to the variable $\gamma$. This is reflected in the property that the 
components of $g_{\theta}$ don't depend on $\gamma$. Furthermore, we see that for any smooth function $u(x)$ of the local coordinates $x_{0},\ldots,x_{2n}$ we have
\begin{equation}
    L_{\theta}(u)(x)=\sum a_{\alpha\beta\delta 0}(g_{\theta}(x))(\partial^{\alpha}g_{\theta}(x))^{\beta}\partial_{x}^{\delta}u(x)
     \label{eq:CR.Ltetha.invariant-form}
\end{equation}
In particular, this shows that $L_{\theta}$ is a differential operator. 

As explained in the proof of Lemma~\ref{lem:CR.Riemannian-CR-invariant} 
the components of $g_{\theta}(x)$ in the local coordinates $x_{0},\ldots,x_{2n},\gamma$, as well as their 
derivatives, are universal expressions of the form~(\ref{eq:CR.pseudohermitian-invariants}). 
Therefore, from~(\ref{eq:CR.Ltetha.invariant-form}) we deduce that there exists a finite family $(b_{\fp k\delta\bar{\rho}}) \subset 
C^{\infty}(\Omega)$ such that, in any local coordinates equipped with the $H$-frame associated to a frame
$Z_{1},..,Z_{n}$ of $T_{1,0}$, the differential operator $L_{\theta}$ is equal to
\begin{equation}
  \sum_{k,\delta, \bar{\rho}}  \sum_{\fp \in \cP} \fp(X_{0},Z,\overline{Z})(x)b_{\fp k\delta\bar{\rho}} 
    (h(x),X_{0}(x),Z(x))(-iX_{0})^{k}(-iZ)^{\delta}(-i\bar{Z})^{\bar{\rho}}. 
%     \label{¥}
\end{equation}
It then follows that $L_{\theta}$ is a pseudohermitian invariant differential operator. 

Finally, let $f\in C^{\infty}(M,\R)$. Since $L_{g}$ is conformally invariant of biweight $(w,w')$ we have 
$L_{g_{e^{f}\theta}}=L_{e^{f}g_{\theta}}=e^{w'f}L_{g_{\theta}}e^{-wf}$. Hence 
$L_{e^{f}\theta}=e^{w'f}L_{\theta}e^{-wf}$. This completes the proof that $L_{\theta}$ is a CR invariant differential operator of biweight $(w,w')$. 
\end{proof}

When $L_{g}$ is the Yamabe operator the corresponding CR invariant operator is the CR Yamabe operator introduced by Jerison-Lee~\cite{JL:YPCRM} in 
their solution of the Yamabe problem on CR manifold. It can be defined as follows. 

First, the analogue of the Laplacian is provided by the horizontal sublaplacian $\Delta_{b}: C^{\infty}(M)\rightarrow C^{\infty}(M)$ defined by the 
formula,
 \begin{equation}
     \Delta_{b}=d_{b}^{*}d_{b}, \qquad d_{b}=\pi\circ d,
% %     \label
 \end{equation}
where $\pi\in C^{\infty}(M,\End T^{*}M)$ is the orthogonal projection onto $H^{*}$. In fact, if $Z_{1},\ldots,Z_{n}$ is a local frame of $T_{1,0}$ 
then by~\cite[Prop.~4.10]{Le:FMPHI} we have 
\begin{equation}
    \Delta_{b}=\nabla_{Z_{j}}^{*}\nabla_{Z_{j}} + \nabla_{Z_{\bar{j}}}^{*}\nabla_{Z_{\bar{j}}}. 
     \label{eq:CR.local-formula-Deltab}
\end{equation}
It follows from this formula that $\Delta_{b}$ is a sublaplacian in the sense of~\cite{BG:CHM} and its principal symbol in the Heisenberg calculus 
sense is invertible (see~\cite{BG:CHM}, \cite{Po:MAMS1}).

The CR Yamabe operator is given by the formula, 
  \begin{equation}
  \boxdot_{\theta}=\Delta_{b}+ \frac{n}{n+2} \kappa_{\theta},
  \label{eq:CR.comformal-sublaplacian}
  \end{equation}
where $\kappa_{\theta}$ is the Tanaka-Webster scalar curvature. This is a CR invariant differential operator of biweight $(\frac{-n}{2},-\frac{n+2}{2})$. 
Moreover, as $\boxdot_{\theta}$ and $\Delta_{b}$ 
have same principal symbol, we see that the principal symbol of $\boxdot_{\theta}$ is invertible in the Heisenberg calculus sense. 

Next, Gover-Graham~\cite{GG:CRIPSL} proved that for $k=1,\ldots,n+1$ the GJMS operator $\Box_{g}^{(k)}$ on the Fefferman bundle gives rise to a 
selfadjoint differential operator, 
\begin{equation}
    \boxdot_{\theta}^{(k)}: C^{\infty}(M)\longrightarrow C^{\infty}(M).
%     \label{¥}
\end{equation}
This is a CR invariant operator of biweight $(\frac{k-(n+1)}{2},-\frac{k+n+1}{2})$ and it has same principal symbol as 
  \begin{equation}
      (\Delta_{b}+i(k-1)X_{0})(\Delta_{b}+i(k-3)X_{0})\cdots (\Delta_{b}-i(k-1)X_{0}).
       \label{eq:CR.principal-part-CR-GJMS}
  \end{equation}
 In particular, unless for the critical value $k=n+1$, the principal symbol of $\boxdot_{\theta}^{(k)}$ is invertible in the Heisenberg calculus sense 
 (see~\cite[Prop.~3.5.7]{Po:MAMS1}). The operator $ \boxdot_{\theta}^{(k)}$ is called the CR GJMS operator of order $k$. For $k=1$ we recover the CR 
 Yamabe operator. Notice that by  making use of a CR tractor calculus we also can define CR GJMS operators of order $k\geq n+2$ (see~\cite{GG:CRIPSL}). 
 
 More generally, the conformally invariant Riemannian differential operators of Alexakis~\cite{Al:OCIDO} and Juhl~\cite{Ju:FCDOQCH} 
 give rise to CR invariant 
 differential operators. If we call Weyl CR invariant differential operators the operators induced by the Weyl operators of~\cite{Al:OCIDO}, then a natural 
 question would be to determine to which extent these operators allows us to exhaust all the CR invariant differential operators. 
  
  We are now redy to prove the main result of this section. 

\begin{theorem}\label{thm:CR.main-result}
    Let $P_{\theta}$ be a CR invariant \psivdo\ of order $m$ and biweight $(w,w')$.\smallskip
    
    1) The logarithmic singularity $c_{P_{\theta}}(x)$ takes the form
\begin{equation}
    c_{P_{\theta}}(x)=\cI_{P_{\theta}}(x)|d\theta^{n}\wedge \theta|,
%     \label{¥}
\end{equation}
where $\cI_{\theta}(x)$ is a scalar local CR invariant of weight $n+1+w-w'$. If we further have $w\leq w'$, then  $\cI_{\theta}(x)$ is a
linear combination of Weyl CR invariants of weight $n+1+w-w'$.\smallskip

2) Assume that $P_{\theta}$ is admissible and its principal symbol is invertible in the Heisenberg calculus sense.  Then the Green kernel
    logarithmic singularity of $P_{\theta}$ takes the form
\begin{equation}
    \gamma_{P_{\theta}}(x)=\cJ_{P_{\theta}}(x)|d\theta^{n}\wedge \theta|,
%     \label{¥}
\end{equation}
where $\cJ_{P_{\theta}}(x)$ is a scalar local CR invariant of weight $n+1-w+w'$. If we further have $w\geq w'$, then 
$\cJ_{P_{\theta}}(x)$ is a linear combination of Weyl CR invariants of weight $n+1-w+w'$. 
\end{theorem}
 \begin{proof}
   Since $P_{\theta}$ is a pseudohermitian invariant \psivdo\ of weight $w-w'$, by Proposition~\ref{prop:CR.log-sing-pseudohermitian-invariants} 
   the logarithmic singularity $c_{P_{\theta}}(x)$ is of the form 
   $c_{P_{\theta}}(x)=\cI_{P_{\theta}}(x)$, where $\cI_{P_{\theta}}(x)$ is a local pseudohermitian invariant of weight $w-w'$. 
   
   Let $f\in C^{\infty}(M,\R)$. As $P_{\theta}$ is conformally invariant of biweight $(w,w')$, by Proposition~\ref{prop:Contact.main} we have
   $c_{P_{e^{f}\theta}}(x)= e^{-(w-w')f(x)}c_{P_{\theta}}(x)$. Since $d(e^{f}\theta)^{n}\wedge (e^{f}\theta)=e^{(n+1)f}d\theta^{n}\wedge \theta$ it 
   follows that $\cI_{e^{f}\theta}(x)=e^{-(n+1+w-w')f(x)}\cI_{\theta}(x)$. Thus $\cI_{\theta}$ is a local CR invariant of weight $n+1+w-w'$. 
   If we further have $w\leq w'$ then the weight of $\cI_{\theta}(x)$ is~$\leq n+1$, so we may apply Proposition~\ref{prop:CR.Fe-BEG}  to deduce that  
   $\cI_{\theta}(x)$ is a linear combination of Weyl CR invariants of weight $n+1+w-w'$. 
   
   Now, suppose that $P_{\theta}$ is admissible and its principal symbol is invertible in the Heisenberg calculus sense. Then by using 
   Proposition~\ref{prop:Contact.main} and Proposition~\ref{prop:CR.log-sing-pseudohermitian-invariants}, 
   and by arguing as above, we can show that $\gamma_{P_{\theta}}(x)=\cJ_{P_{\theta}}(x)|d\theta^{n}\wedge 
   \theta|$, where $\cJ_{P_{\theta}}(x)$ is a local CR invariant of weight $n+1-w+w'$. If we further have $w\geq w'$, then by 
   Proposition~\ref{prop:CR.Fe-BEG} the invariant $\cJ_{P_{\theta}}(x)$ is actually a linear combination of Weyl CR invariants of weight $n+1-w+w'$.
  \end{proof}
  
  Finally, we can make us of Theorem~\ref{thm:CR.main-result} to derive the following invariant expression of the Green kernel logarithmic singularities of the 
  CR GJMS operators. 
  
 \begin{theorem}\label{thm:CR.CRGJMS}
     For $k=1,\ldots,n$ we have
    \begin{equation}
        \gamma_{\Box_{\theta}^{(k)}}(x)=c_{\theta}^{k}(x)|d\theta^{n}\wedge \theta|,
%         \label{¥}
    \end{equation}
    where $c_{\theta}^{k}(x)$ is a linear combination of scalar Weyl CR invariants of weight $n+1-k$. In particular, we have 
    \begin{gather}
        c_{\theta}^{(n)}(x)=0, \qquad c_{\theta}^{(n-1)}(x)=\alpha_{n}|S|_{\theta}^{2},
        \label{eq:CR.cthetak-12}\\
        c_{\theta}^{(n-2)}(x)=
 \beta_{n}S_{i\bar{j}}^{\mbox{~}\mbox{~}\bar{k}l} S_{k\bar{l}}^{\mbox{~}\mbox{~}\bar{p}q} 
                        S_{p\bar{q}}^{\mbox{~}\mbox{~}\bar{i}j} 
        + \gamma_{n} 
        S_{i\mbox{~}\mbox{~}\bar{l}}^{\mbox{~}j\bar{k}}S^{\bar{i}\mbox{~}\mbox{~}q}_{\mbox{~}\bar{j}p}
                            S^{\bar{p}\mbox{~}\mbox{~}l}_{\mbox{~}\bar{q}k}
         +\delta_{n}\Phi_{\theta},
         \label{eq:CR.cthetak-3}
    \end{gather}
    where $S$ is the Chern-Moser curvature tensor, $\Phi_{\theta}$ is the CR Fefferman-Graham invariant, and 
    the constants $\alpha_{n}$, $\beta_{n}$, $\gamma_{n}$ and $\delta_{n}$ depend only on $n$. 
 \end{theorem}
 \begin{proof}
     We already now that the CR GJMS operator $\boxdot_{\theta}^{(k)}$ is a CR invariant differential operator of biweight 
     $(\frac{k-(n+1)}{2},-\frac{k+n+1}{2})$ and for $k=1,\ldots,n$ its principal symbol is invertible in the Heisenberg calculus sense. Therefore, in order to be 
     able to apply Theorem~\ref{thm:CR.main-result} it remains to show that $\boxdot_{\theta}^{(k)}$ is admissible. 
     By~(\ref{eq:CR.principal-part-CR-GJMS}) the principal symbol of $\boxdot_{\theta}^{(k)}$ agrees with that of $(\Delta_{b}+i(k-1)X_{0})\cdots 
     (\Delta_{b}-i(k-1)X_{0})$. Therefore, in view of Proposition~\ref{prop:CR.pseudohermitian-invariant-PsiHDOs1} in order to prove that 
     $\boxdot_{\theta}^{(k)}$ is admissible
     it is enough to show that so is any operator of the form $\Delta_{b}-i\mu X_{0}$, $\mu\in \C$.
     
     Consider local coordinates equipped with a $H$-frame  $X_{0},\ldots,X_{2n}$ associated to a 
     frame $Z_{1},\ldots,Z_{n}$ of $T_{1,0}$, so that  we have $Z_{j}=X_{j}-iX_{n+j}$. It follows from~(\ref{eq:CR.local-formula-Deltab}) 
     that $\Delta_{b}$ as same principal part as
     $-h^{j\bar{k}}Z_{\bar{k}}Z_{j}-h^{\bar{j}k}Z_{k}Z_{\bar{j}}$, so the principal symbol of $\Delta_{b}-i\mu X_{0}$ is equal to
     \begin{equation}
         -h^{j\bar{k}}(x)(\xi_{k}+i\xi_{n+k})(\xi_{j}-i\xi_{n+j})-h^{\bar{j}k}(x)(\xi_{k}-i\xi_{n+k})(\xi_{j}-i\xi_{n+j})+\xi_{0}.
%          \label{¥}
     \end{equation}
     This shows that $\Delta_{b}-i\mu X_{0}$ is admissible. 
     
     Now, we may apply Theorem~\ref{thm:CR.main-result} to deduce that for $k=1,\ldots,n$ the Green 
     kernel logarithmic singularity of $\boxdot_{\theta}^{(k)}$ is of the form $\gamma_{\Box_{\theta}^{(k)}}(x)=c_{\theta}^{(k)}(x)d\theta^{n}\wedge 
     \theta(x)$, where $c_{\theta}^{(k)}(x)$ is a linear combination of Weyl CR invariants of weight $n+1-k$. The 
     formulas~(\ref{eq:CR.cthetak-12})--(\ref{eq:CR.cthetak-3}) then follow from 
     the facts that there is no nonzero scalar Weyl CR invariant of weight 1, that the only scalar Weyl CR invariants of weight 2 is 
 $|S|_{\theta}^{2}$, and that the only scalar Weyl CR invariants of weight 3 are the invariants~(\ref{eq:CR.weight3-invariants}) 
 and the CR Fefferman-Graham invariant $\Phi_{\theta}$. 
 \end{proof}

\end{document}